%% file: main_arxiv.tex
\documentclass[12pt]{article}
\usepackage[colorlinks,citecolor=blue,urlcolor=blue,pagebackref]{hyperref}       
\hypersetup{
  colorlinks,
  linkcolor={blue},
  citecolor={blue},
  urlcolor={blue}
}
\renewcommand*{\backref}[1]{\ifx#1\relax \else Page #1 \fi}
\renewcommand*{\backrefalt}[4]{%
 \ifcase #1 \footnotesize{(Not cited.)}%
  \or        \footnotesize{(Cited on page~#2.)}%
  \else      \footnotesize{(Cited on pages~#2.)}%
  \fi
}
\usepackage{url}            
\usepackage{booktabs}       
\usepackage{amsfonts}       
\usepackage{nicefrac}       
\usepackage{xcolor}         
\usepackage{color}
\usepackage{bm,multirow,xspace}
\usepackage{enumerate,mathtools,enumitem}
\usepackage{graphicx}
\usepackage{epstopdf,amsmath}
\usepackage{bbm}

\usepackage{amssymb, amsthm}
\usepackage{mathrsfs}
\usepackage{makecell}

\allowdisplaybreaks

\usepackage{amsmath}    
\usepackage{amsmath,color}

\usepackage{algpseudocode}
\usepackage{graphicx}
\usepackage{epstopdf}

\allowdisplaybreaks

\newtheorem{definition}{Definition}

\DeclareMathOperator*{\argmin}{argmin}

\newcommand{\abs}[1]{\left\lvert #1 \right\rvert}
\DeclareMathOperator*{\iid}{\texttt{iid}}

\newcommand{\floor}[1]{\lfloor #1 \rfloor}

\newcommand{\pro}{{\mathbb P}}
\newcommand{\expec}[2]{\mathbb{E}_{#2}\left[ #1 \right] }
\newcommand\numberthis{\addtocounter{equation}{1}\tag{\theequation}}  

\def\fn[#1]#2{{f_{#1}\left(x_{#2}\right)}}

\newcommand\given[1][]{\:#1\vert\:}

\newtheorem{lemma}{Lemma}[section]
\newtheorem{theorem}{Theorem}[section]
\newtheorem{corollary}{Corollary}[section]
\newtheorem{proposition}{Proposition}[section]
\newtheorem{assumption}{Assumption}[section]
\newtheorem{remark}{Remark}

\def\E{{\mathbb{E}}}
\def\exp{{\rm exp}}

\def\cD{{\cal D}}

\def\cH{{\cal H}}

\def\cF{{\cal F}}

\def\dgp{DGP\xspace}

\newenvironment{talign*}
 {\csname align*\endcsname}
 {\endalign}

\newcommand{\murat}[1]{{\color{magenta}#1}}
\title{On Empirical Risk Minimization with Dependent \\ and Heavy-Tailed Data}
\author{Abhishek Roy\thanks{Department of Statistics, University of California, Davis.\texttt{abroy@ucdavis.edu}. Research of this author was supported in part by NSF TRIPODS grant  CCF-1934568} 
\and Krishnakumar Balasubramanian\thanks{Department of Statistics, University of California, Davis. \texttt{kbala@ucdavis.edu}. Research of this author was supported in part by UC Davis CeDAR (Center for Data Science and Artificial Intelligence Research) Innovative Data Science Seed Funding Program.}
\and Murat A. Erdogdu\thanks{Department of Computer Science and Department of Statistical Sciences at
   the University of Toronto, and Vector Institute. \texttt{erdogdu@cs.toronto.edu}.  Research of this author was supported in part by NSERC Grant [2019-06167], Connaught New Researcher Award, CIFAR AI Chairs program, and CIFAR AI Catalyst grant}
}

\begin{document}

\maketitle

\begin{abstract}
In this work, we establish risk bounds for the Empirical Risk Minimization (ERM) with both dependent and heavy-tailed data-generating processes. We do so by extending the seminal works~\cite{pmlr-v35-mendelson14, mendelson2018learning} on the analysis of ERM with heavy-tailed but independent and identically distributed observations, to the strictly stationary exponentially $\beta$-mixing case. Our analysis is based on  explicitly controlling the multiplier process arising from the interaction between the noise and the function evaluations on inputs. It allows for the interaction to be even polynomially heavy-tailed, which covers a significantly large class of heavy-tailed models beyond what is analyzed in the learning theory literature. We illustrate our results by deriving rates of convergence for the high-dimensional linear regression problem with dependent and heavy-tailed data.
\end{abstract}

\input{intro}
\input{mainresults}
\input{examples}
\input{proofsketch}
\section{Conclusion}
In this work, we analyzed the performance of empirical risk minimization with squared error and convex loss functions, when
the \dgp is both dependent (specifically, exponentially $\beta$-mixing) and heavy-tailed. We derived explicit rates using a combination of small-ball method and  concentration inequalities. We demonstrated the applicability of our results on a high-dimensional linear regression problem, and showed that our assumptions are easily verified for a certain classes of sub-Weibull and Pareto \dgp. Our results clearly show
the benefits of using Huber loss over the squared error loss for ERM with heavy-tailed data in our setting.
For future work, 
we plan to study median-of-means based techniques and examine establishing similar rates of convergence for dependent heavy-tailed data.
\bibliographystyle{amsalpha}
\bibliography{stein}
\appendix

\input{appendix}

\end{document}

%% file: intro.tex

\section{Introduction}\label{sec:intro}
Given a random vector $(X,Y) \in \mathbb{R}^d \times \mathbb{R}$, with joint distribution $(X,Y)\sim\pi$, and a class of closed, convex set of functions $\cF\subset L_2(\pi)$, the objective in statistical learning theory is to find the \emph{best} function in the set $\cF$ that maps the input $X$ to the target $Y$. The quality of this mapping is measured by a user-defined loss function $\ell:\mathbb{R}\to \mathbb{R}^+\cup \{0\}$. The most well-studied approach for the above task is that of \emph{risk minimization}, where the best function is defined as the one that minimizes the expected loss over the set $\cF$: 
\begin{align*}
	f^*=\argmin_{f\in\cF}P\ell_f\coloneqq\argmin_{f\in\cF}\expec{\ell\left(f(X)-Y\right)}{\pi}.
\end{align*}
The above problem requires the knowledge of the distribution $\pi$ which is typically unknown in practice. However, we are usually given observations $Z_i=(X_i,Y_i) $ for $i=1,\dots, N$, from the distribution $\pi$ which leads to the \emph{Empirical Risk Minimization} (ERM) procedure defined as
\begin{align*}
	\hat{f}=\argmin_{f\in\cF}P_N\ell_f\coloneqq\argmin_{f\in\cF}\frac{1}{N}\sum_{i=1}^N\ell\left(f(X_i)-Y_i\right).
\end{align*}
The convergence of the empirical risk minimizer $\hat{f}$ to the true risk minimizer $f^*$ is typically analyzed by considering the underlying empirical process, a topic which dates back to the seminal work of~\cite{vapnik1971uniform}; see also~\cite{van1996weak, geer2000empirical, bartlett2005local, koltchinskii2006local, koltchinskii2011oracle}. In a representative analysis in this setting, a majority of the works assume the observations $Z_i$ are generated independent and identically distributed ($\iid$) from $\pi$, and the analysis is based on uniform concentration. 
However, there are important limitations associated with this approach, particularly due to the (Talagrand's) contraction principle which naturally requires a Lipschitz loss function (see, for example,~\cite[Corollary 3.17]{ledoux2013probability} or~\cite[Theorem 2.3]{koltchinskii2011oracle}). 
As a result, in order to work with standard (unbounded) loss functions such as squared-error loss or Huber loss, it is generally assumed that the range of $f\in \cF$ is uniformly bounded and/or the noise $\xi \coloneqq Y-f(X)$ is also uniformly bounded $\pi$-almost surely.

Several attempts have been made in the literature to overcome the limitations of the standard ERM analysis. A significant progress was made by Mendelson~\cite{pmlr-v35-mendelson14, mendelson2018learning}, who proposed the so-called \emph{learning without concentration} framework for analyzing ERM procedures with unbounded noise or loss functions. The approach is based on a combination of small-ball type assumption on the input samples $X_i$, along with developing multiplier empirical process inequalities under weaker moment assumptions. We refer the interested reader, for example, to~\cite{mendelson2017multiplier,mendelson2017local,lecue2018regularization,liang2015learning,grunwald2020fast} for details. The aforementioned works, while relaxing the prior analysis of ERM to handle heavy-tailed data-generating process (\dgp), still require the more stringent $\iid$ assumption for their analysis. This restricts the practical applicability of the developed theoretical results significantly. Indeed, heavy-tailed and dependent data appear naturally in various practical learning scenarios~\cite{boente1989robust, jiang2001robust, dundar2007learning};
however, theoretical guarantees are still missing.  

\textbf{Our Contributions:} Aiming to fill the above gap, we analyze ERM with convex loss functions (that are locally strongly-convex around the origin) when the \dgp is both heavy-tailed and non-$\iid$. We do so by extending the small-ball technique of~\cite{pmlr-v35-mendelson14, mendelson2018learning} to the strictly stationary exponentially $\beta$-mixing data. In the $\iid$ case, the interaction between the noise and the inputs is handled by an analysis based on multiplier empirical process. However, developing similar techniques in the non-$\iid$ case is fundamentally restrictive due to the limitations of the analysis based on empirical process. We side-step this issue for the non-$\iid$ case by directly making assumptions on the interaction, which allows for it to be either exponentially or polynomially heavy-tailed. For the exponentially heavy-tailed interactions, we leverage the concentration inequalities developed by~\cite{merlevede2011bernstein}. For the polynomially heavy-tailed case, we develop new concentration inequalities extending the recent work~\cite{bakhshizadeh2020sharp} to $\beta$-mixing random variables. We illustrate our results in the context of ERM with sparse linear function class and stationary $\beta$-mixing \dgp under both squared and Huber loss. 

\textbf{Motivation:} A natural question arises in this context: \emph{Why study ERM with convex loss functions when the \dgp is heavy-tailed?} Firstly, convex loss functions cover a large class of robust loss function that are tailored to deal with the heavy-tailed behavior present in the noise and/or input data. Some examples include the Huber loss~\cite{huber1992robust}, conditional value-at-risk~\cite{rockafellar2002conditional,ruszczynski2006optimization, mhammedi2020pac, soma2020statistical} and the so-called spectral risk measures~\cite{acerbi2002spectral, holland2021spectral}. While there exist studies for nonconvex loss functions suited for heavy-tailed input data (for example,~\cite{loh2017statistical}), such analyses are mostly in a model-based setting and focus on estimation error. Secondly, while alternatives to ERM have also been proposed and analyzed in the literature for the $\iid$ case (with the most prominent one being the median-of-means framework and its variants~\cite{minsker2019excess, lugosi2019regularization, lecue2020robust,bartl2021monte}), it is not immediately clear how to extend such methods to the dependent \dgp that we consider in this paper. We view our work as taking the first step in developing risk bounds for statistical learning when the \dgp is both heavy-tailed and dependent.

\textbf{Related Works:} 
The seminal work~\cite{yu1994rates} extended the analysis based on empirical process to the stationary mixing process using a blocking technique. \cite{irle1997consistency} and~\cite{berti1997glivenko} studied consistency of non-parametric regression methods under mixing and exchangeability conditions on the \dgp, respectively.~\cite{nobel1999limits} established lower bounds to achieving consistency when learning from dependent data.~\cite{steinwart2009learning} studied consistency of ERM with $Z_i$ being an $\alpha$-mixing (not necessarily stationary) process, when $\mathcal{F}$ is a reproducing kernel Hilbert space. More recently,~\cite{hanneke2020learning} and~\cite{dawid2020learnability} studied learnability under a general stochastic process setup. The works~\cite{aldous1990markovian,barve1997complexity,pestov2010predictive,gamarnik2003extension} extend Valiant's Probably Approximately Correct (PAC) learning model to Markovian and related drifting \dgp, assuming bounded loss function and/or noise to obtain rates of convergence. 
Furthermore, \cite{zhang2012learning} and~\cite{hang2014fast} analyzed ERM for least-squares regression (with bounded noise) with clipped loss functions and an $\alpha$-mixing \dgp. Rademacher complexity results for predominantly stationary dependent processes were developed by~\cite{mcdonald2011rademacher} and~\cite{mohri2008rademacher}.~\cite{ralaivola2010chromatic} and~\cite{alquier2013prediction} developed PAC-Bayes bounds in the non-$\iid$ setting.~\cite{rakhlin2012empirical, rakhlin2014online} developed notions of sequential Rademacher complexity to characterize the complexity in online nonparametric learning in the worst-case. More recently~\cite{dagan2019learning, kandiros2021statistical} considered learning under weakly-dependent data for specific models. However, such works mainly rely on bounded loss functions in their analysis. Furthermore, \cite{meir2000nonparametric} and~\cite{alquier2012model} studied model selection for time series forecasting in a possibly unbounded setup. However their work does not consider conditional prediction and is limited to light-tailed cases.

Apart from the aforementioned works, the recent works~\cite{kuznetsov2017generalization,han2019convergence,wong2020} are closely related to our setup as they consider rates of convergence of ERM under heavy-tailed and dependent \dgp. In \cite[Section 8]{kuznetsov2017generalization}, generalization bounds are developed when $Z_i$ is an asymptotically stationary $\beta$-mixing sequence. However, their conditions on the function class $\mathcal{F}$ are rather opaque and it is not clear if their method actually handles the heavy-tailed \dgp that we focus on. \cite{han2019convergence} considered a setup based on a statistical model: for $i=1,\ldots, N$, $Y_i=f^*(X_i) + \epsilon_i$, with the following conditions: (i) $\epsilon_i$ being independent of $X_i$, (ii) $X_i$ being independent of each other, and (iii) $\epsilon_i$ being arbitrarily dependent. For this setting,  they assumed that the noise has a bounded $p$-th moment (with $p\geq 1$) and $\mathcal{F}$ satisfies the standard entropy condition (see, for example~\cite[Example 4]{koltchinskii2006local}) with exponent $\alpha \in (0,2)$ and obtained convergence rates of the order $O(N^{-\frac{1}{2+\alpha}}+N^{-\frac{1}{2} + \frac{1}{2p}})$. Finally,~\cite{wong2020} provides an analysis of $L_1$-regularized ERM with quadratic loss and linear function class where the heavy-tailed behavior 
is induced by
a sub-Weibull assumption, and data dependency is characterized by a stationary $\beta$-mixing condition. However, their analysis specializes to sparse linear function classes and their focus is on parameter estimation error and in-sample prediction accuracy. 


%% file: mainresults.tex
\vspace{-.05in}
\section{Assumptions and Preliminaries}
\vspace{-.05in}
We assume that there exists a function $f^* \in \cF$ that minimizes the population risk $\E[\ell(f(X) - Y)]$. In what follows, we provide the conditions that we require on the \dgp, specifically, the exponentially $\beta$-mixing condition, for characterizing the dependency among data points.
%
\begin{definition}[\cite{yu1994rates}]
	Suppose that $\{Z_i\}_{i=-\infty}^\infty$ is a strictly stationary sequence of random variables. For any $i,j\in\mathbb{Z}\cup\{-\infty,\infty\}$, let $\sigma_i^j$ denote the $\sigma$-algebra generated by $\{Z_b\}_{b=i}^j$. Then for any positive integer $b$, the $\beta$-mixing coefficient of the stochastic process $\{Z_i\}_{i=-\infty}^\infty$ is defined as $$\beta(b)=\sup_n\underset{B\in\sigma^n_{-\infty}}{\mathbb{E}}\big[\sup_{A\in\sigma_{n+b}^\infty}\abs{\mathbb{P}(A\given B)-\mathbb{P}(A)}\big].$$
The sequence $\{Z_i\}_{i=-\infty}^\infty$ is said to be $\beta$-mixing if $\beta(b)\to 0$ as $b\to \infty$. Furthermore, it is said to be exponentially $\beta$-mixing if there exist $\beta_0,\beta_1,r>0$ such that  $\beta(b) \leq \beta_0 \exp(-\beta_1b^r)$ for all $b$.
\end{definition}
The $\beta$-mixing condition is frequently used when studying non-$\iid$ \dgp, and imposes a dependence structure between data samples that weakens over time. The coefficient $\beta(b)$ is a measure of the dependence between events that occur within $b$ units in time. Indeed, $\beta$-mixing
is often used in the analysis of non-$\iid$ data in statistics and machine learning,~\cite{vidyasagar2013learning}. Before stating our assumptions formally, we present the following decomposition of empirical risk for a convex loss using Taylor's expansion:
\vspace{-.05in}
\begin{align}\label{eq:decomp}
	P_N\ell_f\geq\frac{1}{16N}\sum_{i=1}^N \ell''(\widetilde{\xi}_i)(f-f^*)^2(X_i)+\frac{1}{N}\sum_{i=1}^N \ell'(\xi_i)(f-f^*)(X_i),
	\vspace{-.05in}
\end{align}
where $\widetilde{\xi}_i$ is a suitably chosen midpoint between $f(X_i)-Y_i$ and $f^*(X_i)-Y_i\coloneqq \xi_i$. For quadratic loss functions $\ell(t)=t^2$, $\ell'(\xi_i)=2\xi_i$ and $\ell''(\widetilde{\xi}_i)=2$, $\forall i$. At a high level, establishing risk bounds boils down to proving a positive lower bound on the second term on the Right Hand Side (RHS) of \eqref{eq:decomp}, and a concentration result for the first term on the RHS with high probability.  We now introduce the precise assumptions we make on the \dgp and the function class to formalize the above strategy. For the sake of clearer exposition, we first introduce our assumptions in the context of quadratic loss and then indicate the changes required to handle more general locally strongly-convex loss functions. 
\begin{assumption}[Squared loss]\label{as:datagenprocess}
The \dgp $\{Z_i\}_{i=-\infty}^\infty$ and the function class $\mathcal{F}$ satisfy the following: 
\begin{enumerate}[leftmargin=.25in]
\item [(a)] \textbf{$\beta$-mixing data.} The process $\{Z_i\}_{i=-\infty}^\infty$ is a strictly stationary exponentially $\beta$-mixing sequence, i.e., $\beta(k)\leq \exp(-ck^{\eta_1})$, for some $c,\eta_1>0$, with strict stationary distribution $\pi$.
\item [(b)] \textbf{Small ball condition.} Let $\cF\subset L_2(\pi)$ be closed, convex class of functions and define $\cF-\cF \coloneqq \{f-h:f,h\in\cF\}$. Then, the  function class $\mathcal{F}$ is such that there is a $\tau>0$ for which $Q_{\cF-\cF}(2\tau)>0$, where $Q_{\cH}(u)=\inf_{h\in\cH}\pro\left(|h|\geq u\|h\|_{L_2}\right).$
\item [(c)] \textbf{Deviations of interaction.} The stationary noise $\xi_1$ and the error $(f-f^*)(X_1)$ satisfy either:
\begin{enumerate}
    \item [(i)] For all $f\in \cF$, and $Z_1\sim\pi$, for some $\eta_2>0$ we have 
	\begin{align}\label{eq:exptail}
		\pro\left(|\xi_1(f-f^*)(X_1)-\expec{\xi_1(f-f^*)(X_1)}{}| \ge t\right)\leq \exp(1-t^{\eta_2}).
	\end{align}
	or,
	\item [(ii)] For all $f\in \cF$, and $Z_1\sim\pi$, for some $\eta_2>2$, and $c>0$ we have 
	\begin{align}\label{eq:polytail}
		\pro\left(|\xi_1(f-f^*)(X_1)-\expec{\xi_1(f-f^*)(X_1)}{}| \ge t\right)\leq ct^{-\eta_2}.
	\end{align}
\end{enumerate}
\item [(d)] \textbf{Heavy-tail/data-mixing trade-off.} 
Under condition (c)-(i), $1/\eta\coloneqq 1/\eta_1+1/\eta_2 > 1$.
\end{enumerate}
\end{assumption}

Condition (a) above, exponentially $\beta$-mixing data, has been assumed in various works, see for example~\cite{vidyasagar2013learning,wong2020,kuznetsov2017generalization}, to obtain rates of convergence for ERM procedures in general. Indeed, (exponential) mixing assumption holds in several time-series applications. For example,~\cite{mokkadem1988mixing} showed that certain ARMA processes can be modeled as an exponentially $\beta$-mixing stochastic process. Furthermore~\cite{vidyasagar2006learning} showed that globally exponentially stable unforced dynamical systems subjected to finite-variance continuous density input noise give rise to exponentially  mixing stochastic process; see also~\cite{farahmand2012regularized}. Condition (b), referred to as the well-known small-ball condition, has been previously employed in the $\iid$ case~\cite{pmlr-v35-mendelson14}. Intuitively, it models \emph{heavy-tailedness} by restricting the mass allowed near any small neighborhoods of zero; thus, forcing the tails to be necessarily heavy. To our knowledge, the small-ball condition has not been used under dependent \dgp assumptions. Condition (c) is proposed in this work as a way to model the interaction between the stationary noise $\xi_1$ and the stationary error $(f-f^*)(X_1)$. For the $\iid$ setting,~\cite{pmlr-v35-mendelson14} modeled the interaction between $\xi_1$ and $(f-f^*)(X_1)$ uniformly over the class of $\mathcal{F}$ via the multiplier empirical process and captured the complexity through a parameter $\alpha_N$ (see~\eqref{eq:alphadef} for the definition). This requires using symmetrization argument in the proof which is not applicable in the non-$\iid$ setting that we consider in this work. In addition, how different tail conditions on the data and noise affect the high-probability statement on the learning rate is not apparent from the parameter $\alpha_N$. We revisit the relationship between our condition (in the context of $\iid$ observations) and the multiplier empirical process approach used in~\cite{pmlr-v35-mendelson14} in Section~\ref{sec:reltoempiricalprocess}. Finally, condition (d) models the relationship between the allowed degree of dependency and the allowed degree of interaction between $\xi_1$ and $(f-f^*)(X_1)$. 

Next, we modify condition (c) in Assumption~\ref{as:datagenprocess} for the case of general locally strongly-convex loss functions because now the interaction part involves $\ell'(\xi)$ instead of $\xi$ (recall the decomposition \eqref{eq:decomp}). 
\begin{assumption}[Convex loss]\label{as:datagenprocesscon}
When the loss function is locally strongly-convex around the origin and globally convex, condition (c) in Assumption~\ref{as:datagenprocess} is modified as:
\begin{enumerate}[leftmargin=.25in]
\vspace{-.05in}
\item [(c)] \textbf{Deviations of interaction.} The stationary noise $\xi_1$ and the error $(f-f^*)(X_1)$ satisfy either,
\vspace{-.05in}
\begin{enumerate}
    \item[(i)] For all $f\in \cF$, and $Z_1\sim\pi$, for some $\eta_2>0$ we have 
	\begin{align}
		\pro\left(|\ell'(\xi_1)(f-f^*)(X_1)-\expec{\ell'(\xi_1)(f-f^*)(X_1)}{}| \ge t\right)\leq \exp(1-t^{\eta_2}).
	\end{align}
	or,
	\item[(ii)] For all $f\in \cF$, and $Z_1\sim\pi$, for some $\eta_2>2$, and $c>0$ we have 
	\begin{align}
		\pro\left(|\ell'(\xi_1)(f-f^*)(X_1)-\expec{\ell'(\xi_1)(f-f^*)(X_1)}{}| \ge t\right)\leq ct^{-\eta_2}.
	\end{align}
\end{enumerate}
\end{enumerate}
\end{assumption}

%

\textbf{Complexity Measures:} We now introduce the complexity measures that play a crucial role in characterizing the rates of convergence. The use of $\beta$-mixing assumption enables us to define complexity measures based on the blocking technique proposed by~\cite{yu1994rates}, also utilized by the works of~\cite{mohri2008rademacher, kuznetsov2017generalization, wong2020}. We partition the training sample of size $N$, $S\coloneqq\{Z_i\}_{i=1}^N$, into two sequences of blocks $S_a$ and $S_b$. Each block in $S_a$, and $S_b$ is of length $a$, and $b$ respectively. $S_a$ and $S_b$, both are of length $\mu$, i.e., $\mu(a+b)=N$. Formally, $S_a$ and $S_b$ are given by
\begin{talign*}\textstyle
	&S_a=\left(Z_1^{(a)},Z_2^{(a)},\cdots,Z_{\mu}^{(a)}\right) \quad \text{with  } Z_i^{(a)}=\{z_{(i-1)(a+b)+1},\cdots,z_{(i-1)(a+b)+a}\},\\
	&S_b=\left(Z_1^{(b)},Z_2^{(b)},\cdots,Z_{\mu}^{(b)}\right) \quad \text{~with  } Z_i^{(b)}=\{z_{(i-1)(a+b)+a+1},\cdots,z_{(i-1)(a+b)+a+b}\}. \numberthis\label{eq:partition}
\end{talign*}  
Based on this blocking technique, we require the following definition of Rademacher complexity.
\vspace{-.05in}
\begin{definition}\label{def:omegadef}
Let $\{\widetilde{X}_i\}_{i=1}^\mu$ be an $\iid$ sample from the strict stationary distribution $\pi$. Let $\mathcal{D}$ be the unit-$L_2(\pi)$ ball centered at $f^*$.
%
For every $\gamma>0$, define
\vspace{-.02in}
	\begin{align}\label{eq:localrad}
		\omega_\mu(\mathcal{F}-\mathcal{F},\gamma)\coloneqq\inf\left\lbrace r>0:\expec{\sup_{h\in(\mathcal{F}-\mathcal{F})~\cap~r\mathcal{D}}\abs{\frac{1}{\mu}\sum_{i=1}^\mu\epsilon_ih(\widetilde{X}_i)}}{}\leq \gamma r\right\rbrace,
		\vspace{-.02in}
	\end{align}
	where $\{\epsilon_i\}_{i=1}^\mu$ are $\iid$ Rademacher variables taking values $\pm 1$ with probability $1/2$.
\end{definition}
 The quantity $\omega_\mu(\cH,\gamma)$ provides a localized complexity measure for the function class $\mathcal{F}$, and serves as a generalization of the standard Rademacher complexity in the non-$\iid$ setting. For the case of locally strongly-convex losses, we need the following related measures of complexity. 

\begin{definition}\label{def:omegadefcon}
For a function class $\cH\!\subset\! L_2(\pi)$, a sample of size $N$ from a strictly stationary $\beta$-mixing sequence with stationary distribution $\pi$ satisfying Assumption~\ref{as:datagenprocess}-(a), and $\zeta_1, \zeta_2>0$, we define:
\begin{align*}
    \omega_{1}(\cH,N,\zeta_1)&=\inf\left\lbrace r>0:\expec{\|G\|_{\cH\cap rD}}{}\leq \zeta_1rN^{\frac{\eta_1}{2(1+\eta_1)}}\right\rbrace ~~\text{and}~~
    \omega_{2}(\cH,\mu,\zeta_2)=\omega_\mu(\mathcal{H},\zeta_2),
\end{align*}
where $\mu=N^\frac{\eta_1}{(1+\eta_1)}$, $\|G\|_{\cH}=\sup_{h\in \cH}G_h$, and $\{G_h:h\in\cH\}$ is the canonical Gaussian process indexed by $\cH$ with a covariance induced by $L_2(\pi)$. Moreover, we let $$\omega_Q(\cF-\cF,N,\zeta_1,\zeta_2)\coloneqq\max(\omega_{1}(\cH,N,\zeta_1),\omega_{2}(\cH,\mu,\zeta_2)).$$ 
\end{definition}
Note that in Definitions~\ref{def:omegadef} and \ref{def:omegadefcon}, the scaling is in terms of number of blocks $\mu$ instead of $N$. The number of blocks $\mu$ can be thought of as the effective sample size under dependency, and as $\eta_1\to\infty$, one has $\mu\to N$. The term $\expec{\|G\|_{\cH\cap rD}}{}$ appearing in Definition~\ref{def:omegadefcon} is termed as the localized Gaussian width and is also a widely used complexity measure in the literature. Note that while in the case of quadratic loss function, the bound is in terms of local Rademacher-based complexity measure, whereas in the convex case, we require both Gaussian and Rademacher-based complexity measures to establish the bound, mostly due to technical reasons in the proof. An intuitive explanation for this has eluded us thus far; see also~\cite[Lemma 4.2]{mendelson2018learning}. 

\textbf{Concentration Inequalities for Heavy-tails:} We now restate \cite[ Theorem 1]{merlevede2011bernstein}, in a form adapted to our setting below. This result is required to handle interactions of noise and input that satisfy condition (c)-(i) of Assumption~\ref{as:datagenprocess}. It is straightforward to check that the conditions required by~\cite{merlevede2011bernstein} are immediately satisfied under our Assumption~\ref{as:datagenprocess}. Indeed, while the results in~\cite{merlevede2011bernstein} are stated for $\tau$-mixing sequences, condition (a) in Assumption~\ref{as:datagenprocess} implies that the process $\{Z_i\}_{i=-\infty}^\infty$ is exponentially $\tau$-mixing \cite{chwialkowski2016kernel}, i.e., for a constant $c'>0$, $\tau(k)\leq e^{-c'k^{\eta_1}}$.
\begin{lemma}[\cite{merlevede2011bernstein}]\label{lm:rio}
	Let $\{W_j\}_{j \geq 1}$ be a sequence of zero-mean real-valued random variables satisfying conditions (a), (c)-(i), and (d) of Assumption~\ref{as:datagenprocess}. Define $\varkappa_M(x)=(x\wedge M)\vee(-M)$, for some $M$, where $(x\wedge y)=\min(x,y)$, and $(x\vee y)=\max(x,y)$,  and set,
	\vspace{-.05in}
	\begin{align}\label{eq:parameterv}
		V\coloneqq\sup_{M\geq 1}\sup_{i>0}\Big(\textsc{var}(\varkappa_M(W_i))+2\sum_{j>i}\abs{\textsc{cov}\left(\varkappa_M(W_i),\varkappa_M(W_j)\right)}\Big).
	\end{align} 
Note that $V$ is finite. Then, for any $N \geq  4$, there exist positive constants $C_1, C_2, C_3$, and $C_4$ depending only on $c,\eta_1,\eta_2$ such that, for any $t>0$, we have
	\begin{align*}
\hspace{-0.1in}\pro\left(\sup_{j\leq N}\abs{\sum_{i=1}^jW_i}\geq t\right)\leq  Ne^{-\frac{t^\eta}{C_1}}+e^{-\frac{t^2}{C_2NV}}+e^{-\frac{t^2}{C_3N}\exp\left(\frac{t^{\eta(1-\eta)}}{C_4(\log t)^\eta}\right)}. \numberthis\label{eq:rio}
	\end{align*}
\end{lemma}
To deal with polynomially tailed interactions, i.e., under condition (c)-(ii) of Assumption~\ref{as:datagenprocess}, we prove a concentration inequality for the sum of exponentially $\beta$-mixing random variables with polynomially heavy-tails, which may be of independent interest. 

\begin{lemma}[Concentration for heavy-tailed $\beta$-mixing sum]\label{lm:heavytailconc}
	Let $\{W_j\}_{j \geq 1}$ be a sequence of zero-mean real valued random variables satisfying conditions (a) and (c)-(ii) of Assumption~\ref{as:datagenprocess}, for some $\eta_2>2$. Then for any positive integer $N$, $0\leq d_1\leq 1$, and $d_2\geq 0$, and for any $t>1$, we have,
	\begin{align*}
		\mathbb{P}\left(\sup_{j\leq N}\abs{\sum_{i=1}^jW_i}\geq t\right)
		\leq& \frac{2^{\eta_2+3}}{(d_2\log t)^\frac{1-\eta_2}{\eta_1}}\frac{N}{t^{(1+d_1(\eta_2-1))}}+8\frac{N}{t^{(1+c'd_2)}}
		+2e^{-\frac{t^{2-2d_1}(d_2\log t)^{1/\eta_1}}{9N}},
		\numberthis\label{eq:heavytailconc}
	\end{align*}
where $c'>0$ is a constant.
\end{lemma}
Note that we do not need condition (d) of Assumption~\ref{as:datagenprocess} for Lemma~\ref{lm:heavytailconc}. Since the tail probabilities decay polynomially and the mixing coefficients decay exponentially fast, the effect of heavy-tail dominates and hence, there is no trade-off between $\eta_1$ and $\eta_2$. Lemma~\ref{lm:heavytailconc} extends the results of~\cite{bakhshizadeh2020sharp} (on $\iid$ heavy-tailed random variables) to the exponentially $\beta$-mixing setting. The results of~\cite{bakhshizadeh2020sharp} show that even for the $\iid$ case, the tail probability of the sum decays polynomially with $t$. We show that similar polynomial tail bounds can be obtained (up to $\log$ factors) even in the dependent setting. Furthermore, when $\beta(k)=0,k>0$, the sequence is $\iid$, in which case we recover the result of \cite{bakhshizadeh2020sharp}. We also remark that the above two results are crucial to derive our convergence rates for ERM. As a preview, when the \dgp has only $(2+\delta)$-moments, for some $\delta>0$, Lemma~\ref{lm:heavytailconc} eventually will lead to risk bounds that hold with polynomial probability, whereas when the \dgp is sub-Weibull (see Definition~\ref{def:sw}), Lemma~\ref{lm:rio} would lead to risk bounds that hold with exponential probability. 

\section{Main Results}\label{sec:mainresults}
In this section, we state our main results on the rates of convergence of ERM for both squared and convex loss functions. First, we consider the squared loss.

\begin{theorem}[Rates of ERM with squared loss]\label{th:mainthm2}
	Consider the ERM procedure with the squared error loss. For  $\tau_0<\tau^2Q_\cH(2\tau)/8$,  setting $\mu={N^rQ_\cH(2\tau)c^\frac{1}{\eta_1}}/{4}$, for some constants $c,c' >0$, and $0<r<1$,  we have, for sufficiently large $N$, and some positive constants $\widetilde{C}_1, \widetilde{C}_2$, the following:
	\begin{enumerate}[leftmargin=0.2in]
	    \item Under conditions (a), (b), (c)-(i), and (d) of Assumption~\ref{as:datagenprocess}, for $\ 0<\iota<1/4$,
	    \begin{align*}
\|\hat{f}-f^*\|_{L_2}\coloneqq \left(\int (\hat f - f^*)^2 d\pi\right)^\frac{1}{2} \leq \max\left\{N^{-\frac{1}{4}+\iota},\omega_\mu\left(\cF-\cF,\frac{\tau Q_{\cF-\cF}(2\tau)}{16}\right)\right\}, \numberthis\label{eq:errorboundalpha}
	\end{align*}
	with probability at least (for $V$ as defined in~\eqref{eq:parameterv})
	\begin{align*}
		1-\widetilde{C}_1N^rQ_\cH(2\tau)c^\frac{1}{\eta_1}\exp(-N^{(1-r)\eta_1})-\widetilde{C}_2N\exp\left(-(N^{\frac{1}{2}+2\iota}\tau_0)^\eta\right).\numberthis\label{eq:mathscrpalpha}
		\end{align*}
	    \item Under conditions (a), (b), and (c)-(ii) of Assumption~\ref{as:datagenprocess}, for $\ 0<\iota<(1-1/\eta_2)/4$,
	    \begin{align*}
		\|\hat{f}-f^*\|_{L_2} \leq \max\left\{N^{-\frac{1}{4}\left(1-\frac{1}{\eta_2}\right)+\iota},\omega_\mu\left(\cF-\cF,\frac{\tau Q_{\cF-\cF}(2\tau)}{16}\right)\right\}. \numberthis\label{eq:errorboundalphapoly}
	\end{align*}
	with probability at least
	\begin{align*}
		1-\widetilde{C}_1N^rQ_\cH(2\tau)c^\frac{1}{\eta_1}\exp(-N^{(1-r)\eta_1})-\widetilde{C}_2\tau_0^{-\frac{2\eta_2}{1+\eta_2}}N^{-\frac{4\iota \eta_2}{1+\eta_2}}. \numberthis\label{eq:thm10probpoly}
	\end{align*}
	\end{enumerate}
The detailed expression of the probabilities are provided in the Appendix (Theorem~\ref{th:mainthm2app}).
\end{theorem}
\begin{remark}
To the best of our knowledge, the above result is the first result on understanding rates of convergence of ERM with squared error loss functions with unbounded noise (as well as the loss) for heavy-tailed dependent data. For a wide range of function classes  $\mathcal{F}$ used in practice (see Section~\ref{sec:illustexp}), the dominant term in the rate of convergence is $N^{-\frac{1}{4}+\iota}$, for $0 < \iota <1/4$. Furthermore, under the stronger condition (c)-(i) of Assumption~\ref{as:datagenprocess}, the risk bound holds with exponential probability, whereas under the weaker condition (c)-(ii), it holds only with polynomial probability. 
\end{remark}
We now show that when the small-ball condition in Assumption~\ref{as:datagenprocess} is replaced with the stronger norm-equivalence assumption, also considered in~\cite{mendelson2020robust}, one could obtain improved rates. Examples of random vectors that satisfy the norm-equivalence conditions include multivariate student $t$-distribution and sub-exponential random variables. We refer to Section~\ref{sec:illustexp} for illustrative examples.

\begin{assumption}[$L_p-L_2$ norm-equivalence]\label{as:normequiv}
	Let $\cF\subset L_q(\pi)$ be a class of functions for some $q\geq 3$. The function class $\cF-\cF=\{f-h:f,h\in\cF\}$ is $L_p-L_2$ norm-equivalent for some $ p>2$, if there exists an $M_1>0$ such that, $\|h\|_{L_p} \coloneqq (\int |h|^p \, d\pi)^{1/p} \leq M_1 \|h\|_{L_2},~\forall h\in\cF-\cF$.\end{assumption}
\begin{corollary}\label{cor:mainthm2}
 For the ERM procedure with squared error loss, under Assumptions~\ref{as:datagenprocess}, with condition (b) replaced by Assumption~\ref{as:normequiv} with $p=8$, for some $0<\iota<\frac{1}{2}$ and $r, \mu$ and $\tau_0$ same as in Theorem~\ref{th:mainthm2}, for sufficiently large $N$, we have
		\begin{align*}
			\|\hat{f}-f^*\|_{L_2}\leq \max\left\{N^{-\frac{1}{2}+\iota},\omega_\mu(\cF-\cF,\tau Q_{\cF-\cF}(2\tau)/16)\right\}, \numberthis\label{eq:errorboundalphaxcor}
		\end{align*}
with probability at least (for some constants $\widetilde{C}_1$ and $\widetilde{C}_2$)
 		\begin{align*}
				1-\widetilde{C}_1N^rQ_\cH(2\tau)c^\frac{1}{\eta_1}\exp(-N^{(1-r)\eta_1})-\widetilde{C}_2N\exp\left(-(N^{2\iota}\tau_0)^\eta/M_1\right).
		\end{align*}
\end{corollary}


\begin{remark}
In the model-based nonparametric regression setting  (as discussed in \textbf{Related Works}) with $X_i$ being independent of $\xi_i$ for all $i=1,\ldots, n$, but $\xi_i$ being dependent on each other, in~\cite[Proposition 3]{han2019convergence} authors show a lower bound of $N^{-\frac{1}{2+\epsilon}}$ for some $\epsilon >0$, for sufficiently heavy-tailed input
$X_i$. The above result provides an upper bound of similar order, for a more general setting in comparison to~\cite{han2019convergence}.  We also remark that for Corollary~\ref{cor:mainthm2}, in the model-based setting, if we assume that $\xi_i$ is independent of $X_i$, we have the same conclusion with just $p=4$ instead of $p=8$. Furthermore, note that in part 2 of Theorem~\ref{th:mainthm2}, $L_8$ norm does not exist for $\eta_2\leq 8$. We have elaborated more on how this result compares Theorem 3.1 in \cite{pmlr-v35-mendelson14} later in Section~\ref{sec:reltoempiricalprocess}.
\end{remark}

We now present our results for the class of convex loss functions that are locally strongly-convex. 
\begin{assumption}[Convex loss]\label{as:lscgc}
    The loss function $\ell:\mathbb{R}\to \mathbb{R}^+\cup \{0\}$ is a convex loss function which is strongly convex in the neighborhood of $0$, i.e., there exists a $t_2>0$ such that for any $x,y\in[-t_2,t_2]$, $\ell(y)\geq \ell(x)+ \ell(x)'(y-x)+\mu_c(y-x)^2/2$ for some constant $\mu_c>0$.
\end{assumption}
\begin{theorem}[Rates of ERM with convex loss]\label{th:mainthm2con}
	Consider ERM with loss functions that satisfy Assumption~\ref{as:lscgc}. For  $\tau_0<c_2Q_{\cF-\cF}(2\tau)\rho(0,t_2)\tau^2$, $t_2=\mathcal{O}((\kappa_0+1/\sqrt{Q_{\cH}(2\tau)})\|\xi\|_{L_2})$, setting $\mu = N^{\eta_1/(1+\eta_1)}$, for some constants $c,c' >0$, we have, for any $N \geq 4$, the following:
	\begin{enumerate}[leftmargin=0.2in]
	    \item Under conditions (a), (b), (c)-(i), and (d) of Assumption~\ref{as:datagenprocess}, for $\ 0<\iota<\frac{1}{4}$, 
	    \begin{align*}
\|\hat{f}-f^*\|_{L_2} \leq \max\left\{N^{-\frac{1}{4}+\iota}, 2\omega_Q(\cF-\cF,N,\zeta_1,\zeta_2)\right\}, \numberthis\label{eq:errorboundalphacon}
	\end{align*}
	with probability at least (for $V$ is defined in~\eqref{eq:parameterv} and some positive $c_9, c_{10}, \widetilde{C}_3$)
	\begin{align*}
	&1-c_9Q_{\cH}(2\tau)^{1-\frac{1}{\eta_1}} N^{\eta_1/(1+\eta_1)}e^{-c_{10}Q_{\cH}(2\tau)^{1+\frac{1}{\eta_1}}N^{\frac{\eta_1}{1+\eta_1}}}
	-\widetilde{C}_3N\exp\left(-(N^{\frac{1}{2}+2\iota}\tau_0)^\eta/C_1\right).
\end{align*}
	    \item Under conditions (a), (b), and (c)-(ii) of Assumption~\ref{as:datagenprocess}, for $\ 0<\iota<(1-1/\eta_2)/4$,
	    \begin{align*}
		\|\hat{f}-f^*\|_{L_2} \leq \max\left\{N^{-\frac{(1-1/\eta_2)}{4}+\iota}, 2\omega_Q(\cF-\cF,N,\zeta_1,\zeta_2)\right\}, \numberthis\label{eq:errorboundalphapolycon}
	\end{align*}
	with probability at least (for constants $c_9, c_{10}, \widetilde{C}_4>0$)
	\begin{align*}
		&1-c_9Q_{\cH}(2\tau)^{1-\frac{1}{\eta_1}} N^{\eta_1/(1+\eta_1)}e^{-c_{10}Q_{\cH}(2\tau)^{1+\frac{1}{\eta_1}}N^{\eta_1/(1+\eta_1)}}-\widetilde{C}_4\tau_0^{-\frac{2\eta_2}{1+\eta_2}}N^{-\frac{4\iota \eta_2}{1+\eta_2}}. \numberthis\label{eq:thm10probpolycon}
	\end{align*}
	\end{enumerate}
\end{theorem}
To the best of our knowledge, the above result is the first result on understanding rates of convergence of ERM with convex loss functions with unbounded noise (as well as the loss) for heavy-tailed dependent data. The above result highlights the advantage of using a robust loss function, e.g. Huber loss, over a quadratic loss function. For example, if $(f-f^*)(X)$ has a sub-Weibull tail and the noise $\xi$ has polynomial tail, one can still obtain risk bounds with exponential probability. This is because in this case $\ell'(\xi)(f-f^*)(X)$ can still be sub-Weibull for a suitable chosen $\ell'(\xi)$. Such a situation arises, for example, when there are outliers even if the data is light-tailed. With a squared error loss, one won't be able to obtain a risk bound with exponential probability in this scenario.  We will illustrate this in Section~\ref{sec:exhuber} through Huber loss, a popular choice of robust loss function in robust statistics. Similar to the quadratic case, we also have an improved result when the small-ball condition is replaced with the norm-equivalence condition. Due to space constraints, we state and prove it in the Appendix~\ref{cor:mainthm2con}.

%% file: examples.tex
\vspace{-0.1in}
\section{Illustrative Examples}\label{sec:illustexp}
\vspace{-.05in}
We illustrate the results of Section~\ref{sec:mainresults} with three examples, based on sub-Weibull random variables and Pareto random variables, that are canonical models of heavy-tailed data in the literature. \vspace{-0.1in}
\subsection{Example 1:$\beta$-mixing Sub-Weibull DGP with Squared Error Loss}\label{sec:ex1}
Here, we consider sub-Weibull random variables to model the heavy-tailed behavior in the \dgp.

\begin{definition}[Sub-Weibull random vectors]\label{def:sw}
	A real-valued random variable $X$ is said to be sub-Weibull with parameter $\eta>0$, if there are constants $K_1,K_2>0$, such that we have
	$\mathbb{P}\left(|X|>t\right)\leq 2\exp\left(-(t/K_1)^\eta\right),~\text{or equivalently}~~\|X\|_p=\expec{|X|^p}{}^{1/p}\leq K_2p^\frac{1}{\eta}.$
Based on this, a random vector $X\in\mathbb{R}^d$ is said to be marginally sub-Weibull with parameter $\eta>0$ if each coordinate of $X$ is sub-Weibull with $\eta$. We use $X\sim \textsc{sw}(\eta)$ to represent this fact.
\end{definition}
The above family of distributions define a rich class of random variables,
allowing for heavier tails than sub-Gaussian tails ($\eta=2$) or sub-exponential tails ($\eta=1$).
Let $\{\delta_i\}_{i\in\mathbb{Z^+}}$ be an $\iid$ sequence of $d$-dimensional random vectors with independent coordinates with $\delta_i\sim \textsc{sw}(\eta_\delta)$. Assume that the dependent input vectors are generated according to the model
\begin{align}
	X_i=AX_{i-1}+\delta_i, \label{eq:Xdynwei}
\end{align}
where $A \in \mathbb{R}^{d\times d}$ with spectral radius less than $1$. 
For simplicity, let $A=\sigma_0^2I_d$ where $\sigma_0^2<1$,
 and $\{Y_i\}_{i\in\mathbb{Z^+}}$ be a univariate response sequence given by $Y_i={\theta^*}^\top X_i+\xi_i$, where $\theta^*\in B_1^d(R)$ belongs to the $\ell_1$-norm ball in $\mathbb{R}^d$ with the radius $R$, and $\{\xi_i\}_{i=1}^n$ is an i.i.d sequence independent of $X_i$ $\forall i$, and $\xi_i\sim \textsc{sw}(\eta_\xi)$ for $0<\eta_\xi<1$ has independent coordinates. 
To proceed with learning framework, we consider ERM with squared loss and the function class $\mathcal{F}\coloneqq \cF_R=\left\lbrace\langle\theta,\cdot\rangle:\theta\in B_1^d(R)\right\rbrace$.
We denote the difference function class $\cF_R-\cF_R$ by $\cH_R$. We show that conditions (a), (b), (c), and (d) of Assumption~\ref{as:datagenprocess}, and Assumption~\ref{as:normequiv} are satisfied, in the following section.
\subsubsection{Verification of Assumption~\ref{as:datagenprocess} for Example~\ref{sec:ex1}} 
~\cite{wang2020highdimensional} showed that the time series given by \eqref{eq:Xdynwei} is stable, strict sense stationary, with $X_i\sim \textsc{sw}(\eta_X)$, \textcolor{black}{for some $1>\eta_X>0$}. As shown in \cite{wong2020}, $\{(X_i,Y_i)\}$ is a strictly stationary sequence; thus,
we obtain that is also a $\beta$-mixing sequence with exponentially decaying coefficients as in condition (a) of Assumption~\ref{as:datagenprocess}. Now we verify the small-ball condition (b) of  Assumption~\ref{as:datagenprocess}. Let, for any $\theta=(\theta_1,\cdots,\theta_d)\neq 0$, $d_1$ denote the set of non-zero coordinates of $\theta$. W.l.o.g lets assume $T={1,2,\cdots,d_1}$. Let $\expec{X_i^2}{}=\sigma_i^2$ and $\sigma_0=\min_{1\leq i\leq {d_1}}\sigma_i$. Then $\expec{\left(\theta^\top X\right)^2}{}=\sum_{i=1}^{d_1}\theta_i^2\sigma_i^2\geq \sigma_0^2\|\theta\|_2^2.$ Since $X_i\sim SW(\eta_X)$, we have $\|\theta_iX_i\|_8\leq K_1|\theta_i|8^{\eta_X}.$ So,
\begin{align}
	\|\theta^\top X\|_8\leq K_1\|\theta\|_18^{\eta_X}\leq\frac{K_1\|\theta\|_18^{\eta_X}}{\sigma_0\|\theta\|_2}\|\theta^\top X\|_2\leq \frac{K_1\sqrt{d_1}8^{\eta_X}}{\sigma_0}\|\theta^\top X\|_2.
\end{align}
Then using Lemma 4.1 of \cite{pmlr-v35-mendelson14}, for any $0<u<1$, we have $\pro\left(\abs{\theta^\top X}\geq u\|\theta^\top X\|_{L_2}\right)\geq \left((1-u^2)/( K_1^28^{2/\eta_X+1})\right)^{4/3}.$
So condition (b) of Assumption~\ref{as:datagenprocess} is true here. This also implies that Assumption~\ref{as:normequiv} is true in this case for $p=8$. As an immediate consequence of \cite[Proposition 2.3]{vladimirova2020sub}, we have that condition (c)-(i) in Assumption~\ref{as:datagenprocess} is valid here with $\eta_2=\max(\eta_X,\eta_\xi)<1$. Since $1/\eta_2>1$, condition (d) holds true.  
\begin{proposition}\label{prop:subweirate}
	Consider the learning problem described above. Then with probability at least 
\begin{align*}
	1-\widetilde{C}_1N^rQ_\cH(2\tau)c^\frac{1}{\eta_1}\exp(-N^{(1-r)\eta_1})-\widetilde{C}_2N\exp\left(-(N^{2\iota}\tau_0)^\eta/M_1\right),
\end{align*}
we have
\vspace{-.12in}
\begin{align*}
	\|\hat{f}-f^*\|_{L_2}\leq \max\left\{\frac{2c_3R\log (ed)^\frac{1}{\eta}}{\sqrt{Q_\cH(2\tau)}c^\frac{1}{2\eta_1}}N^{-\frac{1}{2}+\iota},N^{-\frac{1}{2}+\iota}\right\}.
\end{align*}
\end{proposition}
The proof of Proposition~\ref{prop:subweirate} could be found in the Appendix~\ref{p:4.1}.
\begin{remark}
 In a related setting (i.e., assuming $\theta^*$ is exactly $s$-sparse)~\cite[Corollary 9]{wong2020} presents parameter estimation error which is of the same order as $\|\hat{f}-f^*\|_{L_2}$ (indeed, for simplicity $R$ could be thought of being at the same order as $s$) since we assume $X$ has finite variance. So with slightly better probability guarantee, we recover the same rate ($\iota$ can be arbitrarily close to 0) as \cite[Corollary 9]{wong2020} in the above proposition. 
\end{remark}

\input{example2_changed}

%% file: example2_changed.tex
\vspace{-.05in}
\subsection{Example 2: $\beta$-mixing Pareto DGP with Squared Error
Loss}\label{sec:examplepoly}
\vspace{-.05in}
Let $\widetilde{X}_{t,i}$ denote the $i$-th coordinate of the vector $\widetilde{X}_t\in\mathbb{R}^d$. We consider the process given in~\cite{pillai1991semi}: For $i=1,2,\cdots,d$, $\eta_3> 2+2\iota$, where $\iota>0$ is a small number, and $t=0,1,\cdots$, define
\[   
\widetilde{X}_{t,i} = 
\begin{cases}
	2^\frac{1}{\eta_3}\widetilde{X}_{t-1,i} &\quad\text{with probability  }1/2\\
	\min\left(2^\frac{1}{\eta_3}\widetilde{X}_{t-1,i},\delta_{t,i}\right)&\quad\text{with probability  }1/2\numberthis\label{eq:arsp}
\end{cases}
\]
where $\{\delta_{t,i}\}_{i=1,2,\cdots,d, t=1,2,\cdots}$ is a sequence of $\iid$ Pareto random variables with the distribution $L_+(\delta;\eta_3,d_i)=\eta_3(d_i\delta)^{\eta_3-1}/\left(1+(d_i\delta)^{\eta_3}\right)^2$ for $\delta>0, \eta_3>2+2\iota$, and we write $X\sim L_+(\eta,\sigma)$ to denote that $X$ is a Pareto random variable with parameters $\eta$ and $\sigma$. The survival function of $\delta\sim L_+(\eta_3,d)$ is given by, $\mathbb{P}(\delta>t)=(1+(dt)^{\eta_3})^{-1}$, for $t>0$.
	Let $\widetilde{X}_{1,t}$ and $\widetilde{X}_{2,t}$ be two independent trails of the process in \eqref{eq:arsp}. Let $\{U_t\}_{t=0,1,\cdots}$ be a sequence of $\iid$ $U[0,1]$ random variables. Now consider the process ${X}_{t,i} = \widetilde{X}_{1,t,i}\mathbbm{1}(U_t\leq 1/2)-\widetilde{X}_{2,t,i}\mathbbm{1}(U_t> 1/2)$.
	The marginal distribution of $X_{t,i}$ is given by symmetric Pareto distribution, i.e.,
\begin{align}
	L(x;\eta_3,d_i)=\frac{\eta_3(d_i|\delta|)^{\eta_3-1}}{2\left(1+(d_i|\delta|)^{\eta_3}\right)^2} \quad -\infty<\delta<\infty, \eta_3>2+2\iota.
\end{align} 
	  Now, let $\{Y_i\}_{i\in\mathbb{Z^+}}$ be a sequence given by $Y_i={\theta^*}^\top X_i+\upsilon_i$, where $\theta^*\in B_1^d(R)$, as before. Let $\{\upsilon_i\}_{i=1}^n$ be an $\iid$ sequence of $0$ mean random variables independent of $X_i$ for all $i$, with heavy tails such that for all $t>0$, $\mathbb{P}\left(|\upsilon_{i}|\geq t\right)\leq 1/(1+t^{\eta_4})$, $\eta_4>2+2\iota$.
Like in Example~\ref{sec:ex1}, we consider ERM with squared loss and the function class $\mathcal{F}\coloneqq \cF_R=\left\lbrace\langle\theta,\cdot\rangle:\theta\in B_1^d(R)\right\rbrace$,
where $B_1^d(R)$ denotes the $d$-dimensional $\ell_1$-ball with radius $R$.
We denote the difference function class $\cF_R-\cF_R$ by $\cH_R$. We show that conditions (a), (b), and (3)-(ii) of Assumption~\ref{as:datagenprocess} hold here in the following section. 
\subsubsection{Verification of Assumption~\ref{as:datagenprocess} for Example~\ref{sec:examplepoly}}\label{sec:verifypareto}
	\cite{pillai1991semi} shows that the AR(1) process given by \eqref{eq:arsp} is strictly stationary if $\tilde{X}_{0,i}\sim L_+(\eta_3,d_i)$, and $\tilde{X}_{0,i}$ are independent of each other for $i=1,2,\cdots,d$. The stationary distribution is given by $L_+(\eta_3,d_i)$.
	Let $\pi(X_{t_1},X_{t_2},\cdots,X_{t_n})$ be the joint distribution of $X_{t_1},X_{t_2},\cdots,X_{t_n}$ for a set of time points $t_1,t_2,\cdots,t_n$. Now for any positive integer $k$, 
	\begin{align*}
		&\pi(\{X_{t_i+k}\}_{i=1}^n)
		=\pi(\{\tilde{X}_{1,t_i+k}\mathbbm{1}(U_{t_i+k}\leq 1/2)-\tilde{X}_{2,t_i+k}\mathbbm{1}(U_{t_i+k}> 1/2)\}_{i=1}^n)\\
		=&\pi(\{\tilde{X}_{1,t_i}\mathbbm{1}(U_{t_i}\leq 1/2)-\tilde{X}_{2,t_i}\mathbbm{1}(U_{t_i}> 1/2)\}_{i=1}^n)
		=\pi(\{X_{t_i}\}_{i=1}^n).   \numberthis\label{eq:strictstation}
	\end{align*}
The second equality above follows from the fact $\tilde{X}_{1,t}$ and $\tilde{X}_{2,t}$ are strictly stationary process and $\{U_t\}$ is $\iid$. So $X_t$ is a strictly stationary process with marginal distribution of $X_{t,i}$ given by symmetric Pareto distribution
\begin{align}
	L(x;\eta_3,d_i)=\frac{\eta_3(d_i|\delta|)^{\eta_3-1}}{2\left(1+(d_i|\delta|)^{\eta_3}\right)^2} \quad -\infty<\delta<\infty, \eta_3>2+2\iota.
\end{align} 
Without loss of generality we will assume that $d_1\leq d_2\leq\cdots\leq d_d$, $\sum_{i=1}^d1/d_i=K_0$, and $d_1\geq C_6'$ for some constants $K_0,C_6'>0$. It is shown in Lemma 1 of \cite{ristic2008generalized}, that the AR(1) process in \eqref{eq:arsp} is $\phi$-mixing with $\phi(k)=k\log 2/(2^k-1)$, $k=1,2,\cdots$.
where $\phi$-mixing coefficients are defined as in \cite{bradley2005basic}. We also have 
$\beta(k)\leq \phi(k)\leq k\log 2/(2^k-1) \leq e^{-k/3}.$
Since $X_t$ depends only on $X_{1,t}$, $X_{2,t}$, and $U_t$, $X_{1,t}$, and $X_{2,t}$ are independent and exponentially $\beta$-mixing, and $U_t$ is $\iid$, $X_t$ is also exponentially $\beta$-mixing.
\begin{itemize}
\setlength{\itemindent}{-0.2mm}
	\item Since $Y_i$ depends only on $X_i$, and $\upsilon_i$ are $\iid$, $\{(X_i,Y_i)\}$ is a strictly stationary $\beta$-mixing sequence with $\beta(k)\leq e^{-k/3}$. So condition (a) of Assumption~\ref{as:datagenprocess} is true here. 
	\item Now we will verify condition (b) of Assumption~\ref{as:datagenprocess}. Let $\sigma_{X,p}$ denote $\|X\|_p$ for $p>0$. Then
	\begin{align}
		\expec{\left(\theta^\top X\right)^2}{}=\sum_{i=1}^{d}\frac{\theta_i^2\sigma_{X_i,2}^2}{d_i^2}\geq \sigma_0^2\sum_{i=1}^{d}\frac{\theta_i^2}{d_i^2},
	\end{align}
where $\sigma_0=\min_i\sigma_{X_i,2}$, $i=1,2,\cdots,d$.
	Since $X_i\sim L(\eta_3)$, for any $\theta_i\in \mathbb{R}$ we have
	$\|\theta_iX_i\|_{\eta_3-0.5\iota}\leq K_1|\theta_i|\sigma_{X,\eta_3-0.5\iota}/d_i$.
	So,
	\begin{align}
		&\|\theta^\top X\|_{\eta_3-0.5\iota}\leq K_1\sigma_{X,\eta_3-0.5\iota}\sum_{i=1}^{d}\frac{|\theta_i|}{d_i}\leq\frac{K_1\sigma_{X,\eta_3-0.5\iota}\sum_{i=1}^{d}\frac{|\theta_i|}{d_i}}{\sigma_0\sqrt{\sum_{i=1}^{d}\frac{\theta_i^2}{d_i^2}}}\|\theta^\top X\|_2\\
		\leq& \frac{K_1\sigma_{X,\eta_3-0.5\iota}\sqrt{d}}{\sigma_0}\|\theta^\top X\|_2 .
	\end{align}
	Then from Lemma 4.1 of \cite{pmlr-v35-mendelson14} we have that the condition (b) of Assumption~\ref{as:datagenprocess} is true here. 
	\item Since $\upsilon_i$ and $X_i$ are independent, and $\expec{\upsilon_{i}}{}=0$, we have $\expec{\upsilon_it^\top X_i}{}=0$. Let $\eta_2=\min(\eta_3,\eta_4)$. Then using Markov's inequality, for any $t\in\cH_R$ and $\forall i$, we have,
	\begin{align*}
		\mathbb{P}\left(\abs{\upsilon_it^\top X_i-\expec{\upsilon_it^\top X_i}{}}\geq \tau\right)
		\leq  \frac{\|\upsilon_i\|_{\eta_2}^{\eta_2}\left(\sum_{j=1}^d\|t_j X_{i,j}\|_{\eta_2}\right)^{\eta_2}}{\tau^{\eta_2}}
		\leq \frac{\left(R\sigma_\upsilon\sigma_{X,\eta_2}d_H\right)^{\eta_2}}{\tau^{\eta_2}} ,
	\end{align*}
	where $d_H=\sum_{i=1}^dd_j^{-1}$, and for all $i,j$, $\|\upsilon_i\|_{\eta_2}=\sigma_\upsilon$, and $\|X_{i,j}\|_{\eta_2}=\sigma_{X,\eta_2}/d_j$. This implies that condition (c)-(ii) of Assumption~\ref{as:datagenprocess} is true in this setting.
\end{itemize}
\begin{proposition}\label{prop:paretorate}
	Consider the learning problem described in Section~\ref{sec:examplepoly}. Then with probability at least \eqref{eq:thm10probpoly}, we have,
	\begin{align*}
		\|\hat{f}-f^*\|_{L_2} \leq \max\left\{N^{-\frac{1}{4}\left(1-\frac{1}{\eta_2}\right)+\iota},\frac{C_9R}{\tau Q_{\cH_R}(2\tau)^{3/2}}d^{1/(\eta_2-0.5\iota)+\iota/8}N^{-1/2+\iota}\right\}. \numberthis\label{eq:errorboundalphapareto}
	\end{align*}
\end{proposition}
The proof of Proposition~\ref{prop:paretorate} could be found in the Appendix~\ref{p:prop4.2}.
\begin{remark}
Note that the heaviness of the tail dominates the exponential $\beta$-mixing rate $\eta_1$ in determining the rates of convergence. Observe that as $\eta_2\to\infty$ all the moments exist, and $N^{-(1-1/\eta_2)/4+\iota}\to N^{-1/4+\iota}$. This rate is the same as the one that we obtain under condition (c)-(i), although with weaker polynomial probability. Furthermore, the dimension dependency is polynomial here. On a related note, \cite{zhang2018ell_1} analyzes $\ell_1$-regression with a truncated loss in the $\iid$ setting with the assumption $\eta_2>2$ and gets $\sqrt{d/N}$ rate with exponential probability. We point out that in the $\iid$ setting Medians-of-mean method achieves the optimal rate with exponential probability under the stronger assumption that $\log d$ moments exist \cite{lugosi2019regularization}.
\end{remark}
\subsection{Example 3: $\beta$-mixing sub-Gaussian data and Pareto noise with Huber Loss}\label{sec:exhuber}
In this example, we consider Huber loss which satisfies Assumption~\ref{as:lscgc}:
\begin{align*}
    \ell_{T_h}(t)=
    \begin{cases}
        t^2/2 & \text{if} \abs{t}\leq T_h\\
        T_h\abs{t}-T_h^2/2 & \text{if} \abs{t}\geq T_h\,.
    \end{cases}
\end{align*}
 We now formally establish the benefits of using Huber loss when the noise $\xi$ has a polynomial tail but $(f-f^*)(X)$ has a sub-Weibull tail. Let $\{\delta_i\}_{i\in\mathbb{Z^+}}$ be an $\iid$ sequence of $d$-dimensional standard Gaussian random vectors $\delta_i\sim N(0,I_d)$. 
 To compare the performance of ERM under Huber loss to that of squared loss,
 for simplicity, we allow $X$ to be Gaussian; but a similar result will hold for sub-Weibull $\delta_i$. Assume that the input vectors are generated according to \eqref{eq:Xdynwei} where $A$ is a $d\times d$ matrix with spectral radius less than $1$. For simplicity, let $A=\sigma_0^2I_d$ where $\sigma_0^2<1$. ~\cite{wang2020highdimensional} showed that this time series is stable, strict sense stationary, with $X_i\sim N(0,1/(1-\sigma_0^4)I_d)$.  Let $\{Y_i\}_{i\in\mathbb{Z^+}}$, be the response sequence given by $Y_i={\theta^*}^\top X_i+\xi_i$, where $\theta^*\in B_1^d(R)$, $\ell_1$-norm ball in $\mathbb{R}^d$. Let $\{\xi_i\}_{i=1}^n$ be independent of $X_i$ for all $i$, and be an $\iid$ sequence of $0$ mean random variables with heavy tails such that for all $t>0$, $\mathbb{P}\left(|\xi_{i}|\geq t\right)\leq 1/(1+t^{\eta_4})$, $\eta_4>2+2\iota$ and $\textsc{var}(\xi)\leq \sigma_\xi$. Set $T_h=3\sigma_\xi$. We verify the required assumptions for this example in the following section.
 \subsubsection{Verification of Assumptions for Example~\ref{sec:exhuber}}\label{sec:verifyhuber}
Condition (a) of Assumption~\ref{as:datagenprocess} and condition (b) of Assumption~\ref{as:datagenprocesscon} are true here by the same argument as in Example~\ref{sec:ex1}. This also implies that Assumption~\ref{as:normequiv} is true in this case for $p=8$. Since we consider Huber loss, a lipschitz continuous loss, we have $l'(\xi)\leq \min(|\xi/2|,T_h)$. Set $T_h=c_{11}(\kappa_0+1/\sqrt{\epsilon})(\sigma_\xi+2R)$. As an immediate consequence of \cite[Proposition 2.3]{vladimirova2020sub}, we have that condition (c)-(i) in Assumption~\ref{as:datagenprocesscon} is valid here with $\eta_2=2$. Note that if a sequence is exponentially $\beta$-mixing, i.e., satisfies condition (a) of Assumption~\ref{as:datagenprocess} with coefficient $\eta_1>0$, then the same condition is true for all $\eta<\eta_1$. So choosing $\eta_1'<\eta_2/(\eta_2-1)$ we get Condition (d) of Assumption~\ref{as:datagenprocess} is true. 
\begin{proposition}\label{prop:huberrate}
	Consider the learning problem described in Section~\ref{sec:exhuber}. Then, for some constant $c_{12}>0$, with probability at least 
\begin{align*}
	1-c_9\epsilon^{1-\frac{1}{\eta_1}} N^{\eta_1/(1+\eta_1)}e^{-c_{10}\epsilon^{1+\frac{1}{\eta_1}}N^{\eta_1/(1+\eta_1)}}-\widetilde{C}_2N\exp\left(-(N^{2\iota}\tau_0)^\eta/M_1\right),
\end{align*}
we have
\vspace{-.12in}
\begin{align*}
	\|\hat{f}-f^*\|_{L_2}\leq \max\left\{N^{-\frac{1}{2}+\iota},c_{12}R\sqrt{\log(ed/N)}N^{-\frac{1}{2}}\right\}.
\end{align*}
\end{proposition}
The proof of Proposition~\ref{prop:huberrate} could be found in the Appendix~\ref{p:huber}.
\begin{remark}
In this setting, using squared loss would mean that condition (c)-(ii) of Assumption~\ref{as:datagenprocess} is true. So, by part 2 of Theorem~\ref{th:mainthm2}, the obtained rate would be of order $N^{-1/8}$ with polynomial probability given by \eqref{eq:thm10probpoly}, which is significantly worse than that of Proposition~\ref{prop:huberrate} with Huber loss.
\end{remark}

%% file: proofsketch.tex
\section{Proof Sketch of Theorem~\ref{th:mainthm2} and~\ref{th:mainthm2con}}
Recall the decomposition \eqref{eq:decomp}. The first and the last terms in the RHS of \eqref{eq:decomp} are handled respectively by condition (b) and (c) of Assumption~\ref{as:datagenprocess}. The basic idea is to show that if for some $f\in\cF$, $\|f-f^*\|_{L_2}$ is large, then with high probability $T_1:=N^{-1}\sum_{i=1}^N(f-f^*)^2(X_i)\geq \underbar{B}$ (Lemma~\ref{lm:th5.3eqv}) and  $T_2:=2N^{-1}\sum_{i=1}^N\xi_i(f-f^*)(X_i)\geq \bar{B}$ (see \eqref{eq:noiseconcalpha}) where $\underbar{B}+\bar{B}>0$. But since $\hat{f}$ minimizes $P_NL_f$ and $f^*\in\cF$, $P_NL_f\leq 0$. So with high probability $\|\hat{f}-f^*\|_{L_2}$ is small. In contrast to~\cite{pmlr-v35-mendelson14}, we face two major challenges: 1- For the lower bound on $T_1$, the symmetrization argument used in the $\iid$ case (e.g. \cite{pmlr-v35-mendelson14,mendelson2018learning}) is not applicable under our dependency structure. 2- For the term $T_2$, \cite{pmlr-v35-mendelson14, mendelson2018learning} use the complexity measure $\alpha_N^*(\gamma,\delta)$ (see~\eqref{eq:alphadef}) in the $\iid$ case to control the noise-input interactions, which is not possible to do in our setting; our analysis to control $T_2$ is different, through which we can show how different tail conditions on the data and noise affect the high-probability statements on the learning rate. The proof of Theorem~\ref{th:mainthm2con} for the locally strongly-convex loss functions, follows a similar strategy. However, the technical details become more involved.

%% file: appendix.tex



\appendix

\section{Proof of Lemma~\ref{lm:heavytailconc}} \label{sec:appendix}
\begin{lemma}[Concentration for heavy-tailed $\beta$-mixing sum, Lemma~\ref{lm:heavytailconc}]\label{lm:heavytailconcapp}
	Let $\{W_j\}_{j \geq 1}$ be a sequence of zero-mean real valued random variables satisfying conditions (a) and (c)-(ii) of Assumption~\ref{as:datagenprocess}, for some $\eta_2>2$. Then for any positive integer $N$, $0\leq d_1\leq 1$, and $d_2\geq 0$, and for any $t>1$, we have,
	\begin{align*}
		\mathbb{P}\left(\sup_{j\leq N}\abs{\sum_{i=1}^jW_i}\geq t\right)
		\leq& \frac{2^{\eta_2+3}}{(d_2\log t)^\frac{1-\eta_2}{\eta_1}}\frac{N}{t^{(1+d_1(\eta_2-1))}}+8\frac{N}{t^{(1+c'd_2)}}
		+2e^{-\frac{t^{2-2d_1}(d_2\log t)^{1/\eta_1}}{9N}},
	\end{align*}
where $c'>0$ is a constant.
\end{lemma}
\begin{proof}[Proof of Lemma~\ref{lm:heavytailconc}]
	Let $W_{i,M}$ denote the truncated random variable $W_i$ such that $W_{i,M}=\max(\min(W_i,M),-M)$. Then define $\Sigma_N:=\sum_{i=1}^NW_i$. Consider the partition of the samples into blocks of length $A$, $I_i=\{1+(i-1)A,\cdots,iA\}$ for $i=1,2,\cdots,2\mu_1$ where $\mu_1=\left[N/(2A)\right]$. Also let $I_{2\mu_1+1}=\{2\mu_1A+1,\cdots,N\}$. Define for a finite set $I$ of positive integers, define $\Sigma_{N,M}(I)=\sum_{i\in I}W_{i,M}$. Then we can write, for $j\leq N$
	\begin{align}
		\Sigma_j=&\sum_{i=1}^jW_i\\
		=&\sum_{i=1}^j(W_i-W_{i,M})+\sum_{i=1}^jW_{i,M}\\
		=&\sum_{i=1}^j(W_i-W_{i,M})+\sum_{i\leq [j/A]}\Sigma_{N,M}(I_{2i})+\sum_{i\leq [j/A]}\Sigma_{N,M}(I_{2i-1})+\sum_{i=A[j/A]+1}^{j}W_{i,M}.
	\end{align}
	Then we have,
	\begin{align}
		\abs{\Sigma_j}&\leq \sum_{i=1}^j\abs{W_i-W_{i,M}}+\abs{\sum_{i\leq [j/A]}\Sigma_{N,M}(I_{2i})}+\abs{\sum_{i\leq [j/A]}\Sigma_{N,M}(I_{2i-1})}+2AM\\
		\sup_{j\leq N}\abs{\Sigma_j}&\leq \sum_{i=1}^N\abs{W_i-W_{i,M}}+\sup_{j\leq N}\abs{\sum_{i\leq [j/A]}\Sigma_{N,M}(I_{2i})}+\sup_{j\leq N}\abs{\sum_{i\leq [j/A]}\Sigma_{N,M}(I_{2i-1})}+2AM.
	\end{align}
	Now we will establish concentration for each of the terms in the above expression. Using Markov's inequality, 
	\begin{align}
		\mathbb{P}\left(\sum_{i=1}^N\abs{W_i-W_{i,M}}\geq t\right)&\leq \frac{1}{t}\sum_{i=1}^N\expec{\abs{W_i-W_{i,M}}}{}\leq \frac{2}{t}\sum_{i=1}^N\int_M^\infty \mathbb{P}(\abs{W_i}\geq x)dx\nonumber\\
		&\leq  \frac{2N}{t}\int_M^\infty x^{-\eta_2}dx
		=\frac{2N}{t(\eta_2-1)}M^{1-\eta_2}.
	\end{align}
	Using Lemma 5 of \cite{dedecker2004coupling}, we get independent random variables $\{\Sigma^*_{N,M}(I_{2i})\}_{1\leq i\leq\mu_1 }$, where $\Sigma^*_{N,M}(I_{2i})$ has the same distribution as $\Sigma_{N,M}(I_{2i})$, such that, 
	\begin{align}
		\expec{\abs{\Sigma_{N,M}(I_{2i})-\Sigma^*_{N,M}(I_{2i})}}{}\leq A\tau(A).
	\end{align}
	Then, using Markov's inequality we have,
	\begin{align*}
		&\mathbb{P}\left(\sup_{j\leq N}\abs{\sum_{i\leq [j/A]}\Sigma_{N,M}(I_{2i})}\geq t\right)\\
		\leq &\mathbb{P}\left(\sup_{j\leq N}\abs{\sum_{i\leq [j/A]}(\Sigma_{N,M}(I_{2i})-\Sigma^*_{N,M}(I_{2i}))}+\sup_{j\leq N}\abs{\sum_{i\leq [j/A]}\Sigma^*_{N,M}(I_{2i})}\geq t\right)\\
		\leq &\mathbb{P}\left(\sup_{j\leq N}\abs{\sum_{i\leq [j/A]}(\Sigma_{N,M}(I_{2i})-\Sigma^*_{N,M}(I_{2i}))}\geq \frac{t}{2}\right)+\mathbb{P}\left(\sup_{j\leq N}\abs{\sum_{i\leq [j/A]}\Sigma^*_{N,M}(I_{2i})}\geq \frac{t}{2}\right)\\
		\leq &\frac{2\expec{\sup_{j\leq N}\abs{\sum_{i\leq [j/A]}(\Sigma_{N,M}(I_{2i})-\Sigma^*_{N,M}(I_{2i}))}}{}}{t}+\mathbb{P}\left(\sup_{j\leq N}\abs{\sum_{i\leq [j/A]}\Sigma^*_{N,M}(I_{2i})}\geq \frac{t}{2}\right)\\
		\leq &\frac{2\expec{\sup_{j\leq N}\sum_{i\leq \mu_1}\abs{\Sigma_{N,M}(I_{2i})-\Sigma^*_{N,M}(I_{2i})}}{}}{t}+\mathbb{P}\left(\sup_{j\leq N}\abs{\sum_{i\leq [j/A]}\Sigma^*_{N,M}(I_{2i})}\geq \frac{t}{2}\right)\\
		\leq &\frac{2A\mu_1\tau(A)}{t}+\mathbb{P}\left(\sup_{j\leq N}\abs{\sum_{i\leq [j/A]}\Sigma^*_{N,M}(I_{2i})}\geq \frac{t}{2}\right).
	\end{align*}
	The same results holds for $\{\Sigma^*_{N,M}(I_{2i-1})\}_{i=1,2,\cdots,k}$. So for any $t\geq 2AM$, we have,
	\begin{align*}
		\mathbb{P}\left(\sup_{j\leq N}\abs{\Sigma_j}\geq 6t\right)\leq& \frac{2N}{t(\eta_2-1)}M^{1-\eta_2}+\frac{4A\mu_1\tau(A)}{t}
		+\mathbb{P}\left(\sup_{j\leq N}\abs{\sum_{i\leq [j/A]}\Sigma^*_{N,M}(I_{2i})}\geq t\right)\\
		&~~+\mathbb{P}\left(\sup_{j\leq N}\abs{\sum_{i\leq [j/A]}\Sigma^*_{N,M}(I_{2i-1})}\geq t\right). \numberthis\label{eq:sumxnboundintermed}
	\end{align*}
	Now, for $\lambda>0$
	\begin{align*}
		 \mathbb{P}\left(\sup_{j\leq N}\abs{\sum_{i\leq [j/A]}\Sigma^*_{N,M}(I_{2i})}\geq t\right)
		&\leq e^{-\lambda t}\expec{\exp\left(\lambda \sup_{j\leq N}\abs{\sum_{i\leq [j/A]}\Sigma^*_{N,M}(I_{2i})} \right)}{}\\
		&\leq  e^{-\lambda t}\expec{\exp\left(\lambda \sum_{i\leq \mu_1}\abs{\Sigma^*_{N,M}(I_{2i})} \right)}{}\\
		&\leq  e^{-\lambda t}\Pi_{i=1}^{\mu_1}\expec{\exp\left(\lambda \abs{\Sigma_{N,M}(I_{2i})} \right)}{}.
	\end{align*}
	We have $\abs{\Sigma_{N,M}(I_{2i})}\leq AM$. So $\abs{\Sigma_{N,M}(I_{2i})}$ is a sub-gaussian random variable and consequently,
	\begin{align*}
		\mathbb{P}\left(\sup_{j\leq N}\abs{\sum_{i\leq [j/A]}\Sigma^*_{N,M}(I_{2i})}\geq t\right)
		\leq e^{-\lambda t}e^{\frac{\lambda^2\mu_1A^2M^2}{2}}.
	\end{align*}
	Optimizing over $\lambda>0$ we have,
	\begin{align*}
		\mathbb{P}\left(\sup_{j\leq N}\abs{\sum_{i\leq [j/A]}\Sigma^*_{N,M}(I_{2i})}\geq t\right)
		\leq e^{-\frac{t^2}{2\mu_1A^2M^2}}. \numberthis\label{eq:sumevenbound}
	\end{align*}
	Similarly, we also obtain
	\begin{align*}
		\mathbb{P}\left(\sup_{j\leq N}\abs{\sum_{i\leq [j/A]}\Sigma^*_{N,M}(I_{2i-1})}\geq t\right)
		\leq e^{-\frac{t^2}{2\mu_1A^2M^2}}. \numberthis\label{eq:sumoddbound}
	\end{align*}
	From \eqref{eq:sumxnboundintermed}, \eqref{eq:sumevenbound}, and \eqref{eq:sumoddbound} we have,
	\begin{align*}
		\mathbb{P}\left(\sup_{j\leq N}\abs{\Sigma_j}\geq 6t\right)\leq& \frac{2N}{t(\eta_2-1)}M^{1-\eta_2}+\frac{4A\mu_1\tau(A)}{t}
		+2e^{-\frac{t^2}{2\mu_1A^2M^2}}.
	\end{align*}
	Condition (a) in Assumption~\ref{as:datagenprocess} implies that the process $\{Z_i\}_{i=-\infty}^\infty$ is exponentially $\tau$-mixing \cite{chwialkowski2016kernel}, i.e., for a constant $c'>0$, $\tau(k)\leq e^{-c'k^{\eta_1}}$. Then we have
	\begin{align*}
		\mathbb{P}\left(\sup_{j\leq N}\abs{\Sigma_j}\geq t\right)\leq& \frac{12N}{t(\eta_2-1)}M^{1-\eta_2}+\frac{24A\mu_1\exp(-c'A^{\eta_1})}{t}
		+2e^{-\frac{t^2}{72\mu_1A^2M^2}}.
	\end{align*}
	As $2A\mu_1\leq N\leq 3A\mu_1$ and $\eta_2>2$,
	\begin{align*}
		\mathbb{P}\left(\sup_{j\leq N}\abs{\Sigma_j}\geq t\right)\leq& \frac{12NM^{1-\eta_2}}{t}+\frac{8N\exp(-c'A^{\eta_1})}{t}
		+2e^{-\frac{t^2}{36NAM^2}}.
	\end{align*}
	Now choosing 
	\begin{align}
		M=\frac{t^{d_1}}{2(d_2 \log t)^\frac{1}{\eta_1}}, \quad A=\left(d_2\log t\right)^\frac{1}{\eta_1}, \qquad 0\leq d_1\leq 1,\quad d_2\geq 0,	
	\end{align}
	we have, $2AM\leq t$, and 
	\begin{align*}
		\mathbb{P}\left(\sup_{j\leq N}\abs{\Sigma_j}\geq t\right)\leq& \frac{2^{\eta_2+3}}{(d_2\log t)^\frac{1-\eta_2}{\eta_1}}Nt^{-(1+d_1(\eta_2-1))}+8Nt^{-(1+d_2c')}
		+2e^{-\frac{t^{2-2d_1}(d_2\log t)^{1/\eta_1}}{9N}}.
	\end{align*}
\end{proof}
\section{{Proofs of Section~\ref{sec:mainresults}}}
\subsection{Proofs for squared error loss}
Similar to the decomposition \eqref{eq:decomp}, for squared loss we have
\begin{align*}
	P_NL_f=\frac{1}{N}\sum_{i=1}^N (f-f^*)^2(X_i)+\frac{2}{N}\sum_{i=1}^N \xi_i(f-f^*)(X_i),
	\vspace{-.05in}
\end{align*}
Since $\cF$ is convex, we also have
\begin{align*}
    \expec{\xi(f-f^*)(X)}{}\geq 0.
\end{align*}
Then,
\begin{align}\label{eq:decompwithexpecsqr}
	P_NL_f\geq\frac{1}{N}\sum_{i=1}^N (f-f^*)^2(X_i)+\frac{2}{N}\sum_{i=1}^N (\xi_i(f-f^*)(X_i)-\expec{\xi_i(f-f^*)(X_i)}{}).
	\vspace{-.05in}
\end{align}
Now our goal is to establish a lower bound (Lemma~\ref{lm:th5.3eqv}) on the first term of the RHS of \eqref{eq:decompwithexpecsqr}, and a two-sided bound (\eqref{eq:noiseconcalpha} and \eqref{eq:noiseconcalphapoly}) on the second term when $\|f-f^*\|_{L_2}$ is large. Combining these bounds we will show that if $\|f-f^*\|_{L_2}$ is large then $P_NL_f>0$ which implies $f$ cannot be a minimizer of empirical risk because for the minimizer $\hat{f}$ we have $P_NL_{\hat{f}}\leq 0$. 
\begin{lemma}\label{lm:th5.3eqv}
	Let Condition (a) and (b) of Assumption~\ref{as:datagenprocess} be true. Given $f^*\in \cF$, set $\cH=\cF-f^*$. Then, for every $\rho>\omega_\mu(\cH,\tau Q_\cH(2\tau)/16)$, with probability at least  $\mathscr{P}_1$, if $\|f-f^*\|_{L_2}\geq \rho$, we have,
	\begin{align}
		\abs{\{i:\abs{(f-f^*)(X_i)}\geq\tau\|f-f^*\|_{L_2}\}}\geq \frac{NQ_\cH(2\tau)}{4}.
	\end{align}
\end{lemma}
The proof of Lemma~\ref{lm:th5.3eqv} follows easily by combining the results of Lemma~\ref{lm:thilowbound}, and Corollary~\ref{cor:Vbound} which we state next. 

\begin{lemma}\label{lm:thilowbound}
	Let $S(L_2)$ be the $L_2(\pi)$ unit sphere and let $\cH\subset S(L_2)$. Consider the partition in \eqref{eq:partition}.  Under conditions~(a) and (b) of Assumption~\ref{as:datagenprocess}, by setting $\mu=\frac{NQ_\cH(2\tau)c^\frac{1}{\eta_1}}{4\mathscr{G}(N)^\frac{1}{\eta_1}}$ for some $\mathscr{G}(N)\leq \frac{cQ_\cH(2\tau)^{\eta_1}N^{\eta_1}}{4^{\eta_1}}$, if,
	\begin{align}
		\mathfrak{R}_\mu(\cH)\leq \frac{\tau Q_\cH(2\tau)N}{16\mu}, \label{eq:Rmubound}
	\end{align}
	then with probability at least $1-2\exp\left(-\frac{NQ_\cH(2\tau)^3}{2(4-Q_\cH(2\tau))^2}\left(\frac{c}{\mathscr{G}(N)}\right)^\frac{1}{\eta_1}\right)-\frac{NQ_\cH(2\tau)c^\frac{1}{\eta_1}}{4\mathscr{G}(N)^\frac{1}{\eta_1}}\exp(-\mathscr{G}(N))$, we have
	\begin{align*}
		\inf_{h\in\cH}\abs{\{i:\abs{h(X_i)}\geq\tau\}}\geq & \frac{NQ_\cH(2\tau)}{4}.
		\numberthis\label{eq:hlowerbound}
	\end{align*}
\end{lemma}
\begin{remark} We have the following illustrative instantiations of Lemma~\ref{lm:thilowbound}:
\begin{enumerate}
\item If one sets $\mathscr{G}(N)=k\log N\leq \frac{cQ_\cH(2\tau)^{\eta_1}N^{\eta_1}}{4^{\eta_1}}$, then the statement of Lemma~\ref{lm:thilowbound} holds as long as,
		\begin{align}
			\mathfrak{R}_\mu(\cH)\leq \frac{\tau (k\log N)^\frac{1}{\eta_1}}{4c^\frac{1}{\eta_1}}, \label{eq:RmuboundGlogN}
		\end{align}
	 with probability at least $$1-2\exp\left(-\frac{NQ_\cH(2\tau)^3}{2(4-Q_\cH(2\tau))^2}\left(\frac{c}{k\log N}\right)^\frac{1}{\eta_1}\right)-\frac{N^{1-k}Q_\cH(2\tau)c^\frac{1}{\eta_1}}{4(k\log N)^\frac{1}{\eta_1}}.$$ 
\item If one sets $\mathscr{G}(N)=N^r\leq \frac{cQ_\cH(2\tau)^{\eta_1}N^{\eta_1}}{4^{\eta_1}}$, \textcolor{black}{for some $0 < r < \eta_1$}, then the statement of Lemma~\ref{lm:thilowbound} holds as long as,
		\begin{align}
			\mathfrak{R}_\mu(\cH)\leq \frac{\tau (N)^\frac{r}{\eta_1}}{4c^\frac{1}{\eta_1}}, \label{eq:RmuboundGNr}
		\end{align}
	 with probability at least $$1-2\exp\left(-\frac{NQ_\cH(2\tau)^3}{2(4-Q_\cH(2\tau))^2}\left(\frac{c}{N^r}\right)^\frac{1}{\eta_1}\right)-\frac{NQ_\cH(2\tau)c^\frac{1}{\eta_1}}{4(N)^\frac{r}{\eta_1}}\exp(-N^r).$$ 
\end{enumerate}
\end{remark}
\begin{proof}[Proof of Lemma~\ref{lm:thilowbound}]
	Let $\psi_u:\mathbb{R}_+\rightarrow [0,1]$ be the function
	\begin{align*}
		\psi_u(t) = 
		\begin{cases}
			1 &\quad t\geq 2u,\\
			\frac{t}{u}-1&\quad u\leq t\leq 2u \\
			0 &\quad t<u
		\end{cases}
	\end{align*} 
		Similar to \eqref{eq:partition}, let us define sequences of i.i.d blocks $\{\tilde{Z}_i^{(a)}\}_{i=1}^\mu$, and $\{\tilde{Z}_i^{(b)}\}_{i=1}^\mu$ where the samples within each block are assumed to be drawn from the same $\beta$-mixing distribution of $\{{Z}_i^{(a)}\}_{i=1}^\mu$, and $\{{Z}_i^{(b)}\}_{i=1}^\mu$. Let $\tilde{S}_a=\left(\tilde{Z}_1^{(a)},\cdots,\tilde{Z}_\mu^{(a)}\right)$, and $\tilde{S}_b=\left(\tilde{Z}_1^{(b)},\cdots,\tilde{Z}_\mu^{(b)}\right)$. 
		Now let us concentrate on the term $\abs{P_N\psi_u(|h|)-P\psi_u(|h|)}$.
		\begin{align*}
			&\abs{P_N\psi_u(|h|)-P\psi_u(|h|)}\\
			\leq& \abs{\frac{1}{N}\sum_{i=1}^N\psi_u(h(|X_i|))-P\psi_u(|h|)}\\
			\leq & \abs{\frac{1}{N}\sum_{i=1}^\mu\sum_{j=1}^a\psi_u(|h(X_{(i-1)(a+b)+j})|)+\frac{1}{N}\sum_{i=1}^\mu\sum_{j=1}^b\psi_u(h(|X_{(i-1)(a+b)+a+j}|))-P\psi_u(|h|)}\\ 
			\leq & \abs{\frac{1}{N}\sum_{i=1}^\mu\sum_{j=1}^a\left(\psi_u(h(|X_{(i-1)(a+b)+j}|))-P\psi_u(|h|)\right)}+\frac{b\mu}{N}. \numberthis\label{eq:blockpsi}
		\end{align*}

		Using \eqref{eq:blockpsi} and Corollary 2.7 of \cite{yu1994rates}, for some $a,b,\mu$ to be chosen later such that $(a+b)\mu=N$ we have,
		\begin{align}
			&\pro\left(\abs{P_N\psi_u(|h|)-P\psi_u(|h|)}\geq t+\frac{b\mu}{N}\right)\\
			\leq &\pro\left(\abs{\frac{1}{N}\sum_{i=1}^\mu\sum_{j=1}^a\left(\psi_u(h(|X_{(i-1)(a+b)+j}|))-P\psi_u(|h|)\right)}+\frac{b\mu}{N}\geq t+\frac{b\mu}{N}\right)\\
			=&\expec{\mathbbm{1}\left(\abs{\frac{1}{N}\sum_{i=1}^\mu\sum_{j=1}^a\left(\psi_u(h(|X_{(i-1)(a+b)+j}|))-P\psi_u(|h|)\right)}\geq t\right)}{}\\
			\leq &\expec{\mathbbm{1}\left(\abs{\frac{1}{N}\sum_{i=1}^\mu\sum_{j=1}^a\left(\psi_u(h(|\tilde{X}_{(i-1)(a+b)+j}|))-P\psi_u(|h|)\right)}\geq t\right)}{}+(\mu-1)\beta(b)\\
			= &\pro\left(\abs{\frac{1}{N}\sum_{i=1}^\mu\sum_{j=1}^a\left(\psi_u(h(|\tilde{X}_{(i-1)(a+b)+j}|))-P\psi_u(|h|)\right)}\geq t\right)+(\mu-1)\beta(b)\\
			=&\pro\left(\abs{\frac{1}{\mu}\sum_{i=1}^\mu\tilde{\psi}(\tilde{Z}_i^{(a)})}\geq \frac{Nt}{\mu}\right)+(\mu-1)\beta(b), \label{eq:psitildeplusbetab}
		\end{align}
		where $$\tilde{\psi}(\tilde{Z}_i^{(a)})=\sum_{j=1}^a\left(\psi_u(h(|\tilde{X}_{(i-1)(a+b)+j}|))-P\psi_u(|h|)\right),$$ and $\mathbbm{1}(\cdot)$ is the indicator function. Observe that the function $$W(\tilde{Z}_1^{(a)},\tilde{Z}_2^{(a)},\cdots,\tilde{Z}_\mu^{(a)})=\mu^{-1}\sum_{i=1}^\mu\tilde{\psi}(\tilde{Z}_i^{(a)})$$ has bounded difference with coefficient $2a/\mu$. Then using Mcdiarmid's bounded-difference inequality on $W(\tilde{Z}_1^{(a)},\tilde{Z}_2^{(a)},\cdots,\tilde{Z}_\mu^{(a)})$ we get,
		\begin{align}
			\pro\left(\abs{\frac{1}{\mu}\sum_{i=1}^\mu\tilde{\psi}(\tilde{Z}_i^{(a)})}\geq \frac{Nt}{\mu}\right)\leq 2\exp\left(-\frac{N^2t^2}{2a^2\mu}\right). \label{eq:psitildeconc}
		\end{align}
		Combining \eqref{eq:psitildeplusbetab}, and \eqref{eq:psitildeconc}, we get
		\begin{align*}
			&\pro\left(\abs{P_N\psi_u(|h|)-P\psi_u(|h|)}\geq t+\frac{b\mu}{N}\right)
			\leq  2\exp\left(-\frac{N^2t^2}{2a^2\mu}\right)+(\mu-1)\beta(b),
			\end{align*}
			which implies 
            \begin{align*}
			&\pro\left(\abs{P_N\psi_u(|h|)-P\psi_u(|h|)}\geq \frac{4\mu}{Nu}\mathbb{R}_\mu(\cH)+\frac{b\mu}{N}+\frac{t}{\sqrt{N}}\right)\\
			\leq & 2\exp\left(-\frac{N^2\left(\frac{4\mu}{Nu}\mathbb{R}_\mu(\cH)+\frac{t}{\sqrt{N}}\right)^2}{2a^2\mu}\right)+(\mu-1)\beta(b) .
		\end{align*}	
		Also note that, for any $t$, we have $			\abs{P_N\psi_u(|h|)-P\psi_u(|h|)}\geq t$ which implies that we also have $ \sup_{h\in\cH}\abs{P_N\psi_u(|h|)-P\psi_u(|h|)}\geq t$. Hence,
		\begin{align*}
			&\pro\left(\sup_{h\in\cH}\abs{P_N\psi_u(|h|)-P\psi_u(|h|)}\geq \frac{4\mu}{Nu}\mathbb{R}_\mu(\cH)+\frac{b\mu}{N}+\frac{t}{\sqrt{N}}\right)\\
			\leq & 2\exp\left(-\frac{N^2\left(\frac{4\mu}{Nu}\mathbb{R}_\mu(\cH)+\frac{t}{\sqrt{N}}\right)^2}{2a^2\mu}\right)+(\mu-1)\beta(b). \numberthis\label{eq:dependmcdiarmid}
		\end{align*}	
		In other words, with probability at least $1-2\exp\left(-\frac{N^2\left(\frac{4\mu}{Nu}\mathbb{R}_\mu(\cH)+\frac{t}{\sqrt{N}}\right)^2}{2a^2\mu}\right)-(\mu-1)\beta(b)$, we have
		\begin{align}
			\sup_{h\in\cH}\abs{P_N\psi_u(|h|)-P\psi_u(|h|)}\leq \frac{4\mu}{Nu}\mathbb{R}_\mu(\cH)+\frac{b\mu}{N}+\frac{t}{\sqrt{N}}.
		\end{align}
Hence, we have 		
		\begin{align*}
			P_N \mathbf{1}_{\{|h|\geq u\}}\geq \inf_{h\in\cH} \pro(|h|\geq2u)-\sup_{h\in\cH}\abs{P_N\psi_u(|h|)-P\psi_u(|h|)}. \numberthis\label{eq:pninfsup}
		\end{align*}
		So, combining \eqref{eq:dependmcdiarmid}, and \eqref{eq:pninfsup}, with probability at least $1-2\exp\left(-\frac{N^2\left(\frac{4\mu}{Nu}\mathbb{R}_\mu(\cH)+\frac{t}{\sqrt{N}}\right)^2}{2a^2\mu}\right)-(\mu-1)\beta(b)$ we have
		\begin{align*}
			P_N \mathbf{1}_{\{|h|\geq u\}}\geq & \inf_{h\in\cH} \pro(|h|\geq2u)-\frac{4\mu}{Nu}\mathbb{R}_\mu(\cH)-\frac{b\mu}{N}-\frac{t}{\sqrt{N}}.
		\end{align*}
		Now, setting
		\begin{align}
			u=\tau \quad t=\frac{\sqrt{N}Q_\cH}{4} \quad a=\frac{(4-Q_\cH(2\tau))(\mathscr{G}(N))^\frac{1}{\eta_1}}{Q_\cH(2\tau)c^\frac{1}{\eta_1}} \quad b=\left(\frac{\mathscr{G}(N)}{c}\right)^\frac{1}{\eta_1} \quad \mu=\frac{NQ_\cH(2\tau)c^\frac{1}{\eta_1}}{4\mathscr{G}(N)^\frac{1}{\eta_1}},
		\end{align}
		and using the condition $\mathscr{G}(N)>c$, we get, with probability at least $$1-2\exp\left(-\frac{NQ_\cH(2\tau)^3}{2(4-Q_\cH(2\tau))^2}\left(\frac{c}{\mathscr{G}(N)}\right)^\frac{1}{\eta_1}\right)-\frac{NQ_\cH(2\tau)c^\frac{1}{\eta_1}}{4\mathscr{G}(N)^\frac{1}{\eta_1}}\exp(-\mathscr{G}(N)),$$
		we have
		\begin{align*}
			P_N \mathbf{1}_{\{|h|\geq u\}}\geq \frac{Q_\cH(2\tau)}{4}.
	\end{align*}
\end{proof}
\begin{corollary}\label{cor:Vbound}
	Let Condition (a) and (b) of Assumption~\ref{as:datagenprocess} be true. Let $\cH$ be star-shaped around $0$ and assume that there is some $\tau>0$ for which $Q_\cH(2\tau)>0$. Then for every $\rho>\omega_\mu(\cH,\tau Q_\cH(2\tau)/16)$, with probability at least $$\mathscr{P}_1\coloneqq1-2\exp\left(-\frac{NQ_\cH(2\tau)^3}{2(4-Q_\cH(2\tau))^2}\left(\frac{c}{\mathscr{G}(N)}\right)^\frac{1}{\eta_1}\right)-\frac{NQ_\cH(2\tau)c^\frac{1}{\eta_1}}{4\mathscr{G}(N)^\frac{1}{\eta_1}}\exp(-\mathscr{G}(N)),$$ for every $h\in\cH$ that satisfies $\|h\|_{L_2}\geq \rho$,
	\begin{align}
		\abs{\{i:\abs{h(X_i)}\geq\tau\|h\|_{L_2}\}}\geq N\frac{Q_\cH(2\tau)}{4}. \label{eq:ilowboundV}
	\end{align}
\end{corollary}
\begin{proof}[Proof of Corollary~\ref{cor:Vbound}]
	Let $\rho>\omega_\mu(\cH,\tau Q_\cH(2\tau)/16)$ and as $\cH$ is star-shaped around $0$,
	\begin{align}
		\mathfrak{R}_\mu(\cH\cap \rho \mathcal{D})\leq \frac{\tau Q_\cH(2\tau)}{16}\rho.
	\end{align}
Consider the set,
\begin{align}
	V=\{h/\rho:h\in \cH\cap \rho S(L_2)\}\subset S(L_2).
\end{align}
Clearly, $Q_V(2\tau)\geq Q_\cH(2\tau)$ and
\begin{align}
	\mathfrak{R}_\mu(V)=\expec{\sup_{h\in\cH\cap \rho S(L_2)}\abs{\frac{1}{\mu}\sum_{i=1}^\mu\epsilon_i\frac{h(\tilde{X}_i)}{\rho}}}{}\leq \frac{\tau Q_\cH(2\tau)}{16}\leq \frac{\tau Q_V(2\tau)}{16}.
\end{align}
Using Lemma~\ref{lm:thilowbound} on the set $V$, we get with probability at least $\mathscr{P}_1$, for every $v\in V$
\begin{align*}
	\inf_{h\in\cH}\abs{\{i:|v(X_i)\geq \tau|\}}\geq \frac{NQ_V(2\tau)}{4}
	\geq \frac{NQ_\cH(2\tau)}{4}.
\end{align*}
Now for any $h$ with $\|h\|_{L_2}\geq \rho$, since $\cH$ is star-shaped around $0$, we have $(\rho/\|h\|_{L_2})h\in \cH\cap \rho S(L_2)$ which implies, $h/\|h\|_{L_2}\in V$. So we have \eqref{eq:ilowboundV}.
\end{proof}
\begin{theorem}[Restatement of Theorem~\ref{th:mainthm2}]\label{th:mainthm2app}
	Consider the LS-ERM procedure. For  $\tau_0<\tau^2Q_\cH(2\tau)/8$,  setting $\mu=\frac{N^rQ_\cH(2\tau)c^\frac{1}{\eta_1}}{4}$, for some constants $c,c' >0$, and $0<r<1$,  we have, for any $N \geq 4$,
	\begin{enumerate}
	    \item under condition (a), (b), (c)-(i), and (d) of Assumption~\ref{as:datagenprocess}, for $\ 0<\iota<\frac{1}{4}$, 
	    \begin{align*}
\left(\int (\hat f - f^*)^2 d\pi\right)^\frac{1}{2}=\|\hat{f}-f^*\|_{L_2} \leq \max\left\{N^{-\frac{1}{4}+\iota},\omega_\mu(\cF-\cF,\tau Q_{\cF-\cF}(2\tau)/16)\right\} \numberthis\label{eq:errorboundalphaapp}
	\end{align*}
	with probability at least
	\begin{align*}
	1-&2\exp\left(-\frac{N^rQ_\cH(2\tau)^3c^\frac{1}{\eta_1}}{2(4-Q_\cH(2\tau))^2}\right)-\frac{N^rQ_\cH(2\tau)c^\frac{1}{\eta_1}}{4}\exp(-N^{(1-r)\eta_1})-N\exp\left(-\frac{(N^{\frac{1}{2}+2\iota}\tau_0)^\eta}{C_1}\right)\\
	-&\exp\left(-\frac{N^{1+4\iota}\tau_0^2}{C_2(1+NV)}\right)
	-\exp\left(-\frac{N^{4\iota}\tau_0^2}{C_3}\exp\left((1-\eta)^\eta\frac{(N^{\frac{1}{2}+2\iota}\tau_0)^{\frac{\eta(1-\eta)}{2}}}{C_42^\eta}\right)\right),\numberthis\label{eq:mathscrpalphaapp}
		\end{align*}
where $V$ is defined in~\eqref{eq:parameterv} and $C_1, C_2, C_3$ are some positive constants.
	    \item under condition (a), (b), and (c)-(ii) of Assumption~\ref{as:datagenprocess}, for $\ 0<\iota<(1-1/\eta_2)/4$,
	    \begin{align*}
		\|\hat{f}-f^*\|_{L_2} \leq \max\left\{N^{-\frac{1}{4}\left(1-\frac{1}{\eta_2}\right)+\iota},\omega_\mu(\cF-\cF,\tau Q_{\cF-\cF}(2\tau)/16)\right\} \numberthis\label{eq:errorboundalphapolyapp}
	\end{align*}
	with probability at least
	\begin{align*}
		&1-2\exp\left(-\frac{N^rQ_\cH(2\tau)^3c^\frac{1}{\eta_1}}{2(4-Q_\cH(2\tau))^2}\right)-\frac{N^rQ_\cH(2\tau)c^\frac{1}{\eta_1}}{4}\exp(-N^{(1-r)\eta_1})-8\tau_0^{-\frac{2\eta_2}{1+\eta_2}}N^{-\frac{4\iota \eta_2}{1+\eta_2}}\\
		&-\frac{2^{\eta_2+3}{c'}^\frac{1-\eta_2}{\eta_1}\tau_0^{-\frac{2\eta_2}{1+\eta_2}}}{\left(\log \left(\tau_0N^{\frac{1}{2}+\frac{1}{2\eta_2}+2\iota}\right)/2\right)^\frac{1-\eta_2}{\eta_1}}N^{-\frac{4\iota \eta_2}{1+\eta_2}}
		-2e^{-\frac{\tau_0^{\frac{2\eta_2}{1+\eta_2}}\left(\log \left(\tau_0N^{\frac{1}{2}+\frac{1}{2\eta_2}+2\iota}\right)/2\right)^{1/\eta_1}}{9{c'}^{1/\eta_1}}N^{\frac{4\iota \eta_2}{1+\eta_2}}} \numberthis\label{eq:thm10probpolyapp}
	\end{align*}
	\end{enumerate}
	
\end{theorem}
\begin{proof}[Proof of  Theorem~\ref{th:mainthm2app}] We first prove \textbf{Part 1}. We will denote the class $\cF-f^*$ by $\cH$. From Lemma~\ref{lm:thilowbound} it follows that if $\rho>\omega_\mu(\cH,\tau Q_{\cF-\cF}(2\tau)/16)$, then with probability at least $$\mathscr{P}_1=1-2\exp\left(-\frac{NQ_\cH(2\tau)^3}{2(4-Q_\cH(2\tau))^2}\left(\frac{c}{\mathscr{G}(N)}\right)^\frac{1}{\eta_1}\right)-\frac{NQ_\cH(2\tau)c^\frac{1}{\eta_1}}{4\mathscr{G}(N)^\frac{1}{\eta_1}}\exp(-\mathscr{G}(N)),$$ for every $f\in \cF$ that satisfies $\|f-f^*\|_{L_2}\geq \rho$, 
\begin{align}
	\frac{1}{N}\sum_{i=1}^N(f-f^*)^2(X_i)\geq \frac{\tau^2\|f-f^*\|^2_{L_2}Q_\cH(2\tau)}{4}.
\end{align}
So, with probability at least $\mathscr{P}_1$, for every $f\in \cF$ that satisfies $\|f-f^*\|_{L_2}\geq \rho$,
\begin{align*}
	P_N\mathcal{L}_f\geq2\left( \frac{1}{N}\sum_{i=1}^N\xi_i(f-f^*)(X_i)-\expec{\xi(f-f^*)}{}\right)+\frac{\tau^2\|f-f^*\|^2_{L_2}Q_\cH(2\tau)}{4}. \numberthis\label{eq:PnLfintermedboundalpha}
\end{align*}
	When $\|f-f^*\|_{L_2}\geq \mathscr{A}(N)> 2(N\tau_0)^{-1/2}$, we have $\log (N\tau_0\|f-f^*\|_{L_2}^2)\leq 2(N\tau_0\|f-f^*\|_{L_2}^2)^{(1-\eta)/2}/(1-\eta)$. Under Conditions~(a), (c)-(i), and (d) of Assumption~\ref{as:datagenprocess}, using Lemma~\ref{lm:rio}, we get
	\begin{align*}
		&\pro\left(\abs{\frac{1}{N}\sum_{i=1}^N\xi_i(f-f^*)(X_i)-\expec{\xi(f-f^*)}{}}\geq \tau_0\|f-f^*\|_{L_2}^2\right)\\
		\leq& N\exp\left(-\frac{(N\tau_0\|f-f^*\|_{L_2}^2)^\eta}{C_1}\right)+\exp\left(-\frac{N^2\tau_0^2\|f-f^*\|_{L_2}^4}{C_2(1+NV)}\right)\\
		&~+\exp\left(-\frac{N\tau_0^2\|f-f^*\|_{L_2}^4}{C_3}\exp\left(\frac{(N\tau_0\|f-f^*\|_{L_2}^2)^{\eta(1-\eta)}}{C_4(\log (N\tau_0\|f-f^*\|_{L_2}^2))^\eta}\right)\right)\\
		\leq& N\exp\left(-\frac{(N\tau_0\|f-f^*\|_{L_2}^2)^\eta}{C_1}\right)+\exp\left(-\frac{N^2\tau_0^2\|f-f^*\|_{L_2}^4}{C_2(1+NV)}\right)\\
		&~+\exp\left(-\frac{N\tau_0^2\|f-f^*\|_{L_2}^4}{C_3}\exp\left(\frac{(1-\eta)^\eta(N\tau_0\|f-f^*\|_{L_2}^2)^\frac{{\eta(1-\eta)}}{2}}{C_42^\eta}\right)\right)\\
		\leq& N\exp\left(-\frac{(N\tau_0\mathscr{A}(N)^2)^\eta}{C_1}\right)+\exp\left(-\frac{N^2\tau_0^2\mathscr{A}(N)^4}{C_2(1+NV)}\right)\\
		&~+\exp\left(-\frac{N\tau_0^2\mathscr{A}(N)^4}{C_3}\exp\left(\frac{(1-\eta)^\eta(N\tau_0\mathscr{A}(N)^2)^\frac{{\eta(1-\eta)}}{2}}{C_42^\eta}\right)\right)\equiv \mathscr{P}_2, \numberthis\label{eq:noiseconcalpha}
	\end{align*}
	where $$V\leq \expec{\left(\xi_1(f-f^*)(X_1)\right)^2}{}+4\sum_{i\geq 0}\expec{B_i\left(\xi_1(f-f^*)(X_1)\right)^2}{},$$
	$\{B_i\}$ is some sequence such that $B_i\in[0,1]$, $\expec{B_i}{}\leq \beta(i)$ and $C_1,C_2,C_3$ are constants which depend on $c,\eta,\eta_1,\eta_2$. Observe that,
	\begin{align*}
		V\leq& \expec{\left(\xi_1(f-f^*)(X_1)\right)^2}{}+4\sum_{i\geq 0}\expec{B_i\left(\xi_1(f-f^*)(X_1)\right)^2}{}\\
		\leq & \expec{\left(\xi_1(f-f^*)(X_1)\right)^2}{}+4\sum_{i\geq 0}\sqrt{\expec{B_i^2}{}\expec{\left(\xi_1(f-f^*)(X_1)\right)^4}{}}\\
		\leq & \expec{\left(\xi_1(f-f^*)(X_1)\right)^2}{}+4\sqrt{\expec{\left(\xi_1(f-f^*)(X_1)\right)^4}{}}\sum_{i\geq 0}\sqrt{\expec{B_i}{}}\\
		\leq & \expec{\left(\xi_1(f-f^*)(X_1)\right)^2}{}+4\sqrt{\expec{\left(\xi_1(f-f^*)(X_1)\right)^4}{}}\sum_{i\geq 0}\sqrt{\beta(i)}\\
		\leq & \expec{\left(\xi_1(f-f^*)(X_1)\right)^2}{}+4\sqrt{\expec{\left(\xi_1(f-f^*)(X_1)\right)^4}{}}\sum_{i\geq 0}\exp(-ci^{\eta_1}/2)\\
		\leq & 2^\frac{2}{\eta_2}+C4^{1+\frac{2}{\eta_2}}.
\end{align*}
	Combining \eqref{eq:PnLfintermedboundalpha}, and \eqref{eq:noiseconcalpha}, with probability at least $\mathscr{P}_1-\mathscr{P}_2$, for every $f\in \cF$ that satisfies $\|f-f^*\|_{L_2}\geq \max(\rho,\mathscr{A}(N))$, we get 
	\begin{align*}
		P_NL_f\geq -2\tau_0\|f-f^*\|^2_{L_2}+\frac{\tau^2\|f-f^*\|^2_{L_2}Q_\cH(2\tau)}{4}.
	\end{align*}
	Choosing $\tau_0<\tau^2Q_\cH(2\tau)/8$, we have, $$P_NL_f>0.$$ But the empirical minimizer $\hat{f}$ satisfies $P_NL_{\hat{f}}\leq 0$. This implies, together with choosing $\mathscr{A}(N)=N^{-1/4+\iota}$, that with probability at least $\mathscr{P}=\mathscr{P}_1-\mathscr{P}_2$,
	\begin{align*}
		\|\hat{f}-f^*\|_{L_2}\leq \max(\omega_\mu(\cF-\cF,\tau Q_{\cF-\cF}(2\tau)/16),\mathscr{A}(N)),
	\end{align*}
where 
\begin{align*}
	&\mathscr{P}=1-2\exp\left(-\frac{NQ_\cH(2\tau)^3}{2(4-Q_\cH(2\tau))^2}\left(\frac{c}{\mathscr{G}(N)}\right)^\frac{1}{\eta_1}\right)-\frac{NQ_\cH(2\tau)c^\frac{1}{\eta_1}}{4\mathscr{G}(N)^\frac{1}{\eta_1}}\exp(-\mathscr{G}(N))\\
	-&N\exp\left(-\frac{(N^{\frac{1}{2}+2\iota}\tau_0)^\eta}{C_1}\right)-\exp\left(-\frac{N^{1+4\iota}\tau_0^2}{C_2(1+NV)}\right)
	-\exp\left(-\frac{N^{4\iota}\tau_0^2}{C_3}\exp\left((1-\eta)^\eta\frac{(N^{\frac{1}{2}+2\iota}\tau_0)^{\frac{\eta(1-\eta)}{2}}}{C_42^\eta}\right)\right).
\end{align*}
Choosing $\mathscr{G}(N)=N^{(1-r)\eta_1}$ for some $0<r<1$, we get,
\begin{align*}
	\mathscr{P}=&1-2\exp\left(-\frac{N^rQ_\cH(2\tau)^3c^\frac{1}{\eta_1}}{2(4-Q_\cH(2\tau))^2}\right)-\frac{N^rQ_\cH(2\tau)c^\frac{1}{\eta_1}}{4}\exp(-N^{(1-r)\eta_1})\\
	&~-N\exp\left(-\frac{(N^{\frac{1}{2}+2\iota}\tau_0)^\eta}{C_1}\right)-\exp\left(-\frac{N^{1+4\iota}\tau_0^2}{C_2(1+NV)}\right)\\
	&~-\exp\left(-\frac{N^{4\iota}\tau_0^2}{C_3}\exp\left((1-\eta)^\eta\frac{(N^{\frac{1}{2}+2\iota}\tau_0)^{\frac{\eta(1-\eta)}{2}}}{C_42^\eta}\right)\right),
\end{align*}
and 
\begin{align}
	\mu=\frac{N^rQ_\cH(2\tau)c^\frac{1}{\eta_1}}{4}.
\end{align}

We now prove \textbf{part 2}. Since Lemma~\ref{lm:th5.3eqv} only depends on Condition (a) and (b) of Assumption~\ref{as:datagenprocess}, Lemma~\ref{lm:th5.3eqv} remains unchanged in this case. To deal with the multiplier process we need a concentration result similar to Lemma~\ref{lm:rio}. So we use the concentration inequality we proved in Lemma~\ref{lm:heavytailconc}.
	When $\|f-f^*\|_{L_2}\geq \mathscr{A}(N)$, using Lemma~\ref{lm:heavytailconc} we have
	\begin{align*}
		&\pro\left(\abs{\frac{1}{N}\sum_{i=1}^N\xi_i(f-f^*)(X_i)-\expec{\xi(f-f^*)}{}}\geq \tau_0\|f-f^*\|_{L_2}^2\right)\\
		\leq& \frac{2^{\eta_2+3}}{(d_2\log N\tau_0\|f-f^*\|_{L_2}^2)^\frac{1-\eta_2}{\eta_1}}N(N\tau_0\|f-f^*\|_{L_2}^2)^{-(1+d_1(\eta_2-1))}+8N(N\tau_0\|f-f^*\|_{L_2}^2)^{-(1+d_2c')}\\
		&+2e^{-\frac{(N\tau_0\|f-f^*\|_{L_2}^2)^{2-2d_1}(d_2\log (N\tau_0\|f-f^*\|_{L_2}^2))^{1/\eta_1}}{9N}}\\
		\leq& \frac{2^{\eta_2+3}}{(d_2\log (N\tau_0\mathscr{A}(N)^2))^\frac{1-\eta_2}{\eta_1}}N(N\tau_0\mathscr{A}(N)^2)^{-(1+d_1(\eta_2-1))}+8N(N\tau_0\mathscr{A}(N)^2)^{-(1+d_2c')}\\
		&+2e^{-\frac{(N\tau_0\mathscr{A}(N)^2)^{2-2d_1}(d_2\log (N\tau_0\mathscr{A}(N)^2))^{1/\eta_1}}{9N}}\equiv \mathscr{P}_2. \numberthis\label{eq:noiseconcalphapoly}
	\end{align*}
	We will choose $d_1$ suitably to allow $\mathscr{A}(N)$ to decrease with $N$ as fast as possible while ensuring $\lim_{N\to \infty}\mathscr{P}_2\to 0$. Combining \eqref{eq:PnLfintermedboundalpha}, and \eqref{eq:noiseconcalphapoly}, with probability at least $\mathscr{P}_1-\mathscr{P}_2$, for every $f\in \cF$ that satisfies $\|f-f^*\|_{L_2}\geq \max(\rho,\mathscr{A}(N))$, we get 
	\begin{align*}
		P_NL_f\geq -2\tau_0\|f-f^*\|^2_{L_2}+\frac{\tau^2\|f-f^*\|^2_{L_2}Q_\cH(2\tau)}{4}.
	\end{align*}
	Choosing $\tau_0<\tau^2Q_\cH(2\tau)/8$, we have, $$P_NL_f>0.$$ But the empirical minimizer $\hat{f}$ satisfies $P_NL_{\hat{f}}\leq 0$. This implies, together with choosing $\mathscr{A}(N)=N^{-(1-1/\eta_2)/4+\iota}$, $d_1=1/(1+\eta_2)$, $d_2=(\eta_2-1)/(\eta_2+1)$, and $\iota<(1-1/\eta_2)/4$, that with probability at least $\mathscr{P}=\mathscr{P}_1-\mathscr{P}_2$,
	\begin{align*}
		\|\hat{f}-f^*\|_{L_2}\leq \max(\omega_\mu(\cF-\cF,\tau Q_{\cF-\cF}(2\tau)/16),\mathscr{A}(N)),
	\end{align*}
	where 
	\begin{align*}
		\mathscr{P}&=1-2\exp\left(-\frac{NQ_\cH(2\tau)^3}{2(4-Q_\cH(2\tau))^2}\left(\frac{c}{\mathscr{G}(N)}\right)^\frac{1}{\eta_1}\right)-\frac{NQ_\cH(2\tau)c^\frac{1}{\eta_1}}{4\mathscr{G}(N)^\frac{1}{\eta_1}}\exp(-\mathscr{G}(N))\\
		&-\frac{2^{\eta_2+3}\tau_0^{-\frac{2\eta_2}{1+\eta_2}}}{\left(\log \left(\tau_0N^{\frac{1}{2}+\frac{1}{2\eta_2}+2\iota}\right)/2\right)^\frac{1-\eta_2}{\eta_1}}N^{-\frac{4\iota \eta_2}{1+\eta_2}}-8\tau_0^{-\frac{2\eta_2}{1+\eta_2}}N^{-\frac{4\iota \eta_2}{1+\eta_2}}-2e^{-\frac{\tau_0^{\frac{2\eta_2}{1+\eta_2}}\left(\log \left(\tau_0N^{\frac{1}{2}+\frac{1}{2\eta_2}+2\iota}\right)/2\right)^{1/\eta_1}}{9}N^{\frac{4\iota \eta_2}{1+\eta_2}}}.
	\end{align*}
	Choosing $\mathscr{G}(N)=N^{(1-r)\eta_1}$ for some $0<r<1$, we get,
	\begin{align*}
		\mathscr{P}=&1-2\exp\left(-\frac{N^rQ_\cH(2\tau)^3c^\frac{1}{\eta_1}}{2(4-Q_\cH(2\tau))^2}\right)-\frac{N^rQ_\cH(2\tau)c^\frac{1}{\eta_1}}{4}\exp(-N^{(1-r)\eta_1})\\
		-&\frac{2^{\eta_2+3}\tau_0^{-\frac{2\eta_2}{1+\eta_2}}}{\left(\log \left(\tau_0N^{\frac{1}{2}+\frac{1}{2\eta_2}+2\iota}\right)/2\right)^\frac{1-\eta_2}{\eta_1}}N^{-\frac{4\iota \eta_2}{1+\eta_2}}-8\tau_0^{-\frac{2\eta_2}{1+\eta_2}}N^{-\frac{4\iota \eta_2}{1+\eta_2}}-2e^{-\frac{\tau_0^{\frac{2\eta_2}{1+\eta_2}}\left(\log \left(\tau_0N^{\frac{1}{2}+\frac{1}{2\eta_2}+2\iota}\right)/2\right)^{1/\eta_1}}{9}N^{\frac{4\iota \eta_2}{1+\eta_2}}},
	\end{align*}
	and 
	\begin{align}
		\mu=\frac{N^rQ_\cH(2\tau)c^\frac{1}{\eta_1}}{4}.
	\end{align}
\end{proof}
\begin{proof}[Proof of Corollary~\ref{cor:mainthm2}]
\label{rm:anhalf}
	Note that Assumption~\ref{as:normequiv} implies Condition (b) of Assumption~\ref{as:datagenprocess} as shown in Lemma 4.1 in \cite{pmlr-v35-mendelson14}. Under Assumption~\ref{as:normequiv} with $p=8$, using Cauchy-Schwarz inequality we have,
	\begin{align*}
		V\leq & \expec{\left(\xi_1(f-f^*)(X_1)\right)^2}{}+4\sqrt{\expec{\left(\xi_1(f-f^*)(X_1)\right)^4}{}}\sum_{i\geq 0}\exp(-ci^{\eta_1}/2)\\
		\leq& \sqrt{\expec{\xi_1^4}{}}\sqrt{\expec{\left((f-f^*)(X_1)\right)^4}{}}+4C\sqrt{\sqrt{\expec{\xi_1^8}{}}\sqrt{\expec{\left((f-f^*)(X_1)\right)^8}{}}}\\
		\leq & M_1^2\|f-f^*\|_{L_2}^2\left(\sqrt{\expec{\xi_1^4}{}}+4C\left(\expec{\xi_1^8}{}\right)^\frac{1}{4}\right)\\
		\leq& M_2^2\|f-f^*\|_{L_2}^2,
	\end{align*}
for some constant $M_2$. Then, from \eqref{eq:noiseconcalpha} we have,
	\begin{align*}
	&\pro\left(\abs{\frac{1}{N}\sum_{i=1}^N\xi_i(f-f^*)(X_i)-\expec{\xi(f-f^*)}{}}\geq \tau_0\|f-f^*\|_{L_2}^2\right)\\
	\leq& N\exp\left(-\frac{(N\tau_0\|f-f^*\|_{L_2}^2)^\eta}{C_1}\right)+\exp\left(-\frac{N^2\tau_0^2\|f-f^*\|_{L_2}^4}{C_2(1+NV)}\right)\\
	&~+\exp\left(-\frac{N\tau_0^2\|f-f^*\|_{L_2}^4}{C_3}\exp\left(\frac{(N\tau_0\|f-f^*\|_{L_2}^2)^{\eta(1-\eta)}}{C_4(\log (N\tau_0\|f-f^*\|_{L_2}^2))^\eta}\right)\right)\\
	\leq& N\exp\left(-\frac{(N\tau_0\|f-f^*\|_{L_2}^2)^\eta}{C_1}\right)+\exp\left(-\frac{N^2\tau_0^2\|f-f^*\|_{L_2}^4}{C_2(1+NM_2^2\|f-f^*\|_{L_2}^2)}\right)\\
	&~+\exp\left(-\frac{N\tau_0^2\|f-f^*\|_{L_2}^4}{C_3}\exp\left(\frac{(N\tau_0\|f-f^*\|_{L_2}^2)^{\eta(1-\eta)}}{C_4(\log (N\tau_0\|f-f^*\|_{L_2}^2))^\eta}\right)\right).
	\end{align*}
If $N\|f-f^*\|_{L_2}^2\geq N\mathscr{A}(N)^2\geq \max(1/M_2^2,1/\tau_0)$, then
\begin{align*}
	&\pro\left(\abs{\frac{1}{N}\sum_{i=1}^N\xi_i(f-f^*)(X_i)-\expec{\xi(f-f^*)}{}}\geq \tau_0\|f-f^*\|_{L_2}^2\right)\\
	\leq & N\exp\left(-\frac{(N\tau_0\|f-f^*\|_{L_2}^2)^\eta}{C_1}\right)+\exp\left(-\frac{N\tau_0^2\|f-f^*\|_{L_2}^2}{2C_2M_2^2}\right)\\
	&+\exp\left(-\frac{N\tau_0^2\|f-f^*\|_{L_2}^4}{C_3}\exp\left(\frac{(1-\eta)^\eta(N\tau_0\mathscr{A}(N)^2)^\frac{{\eta(1-\eta)}}{2}}{C_42^\eta}\right)\right),
\end{align*}
and the second term dominates the third term in the above expression. Now if choose $\mathscr{A}(N)=N^{-1/2+\iota}$, then with probability at least 
\begin{align*}
	\mathscr{P}=1-\exp\left(-\frac{N^{2\iota}\tau_0^2}{2C_2M_2^2}\right)				-\exp\left(-\frac{N^{4\iota-1}\tau_0^2}{C_3}\exp\left(\frac{(1-\eta)^\eta(N^{2\iota}\tau_0)^\frac{{\eta(1-\eta)}}{2}}{C_42^\eta}\right)\right),
	\end{align*}
		we get
\begin{align*}
	\|\hat{f}-f^*\|_{L_2}\leq \max\left(N^{-1/2+\iota},\omega_\mu(\cF-\cF,\tau Q_{\cF-\cF}(2\tau)/16)\right) .
\end{align*}
\end{proof}
\subsection{Proofs for convex loss}
Recall the decomposition \eqref{eq:decomp}
\begin{align*}
	P_NL_f\geq\frac{1}{16N}\sum_{i=1}^N \ell''(\widetilde{\xi}_i)(f-f^*)^2(X_i)+\frac{1}{N}\sum_{i=1}^N \ell'(\xi_i)(f-f^*)(X_i).
	\vspace{-.05in}
\end{align*}
Since $\cF$ is convex, we also have
\begin{align*}
    \expec{\ell'(\xi)(f-f^*)(X)}{}\geq 0.
\end{align*}
Then,
\begin{align}\label{eq:decompwithexpec}
	P_NL_f\geq\frac{1}{16N}\sum_{i=1}^N \ell''(\widetilde{\xi}_i)(f-f^*)^2(X_i)+\frac{1}{N}\sum_{i=1}^N (\ell'(\xi_i)(f-f^*)(X_i)-\expec{\ell'(\xi_i)(f-f^*)(X_i)}{}).
	\vspace{-.05in}
\end{align}
Now our goal is to establish a lower bound (Proposition~\ref{prop:theorem4.7equiv}) on the first term of the RHS of \eqref{eq:decompwithexpec}, and a two-sided bound (\eqref{eq:noiseconcalphacon} and \eqref{eq:noiseconcalphapolycon}) on the second term when $\|f-f^*\|_{L_2}$ is large. Combining these bounds we will show that if $\|f-f^*\|_{L_2}$ is large then $P_NL_f>0$ which implies $f$ cannot be a minimizer of empirical risk because for the minimizer $\hat{f}$ we have $P_NL_{\hat{f}}\leq 0$. 
Let $\rho(t_1,t_2)\coloneqq\inf\{l''(x):x\in[t_1,t_2], 0\leq t_1< t_2\}$. First we prove the following extension of bounded difference inequality to the $\beta$-mixing sequence which we will use frequently in our proofs.  

\input{convex}

\section{Proofs of Section~\ref{sec:ex1}}
\begin{lemma}[Lemma 6.4 of \cite{pmlr-v35-mendelson14}]\label{lm:mend6pt4}
	If $W=(w_i)_{i=1}^d$ is a random vector on $\mathbb{R}^d$, then for every integer $1\leq k\leq d$, 
	\begin{align*}
		\expec{\sup_{t\in\sqrt{k}B_1^d\cap B_2^d}\left\langle W,t\right\rangle}{}\leq 2\expec{\left(\sum_{i=1}^k{w_i^*}^2\right)^\frac{1}{2}}{},
	\end{align*}
	where $(w_i^*)_{i=1}^d$ is a monotone non-increasing reaarangement of $(|w_i|)_{i=1}^d$.
\end{lemma}
\begin{lemma}\label{lm:expecwstarsum}
	Let $w_1,w_2,\cdots,w_d$ are independent copies of a mean-zero, variance 1 random variable $w\sim \textsc{sw}(\eta)$. Then for all $p\geq 1\wedge \eta$, $\|w\|_{L_p}\leq K_1p^\frac{1}{\eta}$ for some constant $K_1>0$. Then for every $1\leq k\leq d$,
	\begin{align*}
		\expec{\left(\sum_{i=1}^k{w_i^*}^2\right)^\frac{1}{2}}{}\leq  \sqrt{2k}K_1\left(\log (ed)\right)^{1/\eta}.
	\end{align*}
\end{lemma}
\begin{proof}[Proof of Lemma~\ref{lm:expecwstarsum}]
	For $1\leq j\leq d$, and $p\geq 2$,
	\begin{align*}
		\pro(w_j^*\geq t)\leq \binom{d}{j}\pro^j(|w|>t)\leq  \binom{d}{j}\left(\frac{\|z\|_{L_p}}{t}\right)^{jp}.
	\end{align*}
Setting $t=uK_1 \left(\log (ed/j)\right)^{1/\eta}$ and $p=\log(ed/j)$, we get
\begin{align*}
	\pro\left(w_j^*\geq uK_3\right)\leq \left(\frac{1}{u} \right)^{j\log(ed/j)},      \numberthis\label{eq:tailboundsubwei}
\end{align*}
where $K_3=K_1 \left(\log (ed/j)\right)^{1/\eta}$.
Using \eqref{eq:tailboundsubwei} we will bound $\expec{{w_j^*}^2}{}$. For some $v$,
\begin{align*}
	\expec{{w_j^*}^2}{}=&\int_{0}^{\infty}\pro({w_j^*}^2>u)du\\
	=&\int_{0}^{v}\pro({w_j^*}^2>u)du+\int_{v}^{\infty}\pro({w_j^*}^2>u)du\\
	\leq & v+\int_{0}^{\infty}\pro({w_j^*}^2>u+v)du\\
	\leq &v+\int_{0}^{\infty}\left(\frac{K_3}{\sqrt{u+v}} \right)^{j\log(ed/j)}du\\
	=&v-K_5\left[\frac{(u+v)^{1-j\log(ed/j)/2}}{j\log(ed/j)/2-1}\right]_0^\infty \quad \quad [\text{where } K_5=K_3^{j\log(ed/j)}]\\
	=&v+K_5\left[\frac{v^{1-j\log(ed/j)/2}}{j\log(ed/j)/2-1}\right].
\end{align*}
To minimize the upper bound on $\expec{{w_j^*}^2}{}$ we choose $$v=K_5^\frac{2}{j\log(ed/j)}=K_3^2=K_1^2 \left(\log (ed/j)\right)^{2/\eta}.$$ and get 
\begin{align*}
	\expec{{w_j^*}^2}{}\leq2K_1^2\left(\log (ed/j)\right)^{2/\eta}.
\end{align*}
For any $1\leq k\leq d$, using Jensen's inequality,
\begin{align*}
	\expec{\left(\sum_{i=1}^k{w_i^*}^2\right)^\frac{1}{2}}{}\leq \left(\sum_{i=1}^k\expec{{w_i^*}^2}{}\right)^\frac{1}{2}\leq \left(\sum_{i=1}^k2K_1^2\left(\log (ed/i)\right)^{2/\eta}\right)^\frac{1}{2}\leq \sqrt{2k}K_1\left(\log (ed)\right)^{1/\eta}.
\end{align*}
\end{proof}
\begin{proof}[Proof of Proposition~\ref{prop:subweirate}]\label{p:4.1}
In order to provide a bound on $\|\hat f -f^* \|_{L_2}$, we need to compute the order of $\omega_\mu(\cF_R-\cF_R,\tau Q_{\cH_R}(2\tau)/16)$. Based on Lemma~\ref{lm:mend6pt4} and~\ref{lm:expecwstarsum} it is easy to see that, in a similar way to \cite{pmlr-v35-mendelson14},
\[
\expec{\sup_{f\in \cF_R\cap s\cD_{f^*}}\abs{\frac{1}{\sqrt{N}}\sum_{i=1}^N\epsilon_i(f-f^*)(X_i)}}{}\leq 
\begin{cases}
	c_1K_1R\left(\log ed\right)^\frac{1}{\eta} &\quad (R/s)^2> d/4,\\
	c_2K_1s\sqrt{d}&\quad u\leq (R/s)^2\leq d/4,\\
\end{cases}
\]
where $c_1,c_2$ are constants. Hence, following similar steps as in the proof of \cite[Lemma 4.6]{pmlr-v35-mendelson14}, we have
\[
\omega_\mu(\cF_R-\cF_R,\tau Q_{\cH_R}(2\tau)/16)\leq 
\begin{cases}
	\frac{c_3R}{\sqrt{\mu}}\log\left(ed\right)^\frac{1}{\eta} & \text{if $\mu\leq c_1d,$} \\
	0 & \text{if $\mu> c_1d$}. 
\end{cases} \numberthis\label{eq:omegaboundwei}
\]
From \eqref{eq:omegaboundwei}, choosing $r=1-2\iota$ by Theorem~\ref{th:mainthm2}, for sufficiently large $N$, we have with probability at least 
\begin{align*}
	1-\tilde{C}_1\frac{N^{1-2\iota}Q_\cH(2\tau)c^\frac{1}{\eta_1}}{4}\exp(-N^{2\iota\eta_1})-\tilde{C}_2N\exp\left(-\frac{(N^{2\iota}\tau_0)^\eta}{M_1}\right),
\end{align*}
we have
\begin{align*}
	\|\hat{f}-f^*\|_{L_2}\leq \max\left(\frac{2c_3R\log (ed)^\frac{1}{\eta}}{\sqrt{Q_\cH(2\tau)}c^\frac{1}{2\eta_1}}N^{-\frac{1}{2}+\iota},N^{-\frac{1}{2}+\iota}\right).
\end{align*}
\end{proof}
\section{Proofs of Section~\ref{sec:examplepoly}}
\begin{lemma}\label{lm:w1starconc}
	Let $\{X'_i\}_{i=1}^\mu$ be an $\iid$ sample with independent coordinates $X'_{i,j}\sim L(\eta_3,d_j)$ and let $w$ be a random vector with coordinates
	\begin{align*}
		w_j=\frac{1}{\sqrt{\mu}}\sum_{i=1}^\mu X'_{i,j} \qquad j=1,2,\cdots,d. \numberthis\label{eq:wjdefpoly}
	\end{align*}
Then we have
	\begin{align*}
		\mathbb{P}\left(\abs{w_j}\geq t\right)\leq C_3\left(d_j^{\eta_3-2p-1}\mu^{1-\frac{\eta_3}{2}}t^{\eta_3-2p}+d_j^{-2}t^{-p}\right), \numberthis\label{eq:w1starconc}
	\end{align*}
	for some constant $C_3>0$.
\end{lemma}
\begin{proof}[Proof of Lemma~\ref{lm:w1starconc}]
Using the symmetry of the distribution of $w_j$ we can write,
\begin{align}
	\mathbb{P}\left(|w_j|\geq t\right)\leq 2\mathbb{P}\left(w_j\geq t\right).
\end{align}
Setting $p=\eta_3-0.5\iota$, using Theorem 2.1 of \cite{chesneau2007tail} we get for any $t>0$
\begin{align*}
	\mathbb{P}\left(w_j\geq t\right)\leq C_pt^{-p}\max\left(r_{\mu,p}(t),\left(r_{\mu,2}(t)\right)^\frac{p}{2}\right)+\exp\left(-\frac{d_j^2t^2}{16\sigma_{X,2}^2}\right), \numberthis\label{eq:pwjconcintermed}
\end{align*}
where $$C_{1,p}=2^{2p+1}\max\left(p^p,p^{p/2+1}e^p\int_{0}^\infty x^{p/2-1}(1-x)^{-p}dx\right),$$ and for any $k\in\{p,2\}$, $$r_{\mu,k}(t)=\sum_{i=1}^\mu\expec{\abs{\frac{X'_{i,j}}{\sqrt{\mu}}}^k\mathbbm{1}\left(\abs{\frac{X'_{i,j}}{\sqrt{\mu}}}\geq \frac{3\sigma_{X,2}^2}{td_j^2}\right)}{}.$$ 

Now, 
\begin{align*}
	\expec{\abs{\frac{X'_{i,j}}{\sqrt{\mu}}}^p\mathbbm{1}\left(\abs{\frac{X'_{i,j}}{\sqrt{\mu}}}\geq \frac{3\sigma_{X,2}^2}{td_j^2}\right)}{}
	=&\int_{-\infty}^{\infty}\abs{\frac{x}{\sqrt{\mu}}}^p\mathbbm{1}\left(\abs{\frac{x}{\sqrt{\mu}}}\geq \frac{3\sigma_{X,2}^2}{td_j^2}\right)\frac{\eta_3(|x|d_j)^{\eta_3-1}}{2\left(1+(|x|d_j)^{\eta_3}\right)^2}dx\\
	=&\int_{\frac{3\sigma_{X,2}^2\sqrt{\mu}}{td_j^2}}^{\infty}\left(\frac{x}{\sqrt{\mu}}\right)^p\frac{\eta_3(xd_j)^{\eta_3-1}}{\left(1+(xd_j)^{\eta_3}\right)^2}dx\\
	\leq &\frac{\eta_3}{d_j^{2p-\eta_3+1}\mu^\frac{\eta_3}{2}(\eta_3-p)}\left(\frac{3\sigma_{X,2}^2}{t}\right)^{p-\eta_3}\\
	\leq & C_2d_j^{\eta_3-2p-1}\mu^{-\frac{\eta_3}{2}}t^{\eta_3-p}, \numberthis\label{eq:rmupbound}
\end{align*}
where $C_2$ is a constant which depends on $\eta_3$ and $p$. Then 
\begin{align*}
	r_{\mu,p}(t)\leq C_2d_j^{\eta_3-2p-1}\mu^{1-\frac{\eta_3}{2}}t^{\eta_3-p}.
\end{align*}
The term $r_{\mu,2}(t)$ can similarly be bounded as follows:
\begin{align*}
	r_{\mu,2}(t)=\sum_{i=1}^\mu\expec{\abs{\frac{X'_{i,j}}{\sqrt{\mu}}}^2\mathbbm{1}\left(\abs{\frac{X'_{i,j}}{\sqrt{\mu}}}\geq \frac{3\sigma_{X,2}}{td_j^2}\right)}{}\leq \frac{\sigma_{X,2}^2}{d_j^2}. \numberthis\label{eq:rmu2bound}
\end{align*}
Using \eqref{eq:rmupbound}, and \eqref{eq:rmu2bound}, from \eqref{eq:pwjconcintermed} we get 
\begin{align*}
	\mathbb{P}\left(\abs{w_j}\geq t\right)\leq C_3\left(d_j^{\eta_3-2p-1}\mu^{1-\frac{\eta_3}{2}}t^{\eta_3-2p}+d_j^{-2}t^{-p}\right),
\end{align*}
for some constant $C_3>0$.
\end{proof}
 
\begin{lemma}\label{lm:expecwstarsumpoly}
 Let $w_1,w_2,\cdots,w_d$ are independent copies of a random variable such that \eqref{eq:w1starconc} is true for all $j=1,2,\cdots,d$, i.e., 
	\begin{align*}
		\mathbb{P}\left(\abs{w_j}\geq t\right)\leq C_3\left(d^{\eta_3-2p-1}\mu^{1-\frac{\eta_3}{2}}t^{\eta_3-2p}+d^{-2}t^{-p}\right) \quad j=1,2,\cdots,d,
	\end{align*}
	for $\eta_3>2+2\iota$, and $p=\eta_3-0.5\iota$. 	Let $\{w_j^*\}_{j=1}^d$ be the non-increasing arrangement of $\{|w_j|\}_{j=1}^d$. Then for every $1\leq k\leq d$,
	\begin{align*}
		\expec{\left(\sum_{i=1}^k{w_i^*}^2\right)^\frac{1}{2}}{}
		\leq  C_6\sqrt{k}\left(d^{\eta_3/(2p)-1/2+1/p}d_1^{\eta_3/2-p-\eta_3/p+3/2-2/p}\mu^{1/2-\eta_3/4}+d^{1/p}d_1^{-2/p}\right) ,
	\end{align*}
	for some constant $C_6>0$ which depends on $\eta_3$ and $p$.
\end{lemma}
\begin{proof}[Proof of Lemma~\ref{lm:expecwstarsumpoly}] First note that we have
	\begin{align*}
		\mathbb{P}\left({w_1^*}^2\geq t\right)
		=\mathbb{P}\left({w_1^*}\geq \sqrt{t}\right)
		\leq  \sum_{j=1}^d\mathbb{P}\left(\abs{w_j}\geq \sqrt{t}\right)
		\leq  C_3\left(dd_1^{\eta_3-2p-1}\mu^{1-\eta_3/2}t^{\eta_3/2-p}+dd_1^{-2}t^{-p/2}\right).
	\end{align*}
Now, using \eqref{eq:w1starconc}, for any $v>0$ (to be chosen later), we have
\begin{align*}
	\expec{{w_1^*}^2}{}=&\int_{0}^\infty\mathbb{P}({w_1^*}^2\geq t) dt\\
	\leq &v+ \int_{0}^\infty\mathbb{P}({w_1^*}^2\geq t+v) dt\\
	\leq & v+\int_{0}^\infty C_3\left(dd_1^{\eta_3-2p-1}\mu^{1-\eta_3/2}(t+v)^{\eta_3/2-p}+dd_1^{-2}(t+v)^{-p/2}\right) dt\\
	\leq & v+C_4\left(dd_1^{\eta_3-2p-1}\mu^{1-\eta_3/2}v^{\eta_3/2-p+1}+dd_1^{-2}v^{1-p/2}\right), 
\end{align*}
where $C_4=C_3\max(1/(p-1-\eta_3/2),1/(p/2-1))$. Choosing $v=d^{2/p}d_1^{-4/p}$, we get
\begin{align*}
	\expec{{w_1^*}^2}{}\leq C_5\left(d^{\eta_3/p-1+2/p}d_1^{\eta_3-2p-2\eta_3/p+3-4/p}\mu^{1-\eta_3/2}+d^{2/p}d_1^{-4/p}\right) ,
\end{align*}
where $C_5=C_4+1$.
Using Jensen's inequality, we have
\begin{align*}
	\expec{\left(\sum_{j=1}^k {w_j^*}^2\right)^\frac{1}{2}}{}&\leq \left(\sum_{j=1}^k \expec{{w_j^*}^2}{}\right)^\frac{1}{2}\leq \left(k \expec{{w_1^*}^2}{}\right)^\frac{1}{2}\\
	&\leq  C_6\sqrt{k}\left(d^{\eta_3/(2p)-1/2+1/p}d_1^{\eta_3/2-p-\eta_3/p+3/2-2/p}\mu^{1/2-\eta_3/4}+d^{1/p}d_1^{-2/p}\right) ,
\end{align*}
where $C_6=\sqrt{C_5}$, thereby completing the proof.
\end{proof}
\vspace{0.21in}

\begin{proof}[Proof of Proposition~\ref{prop:paretorate}]\label{p:prop4.2}
We start by obtaining a bound on the term  $\omega_\mu(\cF_R-\cF_R,\tau Q_{\cH_R}(2\tau)/16)$. Let $w$ be a random vector with coordinates
	\begin{align*}
		w_j=\frac{1}{\sqrt{\mu}}\sum_{i=1}^\mu X'_{i,j} \qquad j=1,2,\cdots,d. \numberthis\label{eq:wjdefpoly}
	\end{align*}
Let $\{w_j^*\}_{j=1}^d$ be the non-increasing arrangement of $\{|w_j|\}_{j=1}^d$. Then, we have
	\begin{align*}
		&\expec{\sup_{f\in\cF_R\cap s\mathcal{D}_{f^*}}\abs{\frac{1}{\sqrt{\mu}}\sum_{i=1}^\mu\epsilon_i(f-f^*)(X'_i)}}{}\leq \expec{\sup_{t\in B_1^d(2R)\cap B_2^d(s)}\left\langle\frac{1}{\sqrt{\mu}}\sum_{i=1}^\mu X'_i,t\right\rangle}{}\\
		=&\expec{\sup_{t\in B_1^d(2R)\cap B_2^d(s)}\left\langle w,t\right\rangle}{}
		=s\expec{\sup_{t\in B_1^d(2R/s)\cap B_2^d(1)}w^\top t}{}
		\leq 2s\expec{\left(\sum_{j=1}^{(2R/s)^2} {w_j^*}^2\right)^\frac{1}{2}}{}.
	\end{align*}

If $(2R/s)^2<d$, using Lemma~\ref{lm:expecwstarsumpoly} we get
\begin{align*}
	\expec{\sup_{t\in B_1^d(2R)\cap B_2^d(s)}w^\top t}{}\leq& 4C_6R\left(d^{\eta_3/(2p)-1/2+1/p}d_1^{\eta_3/2-p-\eta_3/p+3/2-2/p}\mu^{1/2-\eta_3/4}+d^{1/p}d_1^{-2/p}\right) \\ \leq& C_7Rd^{1/p+\iota/8},
\end{align*}
when $d_1\geq C_6'$ for some constants $C_6',C_7>0$.\\

If $(2R/s)^2\geq d$,
\begin{align*}
	\expec{\sup_{t\in B_1^d(2R)\cap B_2^d(s)}w^\top t}{}\leq 2s\sigma_{X,2}\sqrt{d}. 
\end{align*}
So when $(2R/s)^2\geq d$, 
\begin{align*}
	\expec{\sup_{f\in\cF_R\cap s\mathcal{D}_{f^*}}\abs{\frac{1}{\mu}\sum_{i=1}^\mu\epsilon_i(f-f^*)(X'_i)}}{}\leq \gamma s,
\end{align*} for all $s>0$. When $\mu\leq C_8d^{1+2/p+\iota/4}$, we have $(2R/s)\leq \sqrt{d}$ for
\begin{align*}
	s\geq \frac{C_7Rd^{1/p+\iota/8}}{\gamma\sqrt{\mu}}.
\end{align*}
When $\mu> C_8d^{1+2/p+\iota/4}$, we have $(2R/s)\geq \sqrt{d}$ for
\begin{align*}
	s\leq \frac{C_7Rd^{1/p+\iota/8}}{\gamma\sqrt{\mu}}.
\end{align*}
Combining the above facts, and choosing $\mu=\frac{N^rQ_\cH(2\tau)c^\frac{1}{\eta_1}}{4}$, and $r=1-2\iota$ we get
\begin{align*}
\omega_\mu(\cF_R-\cF_R,\tau Q_{\cH_R}(2\tau)/16)\leq 
\begin{cases}
	\frac{C_9R}{\tau Q_{\cH_R}(2\tau)^{3/2}}d^{1/p+\iota/8}N^{-1/2+\iota} & \text{if $\mu\leq C_8d^{1+2/p}$}, \\
	0 & \text{if $\mu> C_8d^{1+2/p}$},
\end{cases} 
\end{align*}
where $C_9=32C_7/c^{1/(2\eta_1)}$. Then using part 2 of Theorem~\ref{th:mainthm2}, we get, 
\begin{align*}
	\|\hat{f}-f^*\|_{L_2} \leq \max\left\{N^{-\frac{1}{4}\left(1-\frac{1}{\eta_2}\right)+\iota},\frac{C_9R}{\tau Q_{\cH_R}(2\tau)^{3/2}}d^{1/p+\iota/8}N^{-1/2+\iota}\right\} ,
\end{align*}
with probability given by at least \eqref{eq:thm10probpoly}.
\end{proof}
\section{Proofs of Section~\ref{sec:exhuber}}
\begin{proof}[Proof of Proposition~\ref{prop:huberrate}]\label{p:huber}
Since we assumed $X$ to be Gaussian, $\cF_R$ is a $L_g$-subGaussian function class for some constant $L_g>0$. So as shown in Section 6.5.2, we have,
\begin{align*}
    \omega_Q(\cF-\cF,N,\zeta_1,\zeta_2)\leq 
    \begin{cases}
        \frac{c_3(L_g)R}{\sqrt{N}}\sqrt{\log(ed/N)} & \text{if} N\leq c_1(L_g)d,\\
        \frac{c_4(L_g)R}{\sqrt{d}} & \text{if} c_1(L_g)d< N\leq c_2(L_g)d,\\
        0 & \text{if} N> c_2(L_g)d,\\
    \end{cases}
\end{align*}
where $c_i(L_g), i=1,2,3,4$ are constants dependent on only $L_g$. Then using Corollary~\ref{cor:mainthm2con}, we have
	\begin{align*}
			\|\hat{f}-f^*\|_{L_2}\leq \max\left(N^{-\frac{1}{2}+\iota},2\frac{c_3(L_g)R}{\sqrt{N}}\sqrt{\log(ed/N)}\right),
		\end{align*}
with probability at least (for some constants $c_9,c_{10},\tilde{C}_2>0$)
 		\begin{align*}
				1-c_9\epsilon^{1-\frac{1}{\eta_1}} N^{\eta_1/(1+\eta_1)}e^{-c_{10}\epsilon^{1+\frac{1}{\eta_1}}N^{\eta_1/(1+\eta_1)}}-\tilde{C}_2N\exp\left(-(N^{2\iota}\tau_0)^\eta/M_1\right).
		\end{align*}
\end{proof}
\section{A Note on Condition (c) in Assumption~\ref{as:datagenprocess} and $\alpha_N^*(\gamma,\delta)$ in~\cite{pmlr-v35-mendelson14}}\label{sec:reltoempiricalprocess}
In this section, we discuss the relationship between Condition (c)-(i) of our Assumption~\ref{as:datagenprocess} and the multiplier process based assumption in~\cite[Equation 2.2 and $\alpha_N^*(\gamma,\delta)$]{pmlr-v35-mendelson14}. For simplicity, we consider the following simple model. Let $\{X_i\}_{i\in\mathbb{Z^+}}$ is an $\iid$ sequence of symmetric, zero-mean, random vectors. Let $\{Y_i\}_{i\in\mathbb{Z^+}}$, $Y_i\in\mathbb{R}$ denote the sequence given by $Y_i={\theta^*}^\top X_i+\xi_i$, where $\theta^*\in B_1^1(R)$ and $\{\xi_i\}_{i=1}^N$ is an $\iid$ sequence and independent of $X_i$, $\forall i$, and $\xi_i\sim N(0,\sigma_1^2)$. The function class $\mathcal{F}$ we consider is $\mathcal{F}\coloneqq \cF_R=\left\lbrace\langle\theta,\cdot\rangle:\theta\in B_1^d(R)\right\rbrace$. Now let us assume $\frac{1}{N}\sum_{i=1}^NX_i\xi_i$ is heavy-tailed random vector, with the tail lower bounded by $N\exp(-M(Nt))$, for some positive increasing function of $t$, $M(t)$, i.e., $\mathbb{P}\left(\abs{N^{-1}\sum_{i=1}^NX_i\xi_i}>t\right)\geq M_3N\exp(-M(Nt))$
for some $M_3>0$. Specifically setting $M(t)=t^\eta,\eta>0$, and $M(t)=\eta_2\log t, \eta_2>2$ one recovers \eqref{eq:rio} and \eqref{eq:heavytailconc}. Now, recall from~\cite{pmlr-v35-mendelson14} that,
\begin{align}
	\alpha_N^*(\gamma,\delta):=\inf\left\lbrace s>0:\pro\left(	\sup_{\theta\in B_1^1(2R)\cap B_2^1(s)}\abs{\frac{1}{N}\sum_{i=1}^N\xi_iX_i\theta}\leq\gamma s^2\right)\geq 1-\delta\right\rbrace. \numberthis\label{eq:alphadef}
\end{align}
Note that, for $s>0$, $$\sup_{\theta\in B_1^1 (2R)\cap B_2^1(s)}\abs{N^{-1}\sum_{i=1}^N\xi_iX_i\theta}= \abs{N^{-1}\sum_{i=1}^N\xi_iX_i}\min(2R,s).$$
We also have,
$$\mathbb{P}\left(\abs{\frac{1}{N}\sum_{i=1}^N\xi_iX_i}\leq\gamma s^2/\min(2R,s)\right)\leq 1-M_3N\exp\left(-M\left(\gamma N s^2/\min(2R,s)\right)\right).$$
Then, from \eqref{eq:alphadef}, when $s=\alpha_N^*(\gamma,\delta)$, 
\begin{align*}
    \delta\geq NM_3\exp\left(-M\left(\gamma N {\alpha_N^*(\gamma,\delta)}^2/\min(2R,\alpha_N^*(\gamma,\delta))\right)\right).
\end{align*}
Hence, if we want a non-trivial bound on the generalization error, we need ${\alpha_N^*(\gamma,\delta)}^2\leq N^{-m_0}$ for some $m_0>0$. Set $2R>N^{-m_0/2}$. When $M(t)\sim t^{\gamma_2}, \gamma_2>0$, $\frac{1}{N}\sum_{i=1}^N\xi_iX_i$ has a sub-weibull tail. If it has a polynomially decaying tail, i.e., $M(t)=M_4\log t$ for some constant $M_4>0$, then
$$\delta\geq NM_3\exp\left(-M_4\log \left(\gamma N^{1-m_0/2}\right)\right) =M_3\gamma^{-M_4}N^{1-(1-\frac{m_0}{2})M_4}.$$
This implies that if $\frac{1}{N}\sum_{i=1}^N\xi_iX_i$ has a polynomially decaying tail, one gets a polynomial probability statement on the rate using complexity measure $\alpha_N^*(\gamma,\delta)$. Note that, since we are considering $\iid$ setting, choosing $m_0<1$ would allow $\alpha_N^*(\gamma,\delta)$ to be of the order of $N^{-1/2+\iota}$ where $\iota>0$ is a small number. Recall that the rates we obtain in Theorem~\ref{th:mainthm2}, and \ref{th:mainthm2con} are for $\beta$-mixing case. Indeed the worse rates are due to the presence of the third terms on the RHS of \eqref{eq:rio}, and \eqref{eq:heavytailconc} -- one needs to choose $\mathcal{A}(N)$ (used in the proofs of Theorem~\ref{th:mainthm2},and \ref{th:mainthm2con}) suitably so that the third terms on the RHS of \eqref{eq:rio}, and \eqref{eq:heavytailconc} decay to $0$ as $N\to\infty$.

%% file: convex.tex
\begin{lemma}[Bounded difference inequality for strictly stationary $\beta$-mixing sequence]\label{lm:boundeddiff}
Let $\{U_i\}_{i=1}^N$ be a sample from a strictly stationary $\beta$-mixing sequence, $|U_i|\leq M$, and $\expec{U_i}{}=U^*$. Let $N> a,b,\mu>0$ be such that $(a+b)\mu=N$. Then with probability at least $1-\exp\left(-\frac{(t-2b\mu)^2}{2\mu a^2M^2}\right)-2M(\mu-1)\beta(b)$, we have $\forall t> 2b\mu$,
\begin{align*}
    \sum_{i=1}^NU_i\leq NU^*+t.
\end{align*}
\end{lemma}
\begin{proof}
Consider the partition as in \eqref{eq:partition}. Then, using Corollary 2.7 of \cite{yu1994rates}, we get $\forall t> 2b\mu$, 
\begin{align*}
    &\mathbb{P}\left(\sum_{i=1}^N (U_i-U^*)\geq t\right)\\
    \leq&\mathbb{P}\left(\sum_{i=1}^\mu\sum_{j=1}^a (U_i-U^*)\geq t-2b\mu\right)\\
    \leq &\mathbb{P}\left(\sum_{i=1}^\mu\sum_{j=1}^a (\tilde{U}_{(a+b)(i-1)+j}-U^*)\geq t-2b\mu\right)+2M(\mu-1)\beta(b).
\end{align*}
where $\sum_{j=1}^a (\tilde{U}_{(a+b)(i-1)+j}-U^*)$ is an $\iid$ sequence for $i=1,2,\cdots,\mu$. Using bounded difference inequality,
\begin{align*}
    \mathbb{P}\left(\sum_{i=1}^N (U_i-U^*)\geq t\right)
    \leq  \exp\left(-\frac{(t-2b\mu)^2}{2\mu a^2M^2}\right)+2M(\mu-1)\beta(b).
\end{align*}
So with probability at least $1-\exp\left(-\frac{(t-2b\mu)^2}{2\mu a^2M^2}\right)-2M(\mu-1)\beta(b)$,
\begin{align*}
    \sum_{i=1}^NU_i\leq NU^*+t.
\end{align*}
\end{proof}
Lemma~\ref{lm:Xiupperbound}-\ref{lm:Xiuplowbound} are needed to prove Lemma~\ref{lm:thm4.3equiv} which is the main result needed to prove Proposition~\ref{prop:theorem4.7equiv}. 
\begin{lemma}\label{lm:Xiupperbound}
Let $X_i, i=1,2,\cdots,N$ be a sample from a sequence for which condition (a) of Assumption~\ref{as:datagenprocess} is true. For every $0< Q_\cH(2\tau) <1$, we have that with probability at least $1-c_1 Q_\cH(2\tau) ^{1-\frac{1}{\eta_1}} N^{\eta_1/(1+\eta_1)}e^{-c_2 Q_\cH(2\tau) ^{1+\frac{1}{\eta_1}}N^{\eta_1/(1+\eta_1)}}$, for some constants $c_1,c_2>0$, there is a subset $S\subset\{1,2,\cdots,N\}$ such that $\abs{S}\geq N(1- Q_\cH(2\tau) )$, and $\forall i\in S$,
\begin{align*}
    \abs{X_i}\leq \frac{2\|X_i\|_{L_2}}{\sqrt{ Q_\cH(2\tau) }}.\numberthis\label{eq:xinotverylarge}
\end{align*}
\end{lemma}
\begin{proof}
Let $\zeta_i=\mathbbm{1}\left(\abs{X_i}\geq \frac{2\|X_i\|_{L_2}}{\sqrt{ Q_\cH(2\tau) }}\right)$. 
Then, by Markov's inequality,  $$\expec{\zeta_i}{}=\mathbb{P}\left(\abs{X_i}\geq \frac{2\|X_i\|_{L_2}}{\sqrt{ Q_\cH(2\tau) }}\right)\leq  Q_\cH(2\tau) /4.$$
Then using Lemma~\ref{lm:boundeddiff}, we have with probability at least $1-\exp\left(-\frac{(t-2b\mu)^2}{2\mu a^2}\right)-2(\mu-1)\beta(b)$,
\begin{align*}
    \sum_{i=1}^N\zeta_i\leq \frac{N Q_\cH(2\tau) }{4}+t.
\end{align*}
Now, setting
		\begin{align}\label{eq:tabmucommonchoice}
			t=\frac{3N Q_\cH(2\tau) }{4} \quad a=\frac{(4- Q_\cH(2\tau) )N^{\frac{1}{1+\eta_1}}}{c^\frac{1}{\eta_1} Q_\cH(2\tau) ^{1-\frac{1}{\eta_1}}} \quad b=\frac{ Q_\cH(2\tau) ^\frac{1}{\eta_1} N^{\frac{1}{1+\eta_1}}}{c^\frac{1}{\eta_1}} \quad \mu=\frac{N^{\frac{\eta_1}{1+\eta_1}}c^\frac{1}{\eta_1} Q_\cH(2\tau) ^\frac{\eta_1-1}{\eta_1}}{4},
		\end{align}
		we have $\sum_{i=1}^N\zeta_i\leq N Q_\cH(2\tau)$, with probability at least
\begin{align*}
1-c_1 Q_\cH(2\tau) ^{1-\frac{1}{\eta_1}} N^{\eta_1/(1+\eta_1)}e^{-c_2 Q_\cH(2\tau)  N^{\eta_1/(1+\eta_1)}}.
\end{align*}
\end{proof}
\begin{lemma}\label{lm:Xilowerbound}
Let $X_i, i=1,2,\cdots,N$ be a sample from a sequence for which condition (a) and (b) of Assumption~\ref{as:datagenprocess} is true. Then with probability at least with $$1-c_1Q_\cH(2\tau)^{1-\frac{1}{\eta_1}} N^{\eta_1/(1+\eta_1)}e^{-c_2Q_\cH(2\tau) N^{\eta_1/(1+\eta_1)}},$$ there is a subset $S\subset \{1,2,\cdots,N\}$ such that $\abs{S}\geq 3NQ_\cH(2\tau)/4$, and $\forall i\in S$, $\abs{X_i}\geq 2\tau\|X_i\|_{L_2}$.
\end{lemma}
\begin{proof}
Let $\zeta_i=\mathbbm{1}\left(\abs{X_i}\geq 2\tau\|X_i\|_{L_2}\right)$. Using condition (b) of Assumption~\ref{as:datagenprocess}, we have $\expec{\zeta_i}{}>Q_\cH(2\tau)$. Then using Lemma~\ref{lm:boundeddiff}, with probability at least $$1-c_1Q_\cH(2\tau)^{1-\frac{1}{\eta_1}} N^{\eta_1/(1+\eta_1)}e^{-c_2Q_\cH(2\tau) N^{\eta_1/(1+\eta_1)}},$$ 
we have
\begin{align*}
    3NQ_\cH(2\tau)/4\leq \sum_{i=1}^N\zeta_i\leq 5N\expec{\zeta_i}{}/4.
\end{align*}
\end{proof}
\begin{lemma}\label{lm:Xiuplowbound}
Let $X_i, i=1,2,\cdots,N$ be a sample from a sequence for which conditions (a) and (b) of Assumption~\ref{as:datagenprocess} is true. Then with probability at least with $$1-c_1Q_\cH(2\tau)^{1-\frac{1}{\eta_1}} N^{\eta_1/(1+\eta_1)}e^{-c_2Q_\cH(2\tau) N^{\eta_1/(1+\eta_1)}},$$ there is a subset $S\subset \{1,2,\cdots,N\}$ such that $\abs{S}\geq NQ_\cH(2\tau)/2$, and $\forall i\in S$,
\begin{align*}
    2\tau\|X_i\|_{L_2}\leq \abs{X_i}\leq \frac{2\|X_i\|_{L_2}}{\sqrt{Q_\cH(2\tau)}}.
\end{align*}
\end{lemma}
\begin{proof}
The proof is immediate from Lemma~\ref{lm:Xiupperbound}, and Lemma~\ref{lm:Xilowerbound}.
\end{proof}
\begin{lemma}\label{lm:thm4.3equiv}
Let $\cH$ be a class of function which is star-shaped around $0$ and satisfies condition (b) of Assumption~\ref{as:datagenprocess}. If $\zeta_1\sim2\tau Q_\cH(2\tau)^{3/2}$, $\zeta_2\sim 2\tau Q_\cH(2\tau)$, and $r=\|h\|_{L_2}>\omega_Q(\zeta_1,\zeta_2)$, there is a set $V_r\subset \cH\cap rS(L_2)$ such that there is an event $\mathcal{A}$ with probability at least $1-c_6Q_\cH(2\tau)^{1-\frac{1}{\eta_1}} N^{\eta_1/(1+\eta_1)}e^{-c_7 Q_\cH(2\tau)  N^{\eta_1/(1+\eta_1)}}$ we have:
\begin{enumerate}
    \item \begin{align}
        \abs{V_r}\leq \exp(c_2' Q_\cH(2\tau)  N^{\eta_1/(1+\eta_1)}/2),
    \end{align}
    where $c_2'\leq 1/1000$
    \item For every $v\in V_r$ there is a subset $S_v\subset\{1,2,\cdots,N\}$ such that $|S_v|\geq Q_\cH(2\tau)  N/2$, and for every $i\in S_v$,
    \begin{align}
         2\tau  r\leq \abs{v(X_i)}\leq \frac{c_3r}{\sqrt{ Q_\cH(2\tau) }}.
    \end{align}
    \item For every $h \in \cH\cap rS(L_2)$ there is some $v\in V_r$, and a subset $K_h\subset S_v$, containing at least $3/4$ of the coordinates of $S_v$, and for every $k \in K_h$,
    \begin{align}
        \tau\|h\|_{L_2}\leq \abs{h(X_k)}\leq c_{9}\left( 2\tau +\frac{1}{\sqrt{ Q_\cH(2\tau) }}\right)\|h\|_{L_2} ,\label{eq:hisoforderhL2}
    \end{align}
    and $h(X_k)$ and $v(X_k)$ have the same sign. 
\end{enumerate}
\end{lemma}
\begin{proof}
Let $r=\|h\|_{L_2}>\omega_Q(\zeta_1,\zeta_2)$. Let $V_r\subset H\cap rS(L_2)$ be a maximal $\rho$-separated set such that $$\abs{V_r}\leq \exp(c_2' Q_\cH(2\tau)  N^{\eta_1/(1+\eta_1)}/2)$$ where $c_2'=\min(c_2,1/500)$. Applying Lemma~\ref{lm:Xiuplowbound} on all the elements of $V_r$, using union bound we obtain that with probability at least $$1-c_1 Q_\cH(2\tau) ^{1-\frac{1}{\eta_1}} N^{\eta_1/(1+\eta_1)}e^{-c_2 Q_\cH(2\tau)  N^{\eta_1/(1+\eta_1)}/2},$$ for every $v\in V_r$ there is a subset $S_v$ such that $\abs{S_v}\geq N Q_\cH(2\tau) /2$ and for all $i\in S_v$, we have
\begin{align*}
     2\tau \|v(X_i)\|_{L_2}\leq \abs{v(X_i)}\leq \frac{c_3\|v(X_i)\|_{L_2}}{\sqrt{ Q_\cH(2\tau) }}. \numberthis\label{eq:vXiuplowbound}
\end{align*}
Since we have assumed $r> \omega_1(\zeta_1)$, from Sudakov's inequality we have,
\begin{align*}
    \rho\leq c_4\frac{\sqrt{2}\expec{\|G\|_{\cH\cap rS(L_2)}}{}}{\sqrt{c_2 Q_\cH(2\tau)  N^{\eta_1/(1+\eta_1)}}}\leq \frac{c_5\zeta_1r}{\sqrt{ Q_\cH(2\tau) }}, \numberthis\label{eq:rhobound}
\end{align*}
where $c_5=\sqrt{2}c_4/\sqrt{c_2}$. For all $h\in\cH\cap rS(L_2)$, let $h_v\in V_r$ so that $\|h-h_v\|_{L_2}\leq \rho$. Now let $\delta_h=\mathbbm{1}_{\left(\abs{h-h_v}> \tau r\right)}$ and put
\begin{align}
    \Delta_r=\left\lbrace\delta_h:h\in\cH\cap rS(L_2)\right\rbrace.
\end{align}
Define a function $\psi_1(t)=\max(\min(t/( \tau  r),1),0)$. Observe that $\delta_h(X)\leq \psi_1\left(\abs{h-h_v}(X)\right)$. Now we want to show that the number of points where $\abs{h-h_v}> \tau r$ is small. 
\begin{align*}
    &\expec{\sup_{\delta_h\in \Delta_r}\frac{1}{N}\sum_{i=1}^N\delta_h(X_i)}{}\\
    \leq&\expec{\sup_{h\in\cH\cap rS(L_2)}\frac{1}{N}\sum_{i=1}^N\psi_1(\abs{h-h_v}(X_i))}{}\\
    \leq & \expec{\sup_{h\in\cH\cap rS(L_2)}\frac{1}{N}\sum_{i=1}^N\left(\psi_1(\abs{h-h_v}(X_i))-\expec{\psi_1(\abs{h-h_v}(X)}{}\right)}{}+\expec{\sup_{h\in\cH\cap rS(L_2)}\expec{\psi_1(\abs{h-h_v}(X)}{}}{},
    \end{align*}
    where $X\sim\pi$. Consider the partition introduced in \eqref{eq:partition}. Then, 
    \begin{align*}
    &\expec{\sup_{\delta_h\in \Delta_r}\frac{1}{N}\sum_{i=1}^N\delta_h(X_i)}{}\\
    \leq & \expec{\sup_{h\in\cH\cap rS(L_2)}\frac{1}{N}\sum_{i=1}^\mu\sum_{j=1}^a\left(\psi_1(\abs{h-h_v}(X_{(a+b)(i-1)+j}))-\expec{\psi_1(\abs{h-h_v}(X)}{}\right)}{}\\
    &+\frac{2b\mu}{N}+\frac{1}{ \tau r}\expec{\sup_{h\in\cH\cap rS(L_2)}\expec{\abs{h-h_v}(X)}{}}{}\\
    \leq & \frac{\mu}{N}\sum_{j=1}^a\expec{\sup_{h\in\cH\cap rS(L_2)}\frac{1}{\mu}\sum_{i=1}^\mu\left(\psi_1(\abs{h-h_v}(X_{(a+b)(i-1)+j}))-\expec{\psi_1(\abs{h-h_v}(X)}{}\right)}{}\\
    &+\frac{2b\mu}{N}+\frac{\rho}{ \tau r}\\
    \leq & \frac{\mu}{N}\sum_{j=1}^a\expec{\sup_{h\in\cH\cap rS(L_2)}\frac{1}{\mu}\sum_{i=1}^\mu\left(\psi_1(\abs{h-h_v}(\tilde{X}_{(a+b)(i-1)+j}))-\expec{\psi_1(\abs{h-h_v}(X)}{}\right)}{}\\
    &+2(\mu-1)\beta(a+b)+\frac{2b\mu}{N}+\frac{\rho}{ \tau r}.
\end{align*}
Now using symmetrization, we get
\begin{align*}
    \expec{\sup_{\delta_h\in \Delta_r}\frac{1}{N}\sum_{i=1}^N\delta_h(X_i)}{}&\leq \frac{\mu}{N}\sum_{j=1}^a\expec{\sup_{h\in\cH\cap rS(L_2)}\frac{1}{\mu}\sum_{i=1}^\mu Q_\cH(2\tau) _i\psi_1(\abs{h-h_v}(\tilde{X}_{(a+b)(i-1)+j}))}{}\\
    &~~~+2(\mu-1)\beta(a+b)+\frac{2b\mu}{N}+\frac{\rho}{ \tau r}.
\end{align*}
Since $\psi_1(\abs{\cdot})$ is a $1/( \tau  r)$-Lipschitz continuous mapping, using properties of Rademacher complexity we have
\begin{align*}
    \expec{\sup_{h\in\cH\cap rS(L_2)}\frac{1}{\mu}\sum_{i=1}^\mu Q_\cH(2\tau) _i\psi_1(\abs{h-h_v}(\tilde{X}_{(a+b)(i-1)+j}))}{}\\
    \leq \frac{1}{ \tau r}\expec{\sup_{h\in\cH\cap rS(L_2)}\frac{1}{\mu}\sum_{i=1}^\mu Q_\cH(2\tau) _i(h-h_v)(\tilde{X}_{(a+b)(i-1)+j}))}{}.
\end{align*}
Since we assumed $r>\omega_2(\zeta_2)$, and using \eqref{eq:rhobound} we have,
\begin{align*}
    \expec{\sup_{\delta_h\in \Delta_r}\frac{1}{N}\sum_{i=1}^N\delta_h(X_i)}{}\leq \frac{a\zeta_2\mu}{N \tau }+2(\mu-1)\beta(a+b)+\frac{2b\mu}{N}+\frac{c_5\zeta_1}{ \tau \sqrt{ Q_\cH(2\tau) }}.
\end{align*}
Choosing
\begin{align}
			\zeta_1\sim 2\tau  Q_\cH(2\tau) ^{\frac{3}{2}} \quad
			\zeta_2\sim 2\tau  Q_\cH(2\tau) &\quad
			a\sim\frac{(4- Q_\cH(2\tau) )N^{1/(1+\eta_1)}}{c^\frac{1}{\eta_1}} \\ b\sim\frac{ Q_\cH(2\tau)  N^{1/(1+\eta_1)}}{c^\frac{1}{\eta_1}} ~~~&\text{and}~~ \mu\sim\frac{N^{\eta_1/(1+\eta_1)}c^\frac{1}{\eta_1}}{4},
		\end{align}
		we have,
\begin{align*}
    \expec{\sup_{\delta_h\in \Delta_r}\frac{1}{N}\sum_{i=1}^N\delta_h(X_i)}{}\leq \frac{ Q_\cH(2\tau) }{32}.
\end{align*}
Now we use Lemma~\ref{lm:boundeddiff}, with the following choice 
\begin{align*}
			t=\frac{N Q_\cH(2\tau) }{32} \quad a=\frac{(4-\frac{ Q_\cH(2\tau) }{16})N^{1/(1+\eta_1)}}{c^\frac{1}{\eta_1} Q_\cH(2\tau) ^{1-\frac{1}{\eta_1}}} \quad b=\frac{ Q_\cH(2\tau) ^\frac{1}{\eta_1} N^{1/(1+\eta_1)}}{32c^\frac{1}{\eta_1}} \quad \mu=\frac{N^{\eta_1/(1+\eta_1)}c^\frac{1}{\eta_1} Q_\cH(2\tau) ^\frac{\eta_1-1}{\eta_1}}{4}.
\end{align*}
With probability at least $1-c_6 Q_\cH(2\tau) ^{1-\frac{1}{\eta_1}} N^{\eta_1/(1+\eta_1)}e^{-c_7 Q_\cH(2\tau)  N^{\eta_1/(1+\eta_1)}}$ we have,
\begin{align*}
    \frac{1}{N}\sum_{i=1}^N\sup_{\delta_h\in \Delta_r}\delta_h(X_i)\leq \expec{\frac{1}{N}\sum_{i=1}^N\sup_{\delta_h\in \Delta_r}\delta_h(X_i)}{}+\frac{t}{N}\leq \frac{ Q_\cH(2\tau) }{16}.
\end{align*}
Then $\forall h\in\cH\cap rS(L_2)$,
\begin{align}
    \abs{\{i:\abs{h-h_v}(X_i)\leq \tau r\}}\geq \left(1-\frac{ Q_\cH(2\tau) }{16}\right)N.
\end{align}
Recall that $h_v\in V_r$, and $|S_{h_v}|\geq N Q_\cH(2\tau) /2$. Let
\begin{align*}
    K_h=\{k:\abs{h-h_v}(X_k)\leq  \tau r\}\cap S_{h_v}. \numberthis\label{eq:Khdef}
\end{align*}
Then $|K_h|\geq 3N Q_\cH(2\tau) /8\geq N Q_\cH(2\tau) /4$. Also, $\forall k\in K_h$,
\begin{align*}
    |h(X_k)|\geq |h_v(X_k)|-|(h-h_v)(X_k)|\geq  2\tau  r- \tau r= \tau r. \numberthis\label{eq:hxklowbound}
\end{align*}
This also implies that $h(X_k)$ and $h_v(X_k)$ have same signs. Similarly, using \eqref{eq:vXiuplowbound} we get
\begin{align*}
    |h(X_k)|\leq |h_v(X_k)|+|(h-h_v)(X_k)|\leq c_{9}( 2\tau +\frac{1}{\sqrt{ Q_\cH(2\tau) }})\|h\|_{L_2}. \numberthis\label{eq:hxkupbound}
\end{align*}
Combining \eqref{eq:hxklowbound} and \eqref{eq:hxkupbound} we have \eqref{eq:hisoforderhL2}. This also implies that $h(X_k)$ and $v(X_k)$ have the same sign.
\end{proof}
\begin{lemma}[{\cite[ Lemma 4.8]{mendelson2018learning}}]\label{lm:lemma4.8equiv}
Let $1\leq k\leq m/40$ and set $\mathscr{D}\subset\{-1,0,1\}^m$ of cardinality at most $\exp(k)$. For every $d=(d(i))_{i=1}^m\in\mathscr{D}$ put $S_d=\{i:d(i)\neq 0\}$ and assume that $|S_d|\geq 40k$. If $\{ \epsilon_i\}_{i=1}^m$ are independent, symmetric $\{-1,1\}$-valued random variables, then with probability at least $1-2\exp(-k)$,
\begin{align*}
    \inf_{d\in\mathscr{D}}\abs{\{i\in S_d:sgn(d(i))= \epsilon_i\}}\geq k/3.
\end{align*}
\end{lemma}
\begin{lemma}\label{lm:lemma4.9equiv}
Conditioned on the event $\mathcal{A}$ as mentioned in Lemma~\ref{lm:thm4.3equiv}, with probability at least $1-2\exp(-c_2 Q_\cH(2\tau)  N)$ we have: for every $h\in \cH_{f^*}\coloneqq \cF-f^*$ with $\|h\|_{L_2}\geq r$, there is a subset $\mathcal{S}_{1,h}\subset \{1,2,\cdots,N\}$ such that $\abs{\mathcal{S}_{1,h}}\geq  Q_\cH(2\tau)  N/24$. and for every $i\in \mathcal{S}_{1,h}$,
\begin{align*}
    \tau\|h\|_{L_2}\leq \abs{h(X_i)}\leq c_{9}\left( 2\tau +\frac{1}{\sqrt{ Q_\cH(2\tau) }}\right)\|h\|_{L_2}, \qquad sgn(h(X_i))= \epsilon_i, \numberthis\label{eq:lemma4.9equiv}
\end{align*}
where $\{\epsilon_i\}_{i=1}^N$ are independent, symmetric $\{-1,1\}$-valued random variables.
\end{lemma}
\begin{proof}
For a $h\in\cH$, let $\|h\|_{L_2}=r$ and let $h_v$ be as in Lemma~\ref{lm:thm4.3equiv}. Recall from \eqref{eq:hisoforderhL2}, that there is a subset $K_h\subset S_{h_v}$ containing at least $3/4$ of the coordinates of $S_{h_v}$ for which,
\begin{align*}
        \tau r\leq \abs{h(X_j)}\leq c_{9}\left( 2\tau +\frac{1}{\sqrt{ Q_\cH(2\tau) }}\right)r ,
\end{align*}
    and $h(X_j)$ and $h_v(X_j)$ have the same sign. Define
    \begin{align*}
        d_{h_v}=\{sgn(h_v(X_i))\mathbbm{1}_{S_{h_v}}(X_i)\}_{i=1}^N, \qquad \mathcal{D}=\{d_{h_v}:h_v\in V_r\}.
    \end{align*}
    Using Lemma~\ref{lm:lemma4.8equiv}, on the set $\mathcal{D}=\{d_{h_v}:d_{h_v}\in V_r\}$ for $k=N Q_\cH(2\tau) /1000$, and observing that every $d_{h_v}\in\mathcal{D}$, $\abs{\{i:d_{h_v}(i)\neq 0\}}\geq N Q_\cH(2\tau) /2\geq 40k$ (recall that $|S_{h_v}|\geq N Q_\cH(2\tau) /2$), we get with probability at least $1-2\exp(-c_2 Q_\cH(2\tau)  N)$, for every ${h_v}\in V_r$, $d_{h_v}(i)= \epsilon_i$ on at least $1/3$ of the coordinates of $S_{h_v}$.
    Then it follows that on at least $1/12$ of the coordinates of $S_{h_v}$, $h(X_j)= \epsilon_j$. Since $\cH_{f^*}$ is assumed to be star-shaped the same result holds when $\|h\|_{L_2}\geq r$.
\end{proof}
\begin{proposition}\label{prop:theorem4.7equiv}
With probability at least $1-c_{9} Q_\cH(2\tau) ^{1-\frac{1}{\eta_1}} N^{\eta_1/(1+\eta_1)}e^{-c_{10} Q_\cH(2\tau)  N^{\eta_1/(1+\eta_1)}}$, for every $f\in\cF$ which satisfies $\|f-f^*\|_{L_2}\geq 2\omega_Q$ we have
\begin{align}\label{eq:quadlower}
	\frac{1}{N}\sum_{i=1}^N \ell''(\widetilde{\xi}_i)(f-f^*)^2(X_i)\geq c_{16} Q_\cH(2\tau) \rho(0,t_0)\tau ^2\|f-f^*\|_{L_2}^2.
\end{align}
where $t_0=c_{11}( 2\tau +1/\sqrt{ Q_\cH(2\tau) })\left(\|\xi\|_{L_2}+\|f-f^*\|_{L_2}\right)$.
\end{proposition}
\begin{proof}[Proof of Proposition~\ref{prop:theorem4.7equiv}]
Recall the decomposition of $P_NL_f$ \eqref{eq:decomp}. For every $(X,Y)$ the midpoint $\tilde{\xi}$ belongs to the interval with end points $-\xi$ and $(f-f^*)(X)-\xi$ where $f\in\cF$. So,
\begin{align*}
    |\tilde{\xi}_i|\leq |\xi_i|+\abs{(f-f^*)(X_i)}. 
\end{align*}
Let $\|f-f^*\|_{L_2}>2\omega_Q$. Now from Lemma~\ref{lm:lemma4.9equiv}, with probability at least $$1-c_{9} Q_\cH(2\tau) ^{1-\frac{1}{\eta_1}} N^{\eta_1/(1+\eta_1)}e^{-c_{10} Q_\cH(2\tau)  N^{\eta_1/(1+\eta_1)}},$$ we have a subset $\mathcal{S}_{1,h}\subset \{1,2,\cdots,N\}$ such that $\abs{\mathcal{S}_{1,h}}\geq  Q_\cH(2\tau)  N/24$, and for every $i\in \mathcal{S}_{1,h}$,
\begin{align*}
    \abs{(f-f^*)(X_i)}\leq c_{9}( 2\tau +1/\sqrt{ Q_\cH(2\tau) })\|f-f^*\|_{L_2}.
\end{align*}
Using Markov's inequality,
\begin{align*}
    \mathbb{P}(|\xi_i|>10\|\xi\|_{L_2}/\sqrt{ Q_\cH(2\tau) })\leq \frac{ Q_\cH(2\tau). }{100}
\end{align*}
Now taking $U_i=\mathbbm{1}\left(|\xi_i|\leq \frac{c_{9}\|\xi\|_{L_2}}{\sqrt{ Q_\cH(2\tau) }}\right)$, and using Lemma~\ref{lm:boundeddiff}, and choosing parameters as in \eqref{eq:tabmucommonchoice} we get, with probability at least  $1-c_1 Q_\cH(2\tau) ^{1-\frac{1}{\eta_1}} N^{\eta_1/(1+\eta_1)}e^{-c_2 Q_\cH(2\tau)  N^{\eta_1/(1+\eta_1)}}$,
\begin{align*}
    \abs{\{i:|\xi_i|\leq \frac{c_{9}\|\xi\|_{L_2}}{\sqrt{ Q_\cH(2\tau) }}\}}\geq N(1- Q_\cH(2\tau) /50).
\end{align*}
This implies that with probability at least $1-c_{16} Q_\cH(2\tau) ^{1-\frac{1}{\eta_1}} N^{\eta_1/(1+\eta_1)}e^{-c_{17} Q_\cH(2\tau)  N^{\eta_1/(1+\eta_1)}}$ we have,
\begin{align*}
    |\tilde{\xi}_i|\leq c_{11}( 2\tau +1/\sqrt{ Q_\cH(2\tau) })\left(\|\xi\|_{L_2}+\|f-f^*\|_{L_2}\right).
\end{align*}
Set $t_0=c_{11}( 2\tau +1/\sqrt{ Q_\cH(2\tau) })\left(\|\xi\|_{L_2}+\|f-f^*\|_{L_2}\right)$. Using Lemma~\ref{lm:lemma4.9equiv}, with probability at least $1-c_{9} Q_\cH(2\tau) ^{1-\frac{1}{\eta_1}} N^{\eta_1/(1+\eta_1)}e^{-c_{10} Q_\cH(2\tau)  N^{\eta_1/(1+\eta_1)}}$,
\begin{align}
	\frac{1}{N}\sum_{i=1}^N \ell''(\widetilde{\xi}_i)(f-f^*)^2(X_i)\geq c_{16} Q_\cH(2\tau) \rho(0,t_0) \tau ^2\|f-f^*\|_{L_2}^2.
\end{align}
\end{proof}
Using Proposition~\ref{prop:theorem4.7equiv}, and proving the two-sided bounds for the second term on the RHS of \eqref{eq:decompwithexpec} in \eqref{eq:noiseconcalphacon} and \eqref{eq:noiseconcalphapolycon}, we have Proposition~\ref{pro:mainthm2appcon}.
\begin{proposition}\label{pro:mainthm2appcon}
	Consider ERM with loss functions that satisfy Assumption~\ref{as:lscgc}. For  $\tau_0<c_2Q_{\cF-\cF}(2\tau)\rho(0,t_0)\tau^2$, $t_0=\mathcal{O}(( 2\tau +1/\sqrt{Q_{\cH}(2\tau)})(\|\xi\|_{L_2}+\|f-f^*\|_{L_2}))$, setting $\mu = N^{\eta_1/(1+\eta_1)}$, for some constants $c,c' >0$, we have, for any $N \geq 4$, the following:
	\begin{enumerate}[leftmargin=0.2in]
	    \item Under conditions (a), (b), (c)-(i), and (d) of Assumption~\ref{as:datagenprocess}, for $\ 0<\iota<\frac{1}{4}$, 
	    \begin{align*}
\|\hat{f}-f^*\|_{L_2} \leq \max\left\{N^{-\frac{1}{4}+\iota}, 2\omega_Q(\cF-\cF,N, Q_\cH(2\tau) ^{3/2}, Q_\cH(2\tau) )\right\}, \numberthis\label{eq:errorboundalphaconapp}
	\end{align*}
	with probability at least (for $V$ is defined in~\eqref{eq:parameterv} and some positive $c_9, c_{10}, \widetilde{C}_3$)
	\begin{align*}
	&1-c_9Q_{\cH}(2\tau)^{1-\frac{1}{\eta_1}} N^{\eta_1/(1+\eta_1)}e^{-c_{10}Q_{\cH}(2\tau)^{1+\frac{1}{\eta_1}}N^{\frac{\eta_1}{1+\eta_1}}}
	-\widetilde{C}_3N\exp\left(-(N^{\frac{1}{2}+2\iota}\tau_0)^\eta/C_1\right).
\end{align*}
	    \item Under conditions (a), (b), and (c)-(ii) of Assumption~\ref{as:datagenprocess}, for $\ 0<\iota<(1-1/\eta_2)/4$,
	    \begin{align*}
		\|\hat{f}-f^*\|_{L_2} \leq \max\left\{N^{-\frac{(1-1/\eta_2)}{4}+\iota}, 2\omega_Q(\cF-\cF,N, Q_\cH(2\tau) ^{3/2}, Q_\cH(2\tau) )\right\}, \numberthis\label{eq:errorboundalphapolyconapp}
	\end{align*}
	with probability at least (for constants $c_9, c_{10}, \widetilde{C}_4>0$)
	\begin{align*}
		&1-c_9Q_{\cH}(2\tau)^{1-\frac{1}{\eta_1}} N^{\eta_1/(1+\eta_1)}e^{-c_{10}Q_{\cH}(2\tau)^{1+\frac{1}{\eta_1}}N^{\eta_1/(1+\eta_1)}}-\widetilde{C}_4\tau_0^{-\frac{2\eta_2}{1+\eta_2}}N^{-\frac{4\iota \eta_2}{1+\eta_2}}. \numberthis\label{eq:thm10probpolyconapp}
	\end{align*}
	\end{enumerate}
\end{proposition}
\begin{proof}[Proof of Proposition~\ref{pro:mainthm2appcon}]
We first prove \textbf{part 1}. We will denote the class $\cF-f^*$ by $\cH$. From Proposition~\ref{prop:theorem4.7equiv} it follows that for every $f\in\cF$ which satisfies $\|f-f^*\|_{L_2}\geq 2\omega_Q$ with probability at least $$\mathscr{P}_{1,c}=1-c_{9} Q_\cH(2\tau) ^{1-\frac{1}{\eta_1}} N^{\eta_1/(1+\eta_1)}e^{-c_{10} Q_\cH(2\tau) ^{1+\frac{1}{\eta_1}}N^{\eta_1/(1+\eta_1)}},$$ we have
\begin{align*}
    \frac{1}{N}\sum_{i=1}^N \ell''(\widetilde{\xi}_i)(f-f^*)^2(X_i)\geq c_{16} Q_\cH(2\tau) \rho(0,t_2) \tau ^2\|f-f^*\|_{L_2}^2.
\end{align*}
So, with probability at least $\mathscr{P}_{1,c}$, for every $f\in \cF$ that satisfies $\|f-f^*\|_{L_2}\geq 2\omega_Q$,
\begin{align*}
	P_N\mathcal{L}_f\geq\left( \frac{1}{16N}\sum_{i=1}^Nl'(\xi_i)(f-f^*)(X_i)-\expec{l'(\xi)(f-f^*)}{}\right)+c_{16} Q_\cH(2\tau) \rho(0,t_2) \tau ^2\|f-f^*\|_{L_2}^2. \numberthis\label{eq:PnLfintermedboundalphacon}
\end{align*}
	When $\|f-f^*\|_{L_2}\geq \mathscr{A}(N)> 2(N\tau_0)^{-1/2}$, we have $\log (N\tau_0\|f-f^*\|_{L_2}^2)\leq 2(N\tau_0\|f-f^*\|_{L_2}^2)^{(1-\eta)/2}/(1-\eta)$. Under Conditions~(1), (3), and (4) of Assumption~\ref{as:datagenprocess}, using Lemma~\ref{lm:rio}, we get
	\begin{align*}
		&\pro\left(\abs{\frac{1}{N}\sum_{i=1}^Nl'(\xi_i)(f-f^*)(X_i)-\expec{l'(\xi)(f-f^*)}{}}\geq \tau_0\|f-f^*\|_{L_2}^2\right)\\
		\leq& N\exp\left(-\frac{(N\tau_0\|f-f^*\|_{L_2}^2)^\eta}{C_1}\right)+\exp\left(-\frac{N^2\tau_0^2\|f-f^*\|_{L_2}^4}{C_2(1+NV)}\right)\\
		+&\exp\left(-\frac{N\tau_0^2\|f-f^*\|_{L_2}^4}{C_3}\exp\left(\frac{(N\tau_0\|f-f^*\|_{L_2}^2)^{\eta(1-\eta)}}{C_4(\log (N\tau_0\|f-f^*\|_{L_2}^2))^\eta}\right)\right)\\
		\leq& N\exp\left(-\frac{(N\tau_0\|f-f^*\|_{L_2}^2)^\eta}{C_1}\right)+\exp\left(-\frac{N^2\tau_0^2\|f-f^*\|_{L_2}^4}{C_2(1+NV)}\right)\\
		+&\exp\left(-\frac{N\tau_0^2\|f-f^*\|_{L_2}^4}{C_3}\exp\left(\frac{(1-\eta)^\eta(N\tau_0\|f-f^*\|_{L_2}^2)^\frac{{\eta(1-\eta)}}{2}}{C_42^\eta}\right)\right)\\
		\leq& N\exp\left(-\frac{(N\tau_0\mathscr{A}(N)^2)^\eta}{C_1}\right)+\exp\left(-\frac{N^2\tau_0^2\mathscr{A}(N)^4}{C_2(1+NV)}\right)\\
		+&\exp\left(-\frac{N\tau_0^2\mathscr{A}(N)^4}{C_3}\exp\left(\frac{(1-\eta)^\eta(N\tau_0\mathscr{A}(N)^2)^\frac{{\eta(1-\eta)}}{2}}{C_42^\eta}\right)\right)\equiv \mathscr{P}_{2,c}, \numberthis\label{eq:noiseconcalphacon}
	\end{align*}
	where $$V\leq \expec{\left(\ell'(\xi_1)(f-f^*)(X_1)\right)^2}{}+4\sum_{i\geq 0}\expec{B_i\left(\ell'(\xi_1)(f-f^*)(X_1)\right)^2}{},$$
	$\{B_i\}$ is some sequence such that $B_i\in[0,1]$ and  $\expec{B_i}{}\leq \beta(i)$, and $C_1,C_2,C_3$ are constants which depend on $c,\eta,\eta_1,\eta_2$. Observe that,
	\begin{align*}
		V\leq& \expec{\left( \ell'(\xi_1) (f-f^*)(X_1)\right)^2}{}+4\sum_{i\geq 0}\expec{B_i\left( \ell'(\xi_1) (f-f^*)(X_1)\right)^2}{}\\
		\leq & \expec{\left( \ell'(\xi_1) (f-f^*)(X_1)\right)^2}{}+4\sum_{i\geq 0}\sqrt{\expec{B_i^2}{}\expec{\left( \ell'(\xi_1) (f-f^*)(X_1)\right)^4}{}}\\
		\leq & \expec{\left( \ell'(\xi_1) (f-f^*)(X_1)\right)^2}{}+4\sqrt{\expec{\left( \ell'(\xi_1) (f-f^*)(X_1)\right)^4}{}}\sum_{i\geq 0}\sqrt{\expec{B_i}{}}\\
		\leq & \expec{\left( \ell'(\xi_1) (f-f^*)(X_1)\right)^2}{}+4\sqrt{\expec{\left( \ell'(\xi_1) (f-f^*)(X_1)\right)^4}{}}\sum_{i\geq 0}\sqrt{\beta(i)}\\
		\leq & \expec{\left( \ell'(\xi_1) (f-f^*)(X_1)\right)^2}{}+4\sqrt{\expec{\left( \ell'(\xi_1) (f-f^*)(X_1)\right)^4}{}}\sum_{i\geq 0}\exp(-ci^{\eta_1}/2)\\
		\leq & 2^\frac{2}{\eta_2}+C4^{1+\frac{2}{\eta_2}}.
\end{align*}
	Combining \eqref{eq:PnLfintermedboundalphacon}, and \eqref{eq:noiseconcalphacon}, with probability at least $\mathscr{P}_{1,c}-\mathscr{P}_{2,c}$, for every $f\in \cF$ that satisfies $\|f-f^*\|_{L_2}\geq \max(2\omega_Q,\mathscr{A}(N))$, we get 
	\begin{align*}
		P_NL_f\geq -2\tau_0\|f-f^*\|^2_{L_2}+c_{16} Q_\cH(2\tau) \rho(0,t_2) \tau ^2\|f-f^*\|_{L_2}^2.
	\end{align*}
	Choosing $\tau_0<c_{16} Q_\cH(2\tau) \rho(0,t_2) \tau ^2/4$, we have, $P_NL_f>0.$ But the empirical minimizer $\hat{f}$ satisfies $P_NL_{\hat{f}}\leq 0$. This implies, together with choosing $\mathscr{A}(N)=N^{-1/4+\iota}$, that with probability at least $\mathscr{P_c}=\mathscr{P}_{1,c}-\mathscr{P}_{2,c}$,
	\begin{align*}
		\|\hat{f}-f^*\|_{L_2}\leq \max(2\omega_Q(\cF-\cF,N,Q_\cH(2\tau)^\frac{3}{2},Q_\cH(2\tau)),\mathscr{A}(N)),
	\end{align*}
where 
\begin{align*}
	&\mathscr{P_c}=1-c_{9} Q_\cH(2\tau) ^{1-\frac{1}{\eta_1}} N^{\eta_1/(1+\eta_1)}e^{-c_{10} Q_\cH(2\tau) ^{1+\frac{1}{\eta_1}}N^{\eta_1/(1+\eta_1)}}\\
	-&N\exp\left(-\frac{(N^{\frac{1}{2}+2\iota}\tau_0)^\eta}{C_1}\right)-\exp\left(-\frac{N^{1+4\iota}\tau_0^2}{C_2(1+NV)}\right)
	-\exp\left(-\frac{N^{4\iota}\tau_0^2}{C_3}\exp\left((1-\eta)^\eta\frac{(N^{\frac{1}{2}+2\iota}\tau_0)^{\frac{\eta(1-\eta)}{2}}}{C_42^\eta}\right)\right).
\end{align*}

We now prove \textbf{part 2}. When $\|f-f^*\|_{L_2}\geq \mathscr{A}(N)$, using Lemma~\ref{lm:heavytailconc} we have
	\begin{align*}
		&\pro\left(\abs{\frac{1}{N}\sum_{i=1}^Nl'(\xi_i)(f-f^*)(X_i)-\expec{l'(\xi)(f-f^*)}{}}\geq \tau_0\|f-f^*\|_{L_2}^2\right)\\
		\leq& \frac{2^{\eta_2+3}}{(d_2\log N\tau_0\|f-f^*\|_{L_2}^2)^\frac{1-\eta_2}{\eta_1}}N(N\tau_0\|f-f^*\|_{L_2}^2)^{-(1+d_1(\eta_2-1))}+8N(N\tau_0\|f-f^*\|_{L_2}^2)^{-(1+d_2c')}\\
		&+2e^{-\frac{(N\tau_0\|f-f^*\|_{L_2}^2)^{2-2d_1}(d_2\log (N\tau_0\|f-f^*\|_{L_2}^2))^{1/\eta_1}}{9N}}\\
		\leq& \frac{2^{\eta_2+3}}{(d_2\log (N\tau_0\mathscr{A}(N)^2))^\frac{1-\eta_2}{\eta_1}}N(N\tau_0\mathscr{A}(N)^2)^{-(1+d_1(\eta_2-1))}+8N(N\tau_0\mathscr{A}(N)^2)^{-(1+d_2c')}\\
		&+2e^{-\frac{(N\tau_0\mathscr{A}(N)^2)^{2-2d_1}(d_2\log (N\tau_0\mathscr{A}(N)^2))^{1/\eta_1}}{9N}}\equiv \mathscr{P}_{2,c}. \numberthis\label{eq:noiseconcalphapolycon}
	\end{align*}
	We will choose $d_1$ suitably to allow $\mathscr{A}(N)$ to decrease with $N$ as fast as possible while ensuring $\lim_{N\to \infty}\mathscr{P}_{2,c}\to 0$. Combining \eqref{eq:PnLfintermedboundalpha}, and \eqref{eq:noiseconcalphapoly}, with probability at least $\mathscr{P}_{1,c}-\mathscr{P}_{2,c}$, for every $f\in \cF$ that satisfies $\|f-f^*\|_{L_2}\geq \max(2\omega_Q,\mathscr{A}(N))$, we get 
	\begin{align*}
		P_NL_f\geq -2\tau_0\|f-f^*\|^2_{L_2}+c_{16} Q_\cH(2\tau) \rho(0,t_2) \tau ^2\|f-f^*\|_{L_2}^2.
	\end{align*}
	Choosing $\tau_0<c_{16} Q_\cH(2\tau) \rho(0,t_2) \tau ^2/4$, we have, $P_NL_f>0.$ But the empirical minimizer $\hat{f}$ satisfies $P_NL_{\hat{f}}\leq 0$. This implies, together with choosing $\mathscr{A}(N)=N^{-(1-1/\eta_2)/4+\iota}$, $d_1=1/(1+\eta_2)$, $d_2=(\eta_2-1)/(\eta_2+1)$, and $\iota<(1-1/\eta_2)/4$, that with probability at least $\mathscr{P}_c=\mathscr{P}_{1,c}-\mathscr{P}_{2,c}$,
	\begin{align*}
		\|\hat{f}-f^*\|_{L_2}\leq \max(2\omega_Q(\cF-\cF,N,\zeta_1,\zeta_2),N^{-(1-1/\eta_2)/4+\iota}),
	\end{align*}
	where 
	\begin{align*}
		\mathscr{P}_c &=1-c_{9} Q_\cH(2\tau) ^{1-\frac{1}{\eta_1}} N^{\eta_1/(1+\eta_1)}e^{-c_{10} Q_\cH(2\tau) ^{1+\frac{1}{\eta_1}}N^{\eta_1/(1+\eta_1)}}-\frac{2^{\eta_2+3}\tau_0^{-\frac{2\eta_2}{1+\eta_2}}}{\left(\log \left(\tau_0N^{\frac{1}{2}+\frac{1}{2\eta_2}+2\iota}\right)/2\right)^\frac{1-\eta_2}{\eta_1}}N^{-\frac{4\iota \eta_2}{1+\eta_2}}\\&-8\tau_0^{-\frac{2\eta_2}{1+\eta_2}}N^{-\frac{4\iota \eta_2}{1+\eta_2}}
		-2e^{-\frac{\tau_0^{\frac{2\eta_2}{1+\eta_2}}\left(\log \left(\tau_0N^{\frac{1}{2}+\frac{1}{2\eta_2}+2\iota}\right)/2\right)^{1/\eta_1}}{9}N^{\frac{4\iota \eta_2}{1+\eta_2}}}.
	\end{align*}
\end{proof}
Note that Proposition~\ref{pro:mainthm2appcon} is exactly same as Theorem~\ref{th:mainthm2con} except for the fact one needs $\ell$ to be strongly convex in $[-t_0,t_0]$ instead of $[-t_2,t_2]$ where $t_2$ is of the order $\mathcal{O}(( 2\tau +1/\sqrt{Q_{\cH}(2\tau)})\|\xi\|_{L_2})$. So now we will show that empirical minimizer $\hat{f}\in\cF$ satisfies $\|\hat{f}-f^*\|_{L_2}\leq \max(\|\xi\|_{L_2},2\omega_Q)$ with high probability. 
One has the following result from \cite{mendelson2018learning}:
\begin{align}
    \{h-f^*:h\in\cF,\|h-f^*\|_{L_2}\geq R\}\subset\{\lambda(f-f^*):\lambda\geq 1, f\in\cF, \|f-f^*\|_{L_2}=R\}. \label{eq:lambdastar}
\end{align}
\begin{lemma}[{\cite[Lemma 5.6]{mendelson2018learning}}]\label{lm:5.6equiv}
When \eqref{eq:quadlower} is true, if $\|f-f^*\|_{L_2}\geq \max(\|\xi\|_{L_2},2\omega_Q)$, and $\lambda\geq 1$, then
    \begin{align}
        \frac{1}{N}\sum_{i=1}^N \ell''(\widetilde{\xi}_i)(\lambda(f-f^*))^2(X_i)\geq \floor{\lambda}c_{16} Q_\cH(2\tau) \rho(0,t_0) \tau ^2\max\left(\|\xi\|_{L_2}^2, 4\omega_Q^2\right).
    \end{align}
\end{lemma}
\begin{lemma}\label{lm:fhatfstarclose}
    With probability at least $1-\mathscr{P}_{2,c}$ with $\tau_0=c_{16} Q_\cH(2\tau) \rho(0,t_0) \tau ^2/4$, we have 
    \begin{align*}
        \|\hat{f}-f^*\|_{L_2}\leq \max(\|\xi\|_{L_2},2\omega_Q).
    \end{align*}
\end{lemma}
\begin{proof}
From \eqref{eq:noiseconcalphacon}, with probability at least $1-\mathscr{P}_{2,c}$ we have,
\begin{align*}
    \abs{\frac{1}{N}\sum_{i=1}^Nl'(\xi_i)(f-f^*)(X_i)-\expec{l'(\xi)(f-f^*)}{}}\leq \tau_0\|f-f^*\|_{L_2}^2.
\end{align*}
To make the dependency of $\mathscr{P}_{2,c}$ on $\tau_0$ explicit, we use the notation $\mathscr{P}_{2,c,\tau_0}$ to denote $\mathscr{P}_{2,c}$ for this proof.  
If $\|f-f^*\|_{L_2}\leq \max(\|\xi\|_{L_2},2\omega_Q)$, choosing $\tau_0=c_{16} Q_\cH(2\tau) \rho(0,t_0) \tau ^2/4$, with probability at least $1-\mathscr{P}_{2,c,c_{16} Q_\cH(2\tau) \rho(0,t_0) \tau ^2/4}$, for the same $\lambda\geq 1$ as in Lemma~\ref{lm:5.6equiv}, we have,
\begin{align*}
    \abs{\frac{1}{N}\sum_{i=1}^Nl'(\xi_i)(\lambda(f-f^*))(X_i)-\expec{l'(\xi)(\lambda(f-f^*))}{}}
    \leq \frac{c_{16}\lambda Q_\cH(2\tau) \rho(0,t_0) \tau ^2}{4}\max(\|\xi\|_{L_2},2\omega_Q)^2.
\end{align*}
If $\|f-f^*\|_{L_2}= \max(\|\xi\|_{L_2},2\omega_Q)$, and $\lambda\geq 1$, we also have
\begin{align*}
    &\frac{1}{N}\sum_{i=1}^N \ell''(\widetilde{\xi}_i)(\lambda(f-f^*))^2(X_i)-\abs{\frac{1}{N}\sum_{i=1}^Nl'(\xi_i)(\lambda(f-f^*))(X_i)-\expec{l'(\xi)(\lambda(f-f^*))}{}}\\
    \geq & \floor{\lambda}c_{16} Q_\cH(2\tau) \rho(0,t_0) \tau ^2\max\left(\|\xi\|_{L_2}^2, 4\omega_Q^2\right)-\frac{c_{16}\lambda Q_\cH(2\tau) \rho(0,t_0) \tau ^2}{4}\max(\|\xi\|_{L_2},2\omega_Q)^2\\ >&0.
\end{align*}
So by \eqref{eq:lambdastar}, and Lemma~\ref{lm:5.6equiv}, the empirical minimizer $\hat{f}$ satisfies,
\begin{align*}
        \|\hat{f}-f^*\|_{L_2}\leq \max(\|\xi\|_{L_2},2\omega_Q).
    \end{align*}
\end{proof}
\begin{proof}[Proof of Theorem~\ref{th:mainthm2con}]\label{p:thm2con}
Combining Lemma~\ref{lm:fhatfstarclose} with the two parts of Proposition~\ref{pro:mainthm2appcon} gives us Theorem~\ref{th:mainthm2con}.
\end{proof}
\begin{corollary}\label{cor:mainthm2con}
 For the convex ERM procedure, under Assumptions~\ref{as:datagenprocess}, with condition (b) replaced by Assumption~\ref{as:normequiv} with $p=8$, for some $0<\iota<\frac{1}{2}$ and $r, \mu$ and $\tau_0$ same as in Theorem~\ref{th:mainthm2con}, for sufficiently large $N$, we have
		\begin{align*}
			\|\hat{f}-f^*\|_{L_2}\leq \max\left(N^{-\frac{1}{2}+\iota},2\omega_Q(\cF-\cF,N, Q_\cH(2\tau) ^\frac{3}{2}, Q_\cH(2\tau) )\right) \numberthis\label{eq:errorboundalphaxcor}
		\end{align*}
with probability at least (for some constants $c_9,c_{10},\tilde{C}_2>0$)
 		\begin{align*}
				1-c_9Q_{\cH}(2\tau)^{1-\frac{1}{\eta_1}} N^{\eta_1/(1+\eta_1)}e^{-c_{10}Q_{\cH}(2\tau)^{1+\frac{1}{\eta_1}}N^{\eta_1/(1+\eta_1)}}-\tilde{C}_2N\exp\left(-(N^{2\iota}\tau_0)^\eta/M_1\right).
		\end{align*}
\end{corollary}
\begin{proof}[Proof of Corollary~\ref{cor:mainthm2con}]
The proof is same as Corollary~\ref{cor:mainthm2} and hence we omit it here. 
\end{proof}